\documentclass[11pt,a4paper,reqno]{amsart}

\usepackage{graphicx,pdfsync,amssymb,mathrsfs,amsfonts,amsbsy,color,esint,mathtools,tikz,mdframed}
\usetikzlibrary{positioning}

\usepackage[shortlabels]{enumitem}

\usepackage{hyperref}
\hypersetup{
colorlinks=true,
linkcolor=blue,
filecolor=blue,
urlcolor=blue,
citecolor=blue,
}

\usepackage[utf8]{inputenc}
\usepackage[T1]{fontenc}

\DeclareMathOperator{\Supp }{supp}

\DeclareMathOperator{\dist}{dist}

\newtheorem{theorem}{Theorem}[section]
\newtheorem{lemma}[theorem]{Lemma}
\newtheorem{proposition}[theorem]{Proposition}
\newtheorem{definition}[theorem]{Definition}
\newtheorem{corollary}[theorem]{Corollary}
\newtheorem{remark}[theorem]{Remark}

\setlist[enumerate]{itemsep=3pt}
\setlist[itemize]{itemsep=3pt}

\def \TT  {\mathbb{T}} 
\def \RR {\mathbb{R}}  
\def \NN {\mathbb{N}}  
\def \ZZ {\mathbb{Z}}  

\def \p {\partial}
\def \g {\gamma}
\def \k {\kappa}
\def \l {\lambda}

\def \ep {\epsilon}
\def \om {\omega}
\def \Om {\Omega}

\def \T {\mathbf{T}}
\def \N {\mathbf{N}}

\def \oT {\overline{\mathbf{T}}}
\def \oN {\overline{\mathbf{N}}}
\def \ok {\overline{\kappa}}
\def \oga {\overline{\gamma}}

\numberwithin{equation}{section}

\begin{document}

\title[Illposednes of $\alpha$-patches]{The $\alpha$-SQG patch problem is illposed in  $C^{2,\beta}$ and  $W^{2,p}$}

\author{Alexander Kiselev}

\address{Department of Mathematics, Duke University, Durham, NC 27708, USA}
\email{kiselev@math.duke.edu}

\author{Xiaoyutao Luo}

\address{Morningside Center of Mathematics, Academy of Mathematics and Systems Science, Chinese Academy of Sciences, Beijing  100190, China}

\email{xiaoyutao.luo@amss.ac.cn}


\date{\today}

\begin{abstract}
We consider the patch problem for the $\alpha$-SQG system with the values $\alpha=0$ and $\alpha= \frac{1}{2}$ being the 2D Euler and the SQG equations respectively. It is well-known that the Euler patches are globally wellposed in non-endpoint $C^{k,\beta}$ H\"older spaces,
as well as in $W^{2,p},$ $1<p<\infty$ spaces. In stark contrast to the Euler case, we prove that for $0<\alpha< \frac{1}{2}$, the $\alpha$-SQG patch problem is strongly illposed in \emph{every} $C^{2,\beta} $ H\"older space with $\beta<1$. Moreover, in a suitable range of regularity, the same strong illposedness holds for \emph{every} $W^{2,p}$ Sobolev space unless $p=2$.
\end{abstract}

\date{\today}

\maketitle

\section{Introduction}\label{sec:intro}
\subsection{The \texorpdfstring{$\alpha$}{a}-SQG equations}

The $\alpha$-SQG equations are a family of inviscid models in 2D incompressible fluid dynamics. These equations interpolate two of the most important 2D models of incompressible fluids: the 2D Euler equation and the surface quasi-geostrophic (SQG) equation.

The 2D Euler equation describes the motion of inviscid fluids. Its global regularity in sufficiently regular spaces has been known since  Wolibner \cite{Wolibner1933} and H\"older \cite{Holder1933}; the double exponential upper
bound on the growth of the derivatives of vorticity that essentially goes back to these works (see also \cite{MR0163529}) was recently shown to be sharp~\cite{MR3245016}. The surface quasi-geostrophic (SQG) equation models the rotating, stratified fluid in the atmosphere and ocean and is used in geophysical contexts, see for instance~\cite{pedlosky1992geophysical,Held,Majda}. Some aspects of the SQG equation bear resemblance to the 3D Euler equation~\cite{MR1304437} and the equation attracted much of attention over the years. Unlike the 2D Euler equation, even though the existence of global weak solutions for the SQG equation is known \cite{MR2716577,MR2357424}  (not expected to be unique~\cite{MR3987721}), the global regularity vs finite time blow-up question remains a major open problem; only infinite-in-time growth of derivatives has been shown in some examples~\cite{10.1215/00127094-2020-0064}.

Both the 2D Euler and the SQG equations are active scalar equations where a scalar $\om$ is advected by the velocity field $v$ generated by $\om.$ The $\alpha$-SQG equations, 
are a natural family of active scalars interpolating between the 2D Euler and SQG equations (see e.g. \cite{Held,CW,MR2141918}). This family has also been called modified or generalized SQG equations in the literature. As in the case of the SQG equation, the global regularity vs finite time blow-up question is still open for all $\alpha> 0.$

Mathematically, $\alpha$-SQG equations take form of a transport equation
\begin{equation}\label{eq:aSQG}
\p_t \omega + (v \cdot \nabla ) \omega =0.
\end{equation}
The relation between the scalar $\om :\RR^2 \times [0,T] \to \RR$ and the velocity $v :\RR^2 \times [0,T] \to \RR^2$ at each time $t$ is determined by the (modified) Biot-Savart law
\begin{equation}\label{eq:aSQG_BS}
v(x, t) =  \nabla^\perp(-\Delta)^{-1+\alpha}\om = K_\alpha  *  \om (x,t) : =  c_\alpha  \int_{\RR^2} \frac{(x-y)^\perp }{|x - y|^{2+2\alpha}} \omega(y, t) \, dy,
\end{equation}
where $x^\perp = (-x_2, x_1)$ for any $x\in \RR^2$. In this paper, for simplicity, we adjust the Biot-Savart law to have $c_\alpha=1;$ this is equivalent to simple time rescaling.
The key parameter $\alpha \geq 0$ in \eqref{eq:aSQG_BS} dictates the regularity of the velocity field compared to the scalar $\om$. In our setting, $\alpha =0$ corresponds to the 2D Euler equation while $\alpha = \frac{1}{2}$ to the SQG case.

\subsection{The \texorpdfstring{$\alpha$}{a}-SQG patch problem}
Another interesting set of solutions for the $\alpha$-SQG equation are patches. These are weak solutions where the scalar $\om$ takes the form of a characteristic function of evolving domain.

The question of wellposedness in the patch setting refers to the conservation of regularity of patch boundaries and the absence of patch collisions. More precisely, the single patch problem is well-posed if the initial smoothness of the patch boundary is preserved in time and no self-intersection of the patch
appears. In the 2D Euler case, the celebrated Yudovich theory~\cite{Yudth} guarantees the existence and uniqueness of the vortex patch. Due to the log-Lipschitz estimates on the flow map, different branches of the patch boundaries can not intersect. However, the Yudovich theory does not provide any information on the smoothness of the patch boundary, and it was a question of debate whether the smoothness of the vortex patch boundary can break down in finite time~\cite{doi:10.1063/1.857353,MR1050012}. Later, Chemin~\cite{MR1235440} proved that the vortex patches preserve their initial smoothness for all times. To date, several different proofs of global regularity for Euler patches are available:  Serfati \cite{MR1270072} and Bertozzi and Constantin \cite{MR1207667} (see also more recent \cite{verdera2021regularity,MR4509831}).  Let us also mention the works   \cite{MR1452164,MR1809342,MR4537304} concerning Euler patches with  cusps or corners.

However, the wellposedness theory for $\alpha$-patches is much less understood. This is due to the lack of the uniqueness theory at the level of bounded weak solutions and more severe loss of derivative in the Biot-Savart law.
We mainly discuss the general Cauchy problem of $\alpha$-patches in this paper, but also refer to ~\cite{MR3054601,MR3466846,MR3462104,MR3466846,MR3880855,MR4312192,2110.08615} for interesting constructions of special classes of $\alpha$-patches.

Compared to the 2D Euler case, the $\alpha$-patch problem is more difficult to set up as the velocity field is no longer log-Lipschitz. Perhaps the most fundamental question is
\begin{center}
\emph{What is an $\alpha$-patch?}
\end{center}
Historically, different notions of $\alpha$-patches have been proposed and studied by many people. For instance, the following is a list of possible definitions.
\begin{itemize}
\item \textbf{(Contour dynamics~\cite{MR2397460,MR2928091,MR3748149,MR4235799,2105.10982})} parametric curves $\g: \TT\times[0,T] \to \RR^2$ such that
\begin{equation}\label{eq:general_gamma}
\p_t \g (x,t) \cdot \g^{\perp}(x,t) = v(\gamma(x,t),t) \cdot \g^{\perp}(x,t);
\end{equation}

\item  \textbf{(Geometric~\cite{MR3666567,MR3549626,asqg_nosplash,2112.00191})}  time-dependent domains $\Omega_t$ whose boundary ``moves'' with velocity field $v$;

\item \textbf{(Eulerian~\cite{MR2142632})} characteristic functions $\chi_{\Omega_t}$ that solve \eqref{eq:aSQG} in the sense of distributions.
\end{itemize}

Starting from Gancedo's work \cite{MR2397460} on $H^3$ wellposedness, the contour dynamics approach has seen great success over the years. This formulation of the $\alpha$-SQG problem allows the use of powerful tools from functional and harmonic analysis. The principal idea builds on the observation rooted in \cite{MR2716577} that the tangential velocity does not change the shape of the patch, and involves suitably modifying the tangential velocity so that the contour equation can satisfy favorable energy estimates. Since~\cite{MR2397460}, the wellposedness of contour equations has been established for the range of $ 0< \alpha < 2$ including the velocity more singular than the SQG equation~\cite{MR2928091}. Related to our study here is the recent work~\cite{MR4235799}, where the authors successfully proved, among other things, the wellposedness of the $\alpha$-patch problem to $H^2$ when $0 <  \alpha <\frac{1}{2}$. See also \cite{2105.10982} for $H^{2+}$ when $\alpha = \frac{1}{2}$ the SQG case and \cite{2310.15963,2402.06364} for recent progress on long time wellposedness.

On the other hand, the geometric definition, introduced in \cite{MR3666567} by the second author, Yao, and  Zlato\v{s}, is intrinsic to the Biot-Savart law. This point of view allows the authors to prove $H^3$ uniqueness in a large class of patch solutions independent of contour equations. And it provides a framework of the singularity formulation for the $\alpha$-patch problem in a half-plane setting with small $\alpha< 1/24$ in \cite{MR3549626}, where boundaries of patches remain in touch with the boundary of the half-plane for all times before the singularity occurs. The same scenario was revisited in \cite{MR4235799}, showing the singularity formulation for $\alpha <1/6$.  This geometric definition was also used very recently in \cite{asqg_nosplash,2112.00191} to rule out various splash singularities under smoothness assumptions, extending the work~\cite{MR3181769}.
It remains open whether the original $\alpha$-patch problem in the whole space can develop a singularity.

Finally, the Eulerian notion for $\alpha$-patches seems less studied compared to the previous two approaches. Part of the reason is the lack of control on the patch boundary through the weak formulation. In the 2D Euler equation case, the celebrated Yudovich theory~\cite{MR0163529} provides the uniqueness framework in the Eulerian setting, and hence all notions of a patch solution automatically coincide. However for $\alpha>0$, the velocity is no longer log-Lipschitz, and it is not clear a priori that the Eulerian notion of $\alpha$-patches corresponds to the Lagrangian picture. For general transport-type problems, it is known that when the velocity is not Lipschitz (or log-Lipschitz), the usual link between the Lagrangian and Eulerian world may no longer be valid, for instance  the almost everywhere uniqueness of the trajectories \cite{1812.06817,2003.00539} does not imply the Eulerian uniqueness of weak solutions, \cite{2004.09538,2204.08950}. Nevertheless, in \cite{MR2142632}, for the SQG case, Rodrigo showed the wellposedness for a class of patch solutions $\om =\chi_{x_2 \leq \varphi(x_1 ,t)}$ with smooth $ \varphi$ periodic in $x_1$ by a Nash-Moser iteration procedure - this was in fact the first result on local wellposedness for the SQG patches.

Not surprisingly, the aforementioned difficulty of the Eulerian notion goes along with its great generality --- it is not hard to show~\cite[Proposition 3.2]{MR2397460} and \cite[Definition 1.2]{MR3666567} that at the level of $C^1$ regularity, the previous two notions of patch solutions are Eulerian ones. The Eulerian version is the definition of $\alpha$-patches we will use in this paper, see Definition \ref{def:asqg_patch} below. In fact,   we will show in Section \ref{sec:Eu_to_Lagran} that the Eulerian notion is equivalent to the geometric definition used in \cite{MR3666567,MR3549626} at $C^1$ regularity and thus provides a more solid foundation for future analysis of the  $\alpha$-patch problem.

\subsection{Motivation}
Part of our motivation for the analysis of the patch problem for $\alpha$-SQG equation came from the beautiful numerical simulations~\cite{MR2141918,MANCHO2015152,SD1,SD2} that provide numerical evidence for the existence of instabilities/finite time singularity formation in the SQG patch dynamics. In earlier numerical simulations~\cite{MR2141918,MANCHO2015152} two patches appear to touch each other and simultaneously develop corners at the point of touch. In the Dritschel-Scott scenarios~\cite{SD1,SD2}, the patch exhibits a cascade of rapid collapse and filamentation. They also found a simultaneous blowup of the maximum curvature and the inverse of filament width. While rigorously justifying such numerical simulation remains out of reach, our analysis here reveals some fine structures of the  $\alpha$-patch dynamics and could be of further interest in other settings.
Indeed, our key finding is the special structure of the curvature equation, which features a leading linear dispersion term and drives the ill-posedness results.
Although we discovered this special structure first in the case of the vortex patches \cite{asqg_nosplash}, in the $\alpha$-patch case this linear dispersion term leads to much stronger ill-posedness as will be described below.

We also recall that recent interest in ill-posedness results for equations of fluid mechanics started with the breakthrough of Bourgain and Li~\cite{MR3359050},
and it has been shown that the Euler equations
(in 2D and 3D) are ill-posed for vorticity in critical Sobolev spaces and integer Hölder spaces~(see also \cite{MR4065655,MR3320889,MR3625192,KIM2022109673}). Very recently, results of a similar flavor have been extended to $\alpha>0$. In \cite{2107.07739,2107.07463}, illposedness in the critical $H^2$ for the SQG equation was proved, see also \cite{2107.07463} for illposedness in supercritical Sobolev spaces $H^s$. The strategy of \cite{2107.07463} was extended to \cite{2207.14385} to obtain illposednesss in $C^{k,\beta}$ of $\alpha$-SQG in the more singular case $ \frac{1}{2}<\alpha <1$.

On the other hand, for the non-patch initial data in the SQG setting, there are results on local wellposedness for the H\"older classes \cite{MR2835862,2206.05861}. Moreover, the Euler patches are well-known to be globally wellposed in any H\"older spaces $C^{k,\beta}$
provided $k\in \NN$ and $ 0< \beta <1$~\cite{MR2686124}; wellposedness of vortex patches in $W^{2,p},$ $1<p<\infty$ was also proved more recently in \cite{asqg_nosplash}.
On the other hand, as mentioned above, known wellposedness results for the $\alpha$-path problem only apply in $C^\infty$ (Rodrigo \cite{MR2142632}) or $L^2$ based spaces (initiated by Gancedo's work \cite{MR2397460}).
This observation prompts the question:
\begin{center}
\emph{Is the $\alpha$-patches problem wellposed in $C^{k,\beta} $ H\"older spaces for $\alpha>0$?}
\emph{How about Sobolev spaces $W^{k,p}$ with $p\ne 2$?}
\end{center}

\subsection{Main results}
We now introduce the definition of a patch solution and get ready to state our main results. In this paper, a patch refers to a bounded domain (connected open set) $\Omega \subset \RR^2$  whose boundary is a simple closed curve. We also work with moving/time-dependent bounded domain $\Omega_t$ which is just a map $t\mapsto \Omega_t \subset \RR^2$ on some interval $ t\in [0,T]$.

Recall that a bounded domain $\Omega$ is said to be $C^1$  if its boundary $\p \Omega $ is locally a graph of a $C^1$ function. For a Banach space $X \subset C^1 $, we say $\Omega$ is an $X$ domain if the arc-length parameterization of $\p \Omega$ is in the regularity class $X$. In this paper, we mainly use the Sobolev spaces $X = W^{k,p}$ and H\"older spaces $X = C^{k,\beta}$, and we define
\begin{equation}\label{eq:intro_domain_norm}
\| \Om \|_{X} : = \| \g \|_{ X( [-L/2,L/2] )}  \quad \text{ for $X \in \{ W^{k,p}, C^{k,\beta} \}$ }
\end{equation}
where $\g: [-L/2,L/2] \to \RR^2$ is an arc-length parametrization of the boundary $\p \Om$ and $L$ is the length of $\p \Om $. Note \eqref{eq:intro_domain_norm} does not depend on the particular choice of arc-length parameterizations as they are equivalent up to translations. As in \cite{MR3666567,MR3549626}, the measured regularity is intrinsic to the $\alpha$-patch.

The definition stated below generalizes easily to the case of a solution with multiple patches, i.e. $\omega = \sum_{1 \leq k \leq N } \theta_k \chi_{\Omega_{k , t }}$ where $\theta_k$ are coupling constants and $ \Omega_{k , t }$, $1\leq k\leq N$ are time-dependent bounded domains with disjoint boundaries (allowing for nested $\Om_{k,t}$'s). We focus on the case of a single patch in this paper.
\begin{definition}\label{def:asqg_patch}
Let $\Omega_t \subset \RR^2$ be a time-dependent bounded domain  on $ t\in [0,T]$. We say $\Omega_t$ is an $\alpha$-patch on $[0,T]$ if the function $\omega = \chi_{\Omega_t}$ is a weak solution to \eqref{eq:aSQG}--\eqref{eq:aSQG_BS}, i.e. for any $ \varphi \in C^\infty_c(\RR^2 \times [ 0,T ) )$,
\begin{equation}\label{eq:def:asqg_patch}
 \int_{\Omega_0} \varphi(\cdot,0)  \, dx = -\int_{0}^T \int_{\RR^2} \chi_{\Omega_t} ( \p_t \varphi + v \cdot \nabla \varphi ) \,dx dt,
\end{equation}
where $v = K_\alpha * \omega$ is given by the Biot-Savart law \eqref{eq:aSQG_BS}.

In addition, for a regularity class $X \in \{ W^{k,p}, C^{k,\beta} \}$ we say $\Omega_t$ is a $X$  $\alpha$-patch if  $\sup_{t \in [0,T]} \| \Om_t \|_{X }<\infty$.
\end{definition}

 We now state the main theorems of the paper.

\begin{theorem}\label{thm:main_Holder}
Let $0<\alpha < \frac{1}{2}$.  The $\alpha$-patch problem is strongly illposed in $C^{2, \beta }$ for any  $0 \leq \beta < 1$ in the following sense.

For any $0<\alpha < \frac{1}{2}$ and   $0 \leq \beta < 1$, there exists a $C^{2, \beta }$ bounded domain $\Om_0$ such that  any $\alpha$-patch $\Omega_t$  with initial data $\Om_0$ (in the sense of Definition \ref{def:asqg_patch}) satisfies
\begin{equation}
\sup_{t \in [ 0, \delta ]  } \| \Om_t \|_{C^{2, \beta }} = \infty \quad \text{for any $\delta>0$.}
\end{equation}
\end{theorem}

\begin{remark}
\hfill

\begin{enumerate}

\item The classical Euler patch problem is globally wellposed in $C^{k,\beta}$ when $0< \beta <1$. Our results show such wellposedness fails as soon as $\alpha>0$.
We expect that illposedness holds for all integer $k \geq 2,$ but do not handle larger $k$ in this paper.

\item The omission of the endpoint cast $\beta =1$ is purely due to  technical issues, and one should be able to show the $C^{2,1}$ or $C^3$ illposedness of $\alpha$-patch problem by following the framework developed here.

\item The $C^{2,\beta}$ $\alpha$-patches are expected to be unique when $0<\alpha < \frac{1}{2}$, as noted in \cite{MR3748149}.

\end{enumerate}

\end{remark}

More comments will follow shortly after we state the next main theorem concerning the $\alpha$-patches in the Sobolev case $W^{2,p}$.

\begin{theorem}\label{thm:main_sobolev}
Let $0<\alpha < \frac{1}{2}$ and $p> \frac{1}{1-2\alpha }$. Unless $p =2$, the $\alpha$-patch problem is strongly illposed in $W^{2,p}$ in the following sense.

For any $p\neq 2$, if $ \alpha < \frac{1}{2} -\frac{1}{2p} $, then there exists a $W^{2,p}$ bounded domain $\Om_0$ such that  any $\alpha$-patch $\Omega_t$  with initial data $\Om_0$ (in the sense of Definition \ref{def:asqg_patch}) satisfies
\begin{equation}
\sup_{t \in [ 0, \delta ]  } \| \Om_t \|_{W^{2,p }} = \infty \quad \text{for any $\delta>0$.}
\end{equation}
\end{theorem}

\begin{remark}
\hfill

\begin{enumerate}
	\item We stated both theorems in a way that does not rely on any previous wellposedness results as the existence of $\alpha$-patches for initial data in $W^{2,p}$ for $p \ne 2$ is not known (and, as follows from the theorem, cannot be true).

	\item There is a drastic contrast between $\alpha=0 $ and $\alpha>0$--- The Euler vortex patch is globally wellposed in $W^{2,p}$ for $1 < p< \infty$~\cite{c2euler}.
	
	\item In~\cite{MR4235799}, Gancedo and Patel proved existence of $\alpha$-patches for $H^2$ initial data when $0< \alpha < \frac{1}{2}$. Our results here show that for such solutions, any higher Sobolev regularity $W^{2,p }$, $p>2$ may be lost instantaneously.
	
	\item The threshold $p> \frac{1}{1 - 2\alpha}$ is related to the existence of the Lagrangian flow of the patch boundary according to the Biot-Savart law. When $p < \frac{1}{1 - 2\alpha}$, the velocity field ceases to be Lipschitz on the patch boundary as we show in Lemma \ref{lemma:non-Lipschitz velocity} below.

\item In fact, our proof yields a slightly stronger statement: for any such $p,\beta$ and $\alpha$, there exists a small $\ep  >0$ such that  $\Om_t $  is not ${W^{2,p-\ep }} $ or respectively ${C^{2,\beta-\ep }}$.
In short, these $\alpha$-patches must exhibit a small instantaneous loss of integrability/derivative, but we do not discuss the optimal value of $\ep$ here.

\end{enumerate}
\end{remark}

\subsection{Outline of the proof}

We now explain the basic outline of the proof of main theorems. The proof is by contradiction as our main theorem includes regimes where even local existence was not known before (for instance $\alpha>0$ small and $p<2$ in $W^{2,p}$).  Since the proof of H\"older $C^{2,\beta}$ is essentially a special case in the proof Sobolev $W^{2,p}$ case, we just explain how we prove Theorem \ref{thm:main_sobolev}.

The most essential ingredient is the presence of a  leading linear dispersive term in the equation for the evolution of the curvature of the patch boundary. In this paper, since we are focusing on the Sobolev spaces $W^{2,p}$, we will consider the curvature evolution of the patch boundary as in~\cite{c2euler}. Compared to the Euler case~\cite{c2euler}, our proofs for the $\alpha$-patches are significantly more involved for many reasons that we explain below. To harvest the dispersion and obtain illposedness in this low regularity setting, there are roughly four steps in the proof.
\begin{enumerate}
\item Establish the wellposedness of the Lagrangian flow of the patch boundary from the Eulerian dynamics of the patch;

\item Derive the curvature evolution of the patch boundary and its leading order dynamics;

\item Prove illposedness for a model dispersive equation of the curvature evolution. This is done by a perturbative argument based on its distributional formulation;

\item Show illposedness for the curvature equation of $\alpha$-patches.

\end{enumerate}

The first step is to go from Eulerian to Lagrangian and show the existence and uniqueness of the flow of the $\alpha$-patch boundary. Given a $W^{2,p}$ $\alpha$-patch on $[0,T]$, we establish a suitable Lagrangian flow of its boundary. This step is necessary for tracking any fine regularity of the patch boundary. Due to the very weak definition we adopt here, we first show that the patch boundary $\p \Om_t$ flows with the velocity that the patch $\Omega_t$ generates in the sense of Definition \ref{def:boundary_flow_with}. This is a local statement on the patch boundary (in contrast to \cite[Definition 1.2]{MR3666567}) and allows us to take advantage of the transport nature of the equation. This works as long as the regularity is $C^1$ and the velocity field is continuous. As a byproduct, we obtain the equivalence of our Eulerian definition and the geometric definition~\cite[Definition 1.2]{MR3666567} of $\alpha$-patches.

Once we showed that the patch boundary $\p \Om_t$ flows with the velocity of the patch $\Omega_t$, we can use this local statement to recover the wellposedness of the particle trajectories when $ \Om_t$ is $W^{2,p}$ for $p> \frac{1}{1-2\alpha }$. Here the proof relies on intrinsic estimates of the velocity field. More precisely, even though the velocity field is not Lipschitz on $\RR^2$, we show that it is $C^{1,\sigma}$ for some $\sigma>0$ along the patch boundary. This higher H\"older regularity of the velocity on the boundary in turn allows us to uniquely define a $C^{1,\sigma}$ flow map of the patch boundary.

The flow that we obtained inevitably loses derivatives when compared with the given regularity $W^{2,p}$. This creates an issue for showing illposedness since a parameterization may lose regularity but that does not necessarily imply the same loss for the intrinsic regularity of the patch boundary. We get around this issue by considering the arc-length renormalization. More precisely, since the flow is $C^{1,\sigma}$, we can consider its reparameterization in the arc-length variable which must be $W^{2,p}$ or $C^{2, \beta}$ by assumptions. This reparametrization allows us to track the curvature of the patch boundary in the Lagrangian variable, even though the Lagrangian flow is only $C^{1,\sigma}$. In other words, the irregularity must only appear in the arc-length evolution by our standing assumptions.

The next step is then to analyze in detail the curvature flow of the $\alpha$-patches. A formal computation following~~\cite{c2euler} shows the leading order dynamics of the curvature equation given by
\begin{equation}\label{eq:intro_curvature}
	\p_t \k = c_\alpha \mathcal{L}_\alpha (\overline{ \k} ) + \text{error terms}
\end{equation}
where $\mathcal{L}_\alpha$ is a Fourier multiplier with symbol $ i \xi |\xi|^{2\alpha -1}  $ and $ \overline{\k}$ is a variant of the curvature $ \k$ that  satisfies $ \overline{\k} \to \k$ as $t\to 0$ in a suitable sense.

It is well-known~\cite{MR104973,MR182838} that the linear equation $ \p_t \k = \mathcal{L}_\alpha   { \k}  $ is illposed in non-$L^2$ spaces when $0< \alpha < \frac{1}{2}$. However, to extend such illposedness result to \eqref{eq:intro_curvature} we face two substantial difficulties:
\begin{enumerate}
\item First, the formal derivation of \eqref{eq:intro_curvature} does not make sense at such a low regularity $W^{2,p}$; the error terms are not well-defined functions.

\item Second, even the main term in \eqref{eq:intro_curvature} is a nonlinear perturbation of the linear equation $ \p_t \k = \mathcal{L}_\alpha   { \k}  $. Computing the new evolution group seems to be out of reach.
\end{enumerate}

To resolve these issues, we first derive a suitable distributional variant of \eqref{eq:intro_curvature} that remains meaningful in the $W^{2,p}$ setting whenever $0< \alpha<  \frac{1}{2}$ and $p> \frac{1}{1-2\alpha}$ and then use a perturbative argument at the level of distributional solutions to show illposedness.

More precisely, by using the Lagrange flow established before, we can derive a suitable weak formulation of \eqref{eq:intro_curvature} that makes sense when the velocity field is at least Lipschitz on the patch boundary, precisely when $p> \frac{1}{1-2\alpha}$.

In this process, a new error term of the order $ |\k - \overline{\k}|$ appears. To show the illposedness, we have to carefully design the initial data together with a sequence of test functions and time scales to show the instant blowup in the high frequencies of $\k$. Both the initial data and the test functions are related to Wainger's example~\cite{MR182838} for $e^{t \mathcal{L}_\alpha } $. The method we developed is reminiscent of the usual Duhamel's principle in the distributional setting. In fact, we can prove that under quite general conditions on $\k$, the new error term is under control and we are able to transfer all the existing illposedness results for the group $e^{t \mathcal{L}_\alpha } $ to the curvature equation \eqref{eq:intro_curvature}.

To finish the proof, we need to find initial data whose boundary has curvature matching the initial data in the illposedness part. In general, a $L^p$ or $C^\beta$ function $\k $ is not necessarily the curvature of a simple closed curve even if $\int \k  = 2\pi $ as the endpoints of the arc could mismatch. We rectify this situation by another perturbation argument exploiting the symmetry in the special initial data.

Finally, to prove the illposedness in the H\"older $C^{2,\beta}$ setting, we need to strengthen a few estimates as more regularity is required. Most of the argument is based on the framework developed in the Sobolev setting, and we refer to Section \ref{sec:proof_holder} for more details.

\subsection{Final remarks}
We close the introduction with a few concluding remarks.

The $\alpha$-patch problem can also be thought of as an interface dynamics. In other interface problems such as water waves, a linearization leads to an explicit linear operator whose linear equation has a dispersion relation that generally does not allow for non-$L^2$ wellposedness, see for instance~\cite[Section 3]{MR3409894}. In  \cite{MR3961297},  C\'ordoba, G\'omez-Serrano, and  Ionescu used this ansatz  for $\alpha$-SQG patches\footnote{This result has been extended to $0< \alpha\leq \frac{1}{2}$ in \cite{MR4379142,MR4762611}.} with $\alpha> \frac{1}{2}$ to obtain global patch solutions that are small perturbations of the half-plane stationary solution.  In their analysis, the linearization leads to the equation $\p_t - \mathcal{L}_\alpha$  using sophisticated para-differential calculus, which however requires significantly more regularity than our $W^{2,p} $ and $C^{2,\beta}$ setting.

One can notice that the $W^{2,p}$ illposedness result here does not cover the entire range $0 < \alpha < \frac{1}{2}$ of the $H^2$ wellposedness by Gancedo-Patel~\cite{MR4235799}. For those $H^2$ $\alpha$-patches, the Lagrangian flow of the patch boundary seems to be ill-defined and hence the techniques developed in this paper are not sufficient to analyze their fine dynamics.

Surprisingly, the dispersion in the curvature dynamics, which is the driving mechanism of the Sobolev and H\"older illposedness, is absent when $\alpha =0$ or $\alpha = \frac{1}{2 }$. Specifically, when $\alpha  = \frac{1}{2 }$, the dispersive relation leads to the evolution of a traveling wave on the physical side, and hence one could conjecture that the original SQG patch problem may be well-posed in non-$L^2$ spaces. The possible illposedness of  patches in the SQG case is particularly challenging as the velocity may no longer be continuous. Nevertheless, the framework we develop in this paper might be useful in other evolutionary free-boundary problems in a low-regularity setting.

\subsection{Plan of the paper}

The rest of the paper is organized as follows.
\begin{itemize}
\item We collect necessary notations, conventions, and basic technical tools used in the paper in Section \ref{sec:prelim}.

\item In Section \ref{sec:Eu_to_Lagran}, we prove that each $\alpha$-patch (according to Definition \ref{def:asqg_patch}) induces a unique Lagrangian flow of the patch boundary with certain regularity.

\item In Section \ref{sec:curvature_flow}, we derive and analyze the curvature equation of the $\alpha$-patch problem in the Lagrangian variable that we obtained from Section \ref{sec:Eu_to_Lagran}.

\item Section \ref{sec:disper_estimates} and Section \ref{sec:disper_eq_illposed} are devoted to the illposedness of a class of dispersive equations, inspired by the curvature dynamics of the $\alpha$-patch. We prove dispersive estimates and revisit Wainger's sharp counterexample for $e^{t \mathcal{L}_\alpha }$ in Section \ref{sec:disper_estimates}, then show related illposedness in Section \ref{sec:disper_eq_illposed}.

\item Finally, in Section \ref{sec:proof_illposed} and Section \ref{sec:proof_holder} we complete the proof of the main theorems by showing that the illposedness results in Section \ref{sec:disper_eq_illposed} carry over to the curvature equation of the $\alpha$-patch problem.

\end{itemize}

\subsection*{Acknowledgments}

AK acknowledges partial support of the NSF-DMS grants 2006372 and 2306726. XL has been partially supported by the NSF-DMS grant 1926686 while visiting IAS where a part of this work was done.

\section{Preliminaries}\label{sec:prelim}

\subsection{Notation}

Given $ x \in \RR^2$, $x^\perp$ denotes the counter-clockwise rotation by $\frac{\pi}{2}$ of $x$, i.e. $x^\perp = (-x_2 , x_1)$.

For any set $ E \subset \RR^2$ and $x \in \RR^2$, the distance of $x$ from  $E$ is denoted
$$
\dist(x , E): = \inf_{y\in E} \{ |x -y| \}.
$$
In this paper, the set $E$ will always be compact, so the infimum is always attained for some $y\in E$ (might not be unique).

Functions on the torus $\TT = \RR / 2\pi  \ZZ $ are identified with $2\pi$-periodic functions on $\RR$. Similarly, if $f$ is $L$-periodic for some $L>0$, we consider it a function on the rescaled torus $\TT = \RR / L\ZZ $.

For any $1 \leq p \leq \infty $ we write $|f|_{L^p(X)}$ to denote various Lebesgue norms on domains such as $X = \TT, \RR$ or a bounded domain $\Omega$. When there is no confusion, we simply write $|f|_{L^p}$.

Given any $L$-periodic function $f:\RR \to \RR^n  $,  we denote by $\mathcal{M}f: \RR \to \RR$ the ($L$-periodic) maximal function of $f$,
\begin{equation}\label{eq:def_maximal_function}
\mathcal{M}f(x ) = \sup_{0< \ep < 2 L }\frac{1}{2\ep}\int_{ x-\ep}^{x + \ep} |f ( y )| \, d y .
\end{equation}
The restriction $\ep < 2 L $ is non-essential and the boundedness of $\mathcal{M}$ for periodic function on $L^p$ for $1<p\leq \infty$ follows from the standard $\RR^d$ results \cite{MR0290095}.

Throughout the paper, $X \lesssim Y$ means $X \leq CY$ for some constant $C>0$ that may change from line to line. Similarly, $ X \gtrsim Y$ means $X \geq C Y$ and $X \sim Y$ means $ X \lesssim Y$, and $X \gtrsim Y$ at the same time.

We also use the following big-O and small-o notations: $X = O(Y)$ for a quantity $X$ such that $|X| \leq C Y$ for some absolute constant $C>0$ while $X = o(Y)$ means the ratio $  |X| |Y|^{-1}   \to 0$ as $Y \to 0$ or $Y \to \infty$.

\subsection{Classical H\"older spaces and Sobolev spaces}

For any integer $ k\in \NN\cup\{0\} $ and $0 \leq \beta \leq 1$, the classical H\"older spaces $C^{k,\beta}(\TT)$ consist  of $C^k$ continuous functions $f:\TT \to \RR$ such that the following norm is finite
\begin{equation}
|f|_{C^{k,\beta}}: = |f|_{C^k(\TT)} + \sup_{x\neq y} \frac{|f^{(k)}(x) - f^{(k)}(y)|}{|x-y|^\beta} < \infty.
\end{equation}

For any integer  $k \in \NN\cup\{0\} $ and $p \in [1,\infty] $, the Sobolev space $W^{k,p}(\TT)$ consists of functions whose weak derivatives of order up to $k$ belongs $L^p(\TT)$. The Sobolev norm is defined by
\begin{equation}
|f|_{W^{k,p}} : = \sum_{ 0\leq i \leq k} |\nabla^i f|_{L^p(\TT)}.
\end{equation}

Throughout the paper, we will also consider these spaces over various intervals, such as $[-L/2, L/2 ]$ for the arc-length parameterizations, and we keep the same notations $|\cdot |_{C^{k,\beta}}$ and $|\cdot |_{W^{k,p}}$ for their norms without spelling out the specific spatial domain.

\subsection{Planar curves and domains}

A (closed) curve is the image of a continuous map  $\gamma : \TT \to \RR^2$. We require $\g$ to be of class $C^1$, so the curve is locally the graph of a $C^1$ function. In this paper, all curves considered are simple and closed.

A parametrization of a curve is a $C^1$ periodic map  $\gamma : \RR \to \RR^2$ such that $ |\g'|>0$. If $ |\g'| = 1 $, we say it is an arc-length parametrization. We always assume the parameterization is counterclockwise oriented. In this paper, most of the non-arc-length parameterizations will be $2\pi$-periodic, i.e. on the standard torus $\TT$.

For a Banach space $X \subset C^1 $, we say a curve is of class $X$ if its arc-length parameterizations are of class $X$. Given a parametrization $\g$, we write $\T$ as the unit tangent vector and $\N = - \T^\perp$ the outer unit normal vector of $ \g$. In this paper, we mainly use the Sobolev and H\"older spaces, $X = W^{k,p}$ or $C^{k,\beta}$, to measure the regularity of curves. If $ \g$ is $W^{2,p}$, then its curvature $\k$  defined by
\begin{equation}\label{eq:def_kappa}
\begin{cases}
 \p_s \T   & = - \k \N\\
 \p_s \N    &=   \k \T
\end{cases}
\end{equation}
is a $L^p$ function on $ \g $. Note that counterclockwise orientation of $\g$ is assumed in \eqref{eq:def_kappa}.

Throughout the paper, we consider curves and domains that are time-dependent. We use the notion of a moving domain which is just a map $t \mapsto \Omega_t \subset \RR^2$ on some interval $[0,T] $ that assigns a bounded domain to every time $t\in [0,T]$. In this case, $\p \Omega_t$ refers to the boundary curve of $\Omega_t$ at each time $t$.

\subsection{Rough parametrizations of  regular domains}\label{subsec:rough_gamma_regular_Omega}

In what follows, we often consider parameterizations that are less regular than the domain itself. In this case, the metric induced by the parameterization is less regular. We explain here how the geometric quantities, particularly the curvature, are still well-defined even when the parameterization is only $C^{1,\sigma}$ for some $\sigma>0$.

\begin{lemma}\label{lemma:geo_quant_rough_gamma}
Let $\Omega$ be a $W^{2,p}$ bounded domain and let $\g \in C^{1,\sigma}(\TT)$ be a parametrization of $\p \Omega$.

Define tangent $\T(x)$, normal $\N(x)$, curvature $\k(x)$ in the parameterization $\g$ by the corresponding objects at point $\g(x) $ for any $x \in \TT$. Then the geometric quantities in the parameterization $\g$ satisfy the regularity
\begin{itemize}
    \item $\T, \N \in W^{1,p}(\TT)$.

    \item $\k \in L^p(\TT) $.
\end{itemize}

\end{lemma}
\begin{proof}
Let  $ g(x) = |\p_x {\g}(x)|  $ be the arc-length metric of $\g$, and consider the change of variable
$$
s: =\ell(x)= \int_0^x g(y) \, dy .
$$
Since $x \mapsto \g(x)$ is a   $C^{1,\sigma }$ parametrization, $ x \mapsto \ell(x) $ is  $C^{1,\sigma }$ bijection. Now we define a parameterization of $\p \Omega $ by $s \mapsto \overline{\g}(s)$ by setting $\overline{\g}(s) = \g(\ell^{-1}(s))$. It follows that $\overline{\g}$ is an arc-length parametrization of $\p \Omega $. Since $ \p \Omega$ is a $W^{2,p}$ domain, we have that $\overline{\g}$ is of class $W^{2,p}$(on some scaled torus). Let us denote $\overline{\T},\overline{\N}$, and $\overline{\k}$ the geometric quantities for this arc-length parametrization. Obviously by assumption $\overline{\T}, \overline{\N} \in W^{1,p} $ while $\overline{\k}  \in L^p$.

Due to the pullback formulas $\T (x) = \overline{\T}(\ell(x))$, $\N (x) = \overline{\N}(\ell(x))$, and $\k  (x) = \overline{\k}(\ell(x))$, we have that $ \T,\N$ and $\k$ have the same regularity as their arc-length counterparts (with norms depending also on $\g$).

\end{proof}

Let us remark that these quantities can also be obtained by arc-length differentiation. Due to the $C^{1,\sigma} $ diffeomorphism $x \mapsto \ell(x) $, for any function on $\TT$ we can define its arc-length differentiation with respect to the metric $g= |\p_x {\g}|$ as follows.

For any $f:\TT \to \RR$ we consider $f$ defined on $\p \Omega$ and $\p_s f$ its arc-length differentiation (with respect to $g= |\p_x {\g}|$),
\begin{align}\label{eq:def_arc_derivative}
\p_s f(x) := \lim_{y \to x} \frac{f(x) -f(y) }{ \int_y^x g(\zeta) \, d\zeta }.
\end{align}

The next simple lemma explains various equivalent ways to compute the arc-length derivative.

\begin{lemma}\label{lemma:arc_diff}
Let $ \g \in C^{1,\sigma}(\TT)$ for some $\sigma >0$ such that $g= |\p_x {\g}|>0$. Define the $C^{1,\sigma }$ diffeomorphism by $ x \mapsto \ell(x) = \int_0^x g(y) \, dy$.

Then for any $f \in C^1(\TT)$, its arc-length differentiation with respect to $\g $ defined by \eqref{eq:def_arc_derivative} satisfies
\begin{align*}
\p_s f(x) =  \frac{f'(x)}{|g(x)|} = \lim_{s' \to s} \frac{\overline{f}(s) -\overline{f}(s') }{ s-s'     },
\end{align*}
where $\overline{f} ( \cdot ): = f ( \ell^{-1}(\cdot) )$.
\end{lemma}
\begin{proof}
Note that \eqref{eq:def_arc_derivative} is just $  \lim_{s' \to s} \frac{\overline{f}(s) -\overline{f}(s') }{ s-s'     } $. It suffices to prove the first part of the identity.

Since $\g \in C^{1, \sigma } $ we have $\int_y^x g(\zeta) \, d\zeta  = g(x)(y-x ) + O(|y-x|^{1+\sigma})$. Then \eqref{eq:def_arc_derivative} follows.
\end{proof}

\subsection{Estimates for \texorpdfstring{$W^{2, p }$}{W2p} curves}\label{subsec:curves_estimates}
In this subsection, we collect some estimates for the $W^{2, p }$ curves.  These estimates will rely on $L^p$-boundedness of the maximal function.

When the parameterization is arc-length, we have  improved estimates for  some remainders. See \cite[Lemma 2.2]{c2euler} for the proof of Lemmas \ref{lemma:arc_length_estimates} and \ref{lemma:curve_estimates} below.

\begin{lemma}\label{lemma:arc_length_estimates}
Let  $\Om $ be a $W^{2,p}$ domain for some $1< p \leq \infty$ and $\beta = 1 -\frac{1}{p}$. Let $\g $ be an arc-length parameterization of $ \p\Om$. For  any $s,  s' \in \RR$, we have
\begin{subequations}
\begin{align}
\N(s) \cdot \T(s') & = O(|s-s'|^\beta) \label{NT1} \\
\T(s)  \cdot \T( s' ) &= 1+ O(  |s - s'  |^{2\beta }) \label{eq:arc_length_a}\\
(\g(s) - \g( s' )) \cdot \N(s) &= O(  |s - s'  |^{1+\beta })  \label{eq:arc_length_d} \\
(\T(s) - \T( s' )) \cdot \T(s) &= O(  |s - s'  |^{2\beta }) \label{eq:arc_length_e} \\
\frac{ |s -s'  | }{ |\g(s) - \g( s' )| }  &= 1+ O(  |s - s' |^{2\beta }) \label{eq:arc_length_f} \\
(\g(s) - \g( s' )) \cdot \T(s) &= (s - s'  )+ O(  |s - s'  |^{1+ 2 \beta }) \label{eq:arc_length_linear}
\end{align}
\end{subequations}
and the maximal estimates
\begin{subequations}
\begin{align}
\T( s' ) \cdot \N(s) & =  O(  \mathcal{M}\k ( s ) |s - s'  | ) \label{eq:arc_length_M_a} \\
\T( s' ) \cdot \T(s) & = 1+ O(  \mathcal{M}\k ( s ) |s - s'  |^{1+\beta } ) \label{eq:arc_length_M_b} \\
(\g(s) - \g( s' )) \cdot \N(s) &= O(  \mathcal{M}\k ( s ) |s - s'  |^{2})  \label{eq:arc_length_M_c} \\
  \T(s)  \cdot \left[\T(s) -\T( s' )  \right] &= O(  \mathcal{M}\k ( s ) |s - s'  |^{1+\beta}) \label{eq:arc_length_M_d} \\
\left[(\g(s) - \g( s' ) \right]\cdot \left[\T(s) -\T( s' )  \right] &= O(   \mathcal{M}\k (s ) |s - s'  |^{2+\beta}) \label{eq:arc_length_M_e}.
\end{align}

\end{subequations}

\end{lemma}

We also need a variant when the parameterization is not arc-length. As before, parameterizations are considered  as  periodic functions on $\RR$.

\begin{lemma}\label{lemma:curve_estimates}
Let  $\Om $ be a $W^{2,p}$ domain for some $1< p \leq \infty$. Let $\g $ be a  $C^{1,\sigma}(\TT)$ parameterization of $ \p\Om$ for some $ 0 < \sigma \leq   1 -\frac{1}{p} $.  For    any $ x, y \in \RR$, we have
\begin{subequations}
\begin{align}
\label{eq:curve_estimates_linear}
\T(x ) \cdot \T(y ) & = 1 + O( |x -y |^{2 (1 -\frac{1}{p}) })\\
(\g(x ) - \g(y )) \cdot \T(x ) &= g(x )(x  -y )+ O( |x -y |^{1+ \sigma })\\
\frac{ | x  - y  |}{|\g(x ) - \g(y )|}  &= [ g(x )]^{-1 } + O( |x  -y |^{\sigma }) \label{eq:curve_estimates_f}
\end{align}
\end{subequations}
where $g(x) =  |\p_x \g(x)| >0$ is the arc-length metric of $\g$.

\end{lemma}

\subsection{The Biot-Savart law for \texorpdfstring{$\alpha$}{a}-patches}

Let  $\Om$ be $C^{1}$ domain and $\g$ be a $C^1(\TT)$ parameterization of  $\p \Om$. The velocity field $v$ on $\p \Omega$ given by the Biot-Savart \eqref{eq:aSQG_BS} satisfies
\begin{equation}\label{eq:aux_BS_gamma}
v (x) =    -\frac{1}{2\alpha } \int_{\TT} \frac{\T(y)}{|\g(x) - \g(y)|^{2\alpha }} g (y)\, dy ,
\end{equation}
where $ g = |\p_x \g |$ is the arc-length of $\g$.

It is clear that when $0<\alpha < \frac{1}{2}$, above \eqref{eq:aux_BS_gamma} defines a continuous function on $\p \Om$. If the parameterization is arc-length, we instead use the labels $s,s'$ and
\begin{equation}\label{eq:aux_BS_arc-length_gamma}
v (s) =     -\frac{1}{2\alpha } \int_{\g} \frac{\T(s')}{|\g(s) - \g(s')|^{2\alpha }}  \, ds' .
\end{equation}

In this paper, the velocity $ v$ is always considered  as  a function on the patch boundary $\p \Om$, and hence we may consider its arc-length differentiation on $\p \Om $ defined by \eqref{eq:def_arc_derivative}. In Section \ref{subsec:Eu_to_Lagran_2} we will show that $\p_s v$ exists and is H\"older continuous when $\Om$ is $W^{2,p}$ for some $p > \frac{1}{1-2\alpha}$.

\section{From Eulerian to Lagrangian}\label{sec:Eu_to_Lagran}

In this section, we will show that every $\alpha$-patch in the sense of Definition \ref{def:asqg_patch} will generate a unique Lagrangian flow of its boundary parametrization. The main result of this section is Theorem \ref{thm:flow_existence}.

\subsection{Domains flowing with velocity}\label{subsec:Eu_to_Lagran_1}
We first introduce a definition to capture the local evolution of the boundary of a moving domain that is transported by a continuous vector field. Recall that $\p \Omega_t$ denotes the boundary of a moving domain $\Omega_t$ at time $t \in [0,T]$.

\begin{definition}\label{def:boundary_flow_with}
Let $\Omega_t $ be a   moving bounded domain on $t\in [0,T]$ and $v:\RR^2 \times [0,T] \to \RR^2$ be a continuous vector field. We say $\Omega_t$ flows with the velocity $v$ if
\begin{equation}\label{eq:Omega_move}
\lim_{h\to 0 }  \frac{\dist(   x + h  v  (x,t), \p \Omega_{t+h }) }{ h } =0
\end{equation}
for all $t\in [0,T]$ and $x\in  \p \Omega_t$ (with obvious one-sided modifications for $t= 0,T$).
\end{definition}

\begin{remark}
	\hfill
\begin{itemize}

\item We assume that $\Omega_t$ remains a bounded domain -- this is necessary as non-Lipschitz velocity fields can change the topological properties of $\Omega_t$, such as connectedness~\cite{MR3904158}.

\item The definition \eqref{eq:Omega_move} can also be compared with the ``compatibility'' condition in \cite[Section 4.4]{MR3904158} in terms of the normal component of the velocity. One might need a similar idea for the $SQG$ case $\alpha = \frac{1}{2}$ when the velocity is not bounded anymore.

\end{itemize}

\end{remark}

The next result shows that the patch flows with the velocity that it generates. We prove the lemma in a more general setting as it might be of independent interest in other free-boundary problems.
\begin{lemma}\label{lemma:Eulerian_to_geoflow}
Let $v:\RR^2 \times [0,T] \to \RR^2$ be a divergence-free vector field and $\Omega_t$ be an evolving  $C^1$ bounded domain.

If $x\mapsto v(x, t)$ is continuous on $\RR^2$  uniformly on $\RR^2 \times [0,T]$, then $ \theta = \chi_{\Omega_t } $ is a weak solution (as in \eqref{eq:def:asqg_patch}) of the transport equation
\begin{equation}\label{eq:Eulerian_to_geoflow}
\p_t \theta + v \cdot \nabla \theta =0
\end{equation}
if and only if $\Omega_t$ flows with $v$ on $[0,T]$.
\end{lemma}
\begin{proof}
\hfill

\noindent
\textbf{\underline{The ``if'' part:}}

Suppose $ \Omega_t$ flows with $v$, we first prove that
\begin{equation}\label{eq:aux_Eulerian_to_geoflow}
\frac{d}{dt } \int_{\Om_t} \varphi(x) \, dx = \int_{\Omega_t} v \cdot \nabla \varphi(x) \, dx  \quad \text{for all $\varphi \in C^\infty(\RR^2  )$}.
\end{equation}
The standard  density  and approximation arguments   would give  the desired weak formulation
\begin{equation}
\left.
\begin{aligned}
&  \int_{\Omega_T} \varphi(\cdot,T)  \, dx -\int_{\Omega_0} \varphi(\cdot,0)  \, dx \\
& = \int_{0}^T \int_{\RR^2} \chi_{\Omega_t} ( \p_t \varphi + v \cdot \nabla \varphi ) \,dx dt
\end{aligned} \right\}
\quad \text{for all $\varphi \in C^\infty(\RR^2 \times \RR )$.}
\end{equation}
Now we focus on \eqref{eq:aux_Eulerian_to_geoflow}.

Since $\p \Omega_t$ is $C^1$, by the divergence theorem, for each $t \in [0,T]$ we have
\begin{equation}\label{eq:aux_Eulerian_to_geoflow_2}
\int_{ \Omega_t }     v \cdot  \nabla \varphi (x ) \, dx   =    \int_{ \p \Omega_t } v \cdot \N(x,t)   \varphi (x ) \, dx
\end{equation}
where  $ \N(x,t)  $ denotes the outer normal vector of $ \p \Omega_t $.

From \eqref{eq:aux_Eulerian_to_geoflow} and \eqref{eq:aux_Eulerian_to_geoflow_2}, it suffices to show
\begin{equation}\label{eq:aux_Eulerian_to_geoflow_22}
 \int_{\Om_{t+h}} \varphi(x) \, dx  - \int_{\Om_t} \varphi(x) \, dx  = h \int_{ \p \Omega_t } v \cdot \N(x,t)   \varphi (x ) \, dx  + o(h).
\end{equation}
If $\p \Om_t$ admits a $C^{1}$ Lagrangian contour $\g:\TT\times [0,T] \to \RR^2$ according to $v$, namely $ \p_t \g =v(\g)$ with $\g \in C_t C^1$ being a parameterization of $\p \Om_t$, then \eqref{eq:aux_Eulerian_to_geoflow_22} would follow immediately. Due to the possible lack of such a $C^1$ parameterization,  we use an approximation argument as in \cite{MR3666567}.

Let $n$ be sufficiently large and partition $\p \Om_t $ into arcs of length $L/n$ where $L$ is the total length of $\p \Om_t$. Denote by $x_i \in \p \Om_t$, $1\leq i \leq n$, the partition points on $\p \Om_t $.

 We first take $n$ large so that the polygon with vertices $x_i$ approximates $\p \Om_{t }$ with error  $ h^2$ and also $|x_i -x_{i+1}| = \frac{L}{n} + o(n^{-1})$ by the $C^1$ regularity of $\p \Om_t$. Due to the continuity of $v$ and the flow condition \eqref{eq:Omega_move},  for any small $h>0$ we can take sufficiently large $n$ so that the polygon with vertices $x_i + h v(x_i, t )$ approximates $\p \Om_{t+h}$ with error  $o(h ).  $
So the left-hand side of \eqref{eq:aux_Eulerian_to_geoflow_22} can be approximated by summing the contribution from small polygons with vertices $x_i$, $x_{i+1}$, $x_i + h v(x_i, t ) $, and $x_{i+1}+ h v(x_{i+1}, t ) $ with error $o(h)  $. Notice by taking $n$ large, the contribution for each small polygon   is $ \frac{L}{n}  hv(x_i, t ) \cdot N(x_i ,t )    + ho(n^{-1})$.
Therefore,  in the limit $n \rightarrow \infty,$ the relationship \eqref{eq:aux_Eulerian_to_geoflow_22} follows.

\noindent
\textbf{\underline{The ``only if'' part:}}

Now given a weak solution $ \theta = \chi_{\Omega_t } $ to \eqref{eq:Eulerian_to_geoflow}, we prove \eqref{eq:Omega_move} by contradiction: suppose there exist $\ep_0 >0$, $t_0 \in [0,T]$, $x_0 \in \p \Om_t \subset \RR^2$, and a sequence $h_n \to 0$ such that
\begin{equation}\label{eq:aux_h_n_contradiction}
\dist(x_0 + h_n v(x_0, t_0 ) , \p \Omega_{t_0 + h_n}) \geq \ep_0 \big| h_n \big|.
\end{equation}
We assume $h_n >0$ and is decreasing by possibly passing to a subsequence and also assume $ t_0 \in (0,T)$ without loss of generality.

We construct a sequence of test functions $\varphi_{n} \in C^\infty(\RR^2 \times [0,T] )$, $n\in \NN$, that are compactly supported in space as follows. In short, we have
\begin{equation}
\varphi_{n} (x, t)= \eta_{ n }  \big( x   -  (t -t_0) v_0     \big) .
\end{equation}
where $v_0  = v(x_0, t_0 ) \in \RR^2 $ and $ \eta_{n } \in C^\infty_c(\RR^2)$ are cutoff functions that are smooth approximations (at scale $\sim {h_n} $) to the characteristic function $\chi_{ B_n }$ of the open balls  $B_n $ centered at $x_0$ with radius $r_n = \frac{1}{2} \ep_0 h_n$ such that
\begin{subequations}
\begin{align}
&   \eta_n =1  \quad \text{  if $|x - x_0 | \leq \frac{3}{4} r_n $} \label{eq:aux_chi_n_a}\\
&\Supp  \eta_n \subset B_n \label{eq:aux_chi_n_b} \\
&|\nabla \eta_n | \leq C  h_n^{ -1 }  \label{eq:aux_chi_n_c}
\end{align}
\end{subequations}

Note that   by construction
\begin{equation} \label{eq:aux_varphi_n_bound_E_n}
| \nabla \varphi_{n} |_{L^\infty_t L^1_x  } \leq C  {h_n}^{-1} |  {B}_{n} |
\end{equation}
for some $C>0$ independent of ${n}$ where $|  {B}_{n} |$ denotes the measure of the ball $B_n$.

We now consider the weak formulation of \eqref{eq:Eulerian_to_geoflow} with test functions $\varphi_{n}$. For all $n \in \NN$, we must have
\begin{equation}
\begin{aligned}\label{eq:aux_varphi_n_weak_formulation}
& \frac{1}{h_n}\int_{t_0}^{t_0 + h_n }\int_{\RR^2} \theta ( \p_t \varphi_{n} +  v \cdot \nabla  \varphi_{n} )  \,dx dt  \\
& = \frac{1}{h_n} \int_{\RR^2}  \left[  \theta(x,t_0 + h_n)\varphi_{n}(x,t_0 + h_n)  -\theta(x,t_0  ) \varphi_{n}(x,t_0)  \right]   \, dx  .
\end{aligned}
\end{equation}

We will use \eqref{eq:aux_varphi_n_weak_formulation} to derive a contradiction. Let us first consider the left-hand side of \eqref{eq:aux_varphi_n_weak_formulation}. Since each $\varphi_n$ is transported by the constant vector $v_0  = v(x_0, t_0 )$, splitting the velocity $v =v_0 + v- v_0$, we have
\begin{align}\label{eq:aux_varphi_n_split}
\frac{1}{h_n} \int_{t_0}^{t_0 + h_n }\int_{\RR^2 } \theta ( \p_t \varphi_{n} +  v \cdot \nabla  \varphi_{n} )  \,dx dt  & =   \frac{1}{h_n}\int_{t_0}^{t_0 + h_n }\int_{\RR^2 } \theta   (v - v_0) \cdot \nabla  \varphi_{n}    \,dx dt.
\end{align}
and hence by H\"older and \eqref{eq:aux_varphi_n_bound_E_n}
it follows that
\begin{equation}\label{eq:aux_varphi_n_split_1}
\frac{1}{h_n} \Big| \int_{t_0}^{t_0 + h_n }\int_{\RR} \theta   (v - v_0) \cdot \nabla  \varphi_{n}   \,dx dt \Big| \lesssim  | v - v_0|_{L^\infty( {B}_n \times [t_0 , t_0 + h_n ])}  {h_n }^{-1} | {B}_{n }  |
\end{equation}

Next, we consider the right-hand side of \eqref{eq:aux_varphi_n_weak_formulation}. By the assumption \eqref{eq:aux_h_n_contradiction} and the definition of the ball $B_n$, at time $t =t_0 + h_n$, we have (see Figure \ref{fig:domain_flows})
\begin{equation}\label{eq:aux_varphi_n_support}
\begin{cases}
 \text{either}\,& \Supp \varphi_{n} (\cdot, t_0 + h_n ) \subset \Omega_{t_0+ h_n} \\
\text{or } \quad &\Supp \varphi_{n}  (\cdot, t_0 + h_n ) \subset   \Omega_{t_0+ h_n}^c.
\end{cases}
\end{equation}
Furthermore, by \eqref{eq:aux_chi_n_a} and \eqref{eq:aux_chi_n_b}, we have
\begin{equation}\label{eq:aux_varphi_n_support_1}
\begin{cases}
 \text{either}\, &\int     \theta  \varphi_{n } (\cdot,t_0 + h_n)  \,dx = 0  \\
 \text{or} \,\, &\frac{9}{16}|B_n| \leq \int     \theta    \varphi_{n }  (\cdot,t_0 + h_n) \,dx \leq  |B_n| .
\end{cases}
\end{equation}
On the other hand, by the $C^1$ regularity of $ \p \Omega_t$, the intersection $\Supp \varphi_n \cap \p \Om_t $ is approaching a half-disk. Considering the set $  \{x\in \RR: \varphi_n(\cdot, t) =1 \} $ and its transition region $0 \leq \varphi_n(\cdot, t) \leq 1$, similar to the second case of \eqref{eq:aux_varphi_n_support_1} we also have
\begin{equation}\label{eq:aux_varphi_n_support_2}
\frac{8}{32}|B_n| \leq \int     \theta \varphi_{n } (x,t_0  )    \, dx \leq \frac{17}{32}|B_n| \quad \text{for all sufficiently large $n$}.
\end{equation}
From \eqref{eq:aux_varphi_n_support_1} and \eqref{eq:aux_varphi_n_support_2}, it follows that for $n$ sufficiently large,
\begin{equation}\label{eq:aux_varphi_n_split_2}
\Big| \int    \theta \varphi_{n } (x,t_0 + h_n)  -\theta \varphi_{n } (x,t_0  )   \, dx  \Big| \geq \frac{1}{32}  |B_{n}|  .
\end{equation}

Combining \eqref{eq:aux_varphi_n_split}, \eqref{eq:aux_varphi_n_split_1}, and \eqref{eq:aux_varphi_n_split_2},  we have for all sufficiently large $n$
\begin{equation}
 |B_n |   \leq C | v - v_0|_{L^\infty( {B}_{n} \times [t_0 , t_0 + h_n ])}    | {B}_{n}  |,
\end{equation}
This is a contradiction if $n $ is chosen large enough due to the uniform spatial continuity of $v$.

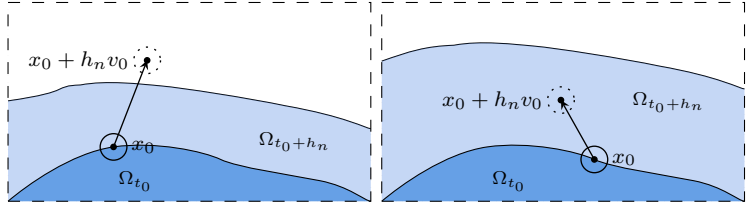
\begin{figure}
\tikzset{every picture/.style={line width=0.5pt}} 
\centering
\begin{tikzpicture}[x=0.75pt,y=0.75pt,yscale=-1,xscale=1]

\draw [line width=0.4pt]    (0,29.5+20) .. controls (13.91 +20 ,24.61 +20 ) and (21.55,22.02 +20 ) .. (36.05,21.52 +20 ) .. controls (49.05,19.02 +20 ) and (66.05,21.02 +20 ) .. (76.5,22 +20 ) .. controls (87.55,23.02 +20 ) and (119.55,28.02 +20 ) .. (138.5,32 +20 ) .. controls (168.36,37.34 +20 ) and (178.47,43.3 +20 ) .. (181,43.5 +20 );

\draw [dash pattern={on 0.0pt off 2.5pt}][fill={rgb, 255:red, 74; green, 130; blue, 226 }  ,fill opacity=0.32 ]   (0,29.5+20) .. controls (13.91 +20 ,24.61 +20 ) and (21.55,22.02 +20 ) .. (36.05,21.52 +20 ) .. controls (49.05,19.02 +20 ) and (66.05,21.02 +20 ) .. (76.5,22 +20 ) .. controls (87.55,23.02 +20 ) and (119.55,28.02 +20 ) .. (138.5,32 +20 ) .. controls (168.36,37.34 +20 ) and (178.47,43.3 +20 ) .. (181,43.5 +20 ) -- (181,100)--(0,100) --(0,0) ;

\draw [line width=0.4pt][fill={rgb, 255:red, 74; green, 144; blue, 226 }  ,fill opacity=0.77 ]   (0,100) .. controls (9.1,92.2) and (31.64,75.66) .. (52.64,72.66) .. controls (73.64,69.66) and (97.62,74.93) .. (106.08,78.97) .. controls (114.55,83.02) and (140.91,86.75) .. (155.64,89.66) .. controls (170.36,92.57) and (177.97,99.8) .. (181,100) ;

\draw [-stealth]   (52.64,72.66)  -- (69.41 ,49.11-20) ;

\filldraw
		(69.41 ,49.11-20) circle (1pt) node[left=1mm] {\scriptsize $x_0 + h_n v_0 $};

\filldraw
		(52.64,72.66) circle (1pt) node[right=1mm] {\scriptsize $x_0 $};

\draw  [dash pattern={on 0.8pt off 2.5pt}] (69.41 ,49.11-20) circle (5pt) ;
\draw   (52.64,72.66)  circle (5pt) ;

\draw  [dash pattern={on 4.5pt off 4.5pt}][line width=0.4]  (0,0) -- (181,0) -- (181,100) -- (0,100) -- cycle ;

\draw 	(120,70) node[anchor=west] {\tiny $\Omega_{t_0+h_n}$} ;
\draw 	(50,90) node[anchor=west] {\tiny $\Omega_{t_0}$} ;

\end{tikzpicture}
\begin{tikzpicture}[x=0.75pt,y=0.75pt,yscale=-1,xscale=1]

\draw [line width=0.4pt]    (0,29.5) .. controls (13.91,24.61) and (21.55,22.02) .. (36.05,21.52) .. controls (49.05,19.02) and (66.05,21.02) .. (76.5,22) .. controls (87.55,23.02) and (119.55,28.02) .. (138.5,32) .. controls (168.36,37.34) and (178.47,43.3) .. (181,43.5);

\draw [dash pattern={on 0.0pt off 2.5pt}][fill={rgb, 255:red, 74; green, 130; blue, 226 }  ,fill opacity=0.32 ]   (0,29.5) .. controls (13.91,24.61) and (21.55,22.02) .. (36.05,21.52) .. controls (49.05,19.02) and (66.05,21.02) .. (76.5,22) .. controls (87.55,23.02) and (119.55,28.02) .. (138.5,32) .. controls (168.36,37.34) and (178.47,43.3) .. (181,43.5) -- (181,100)--(0,100) --(0,0) ;

\draw [line width=0.4pt][fill={rgb, 255:red, 74; green, 144; blue, 226 }  ,fill opacity=0.77 ]   (0,100) .. controls (9.1,92.2) and (31.64,75.66) .. (52.64,72.66) .. controls (73.64,69.66) and (97.62,74.93) .. (106.08,78.97) .. controls (114.55,83.02) and (140.91,86.75) .. (155.64,89.66) .. controls (170.36,92.57) and (177.97,99.8) .. (181,100) ;

\draw [-stealth]   (106.08,78.97)  -- (89.41,49.11) ;

\filldraw
		(89.41,49.11) circle (1pt) node[left=1mm] {\scriptsize $x_0 + h_n v_0 $};

\filldraw
		(106.08,78.97) circle (1pt) node[right=1mm] {\scriptsize $x_0 $};

\draw  [dash pattern={on 0.8pt off 2.5pt}] (89.41,49.11) circle (5pt) ;
\draw   (106.08,78.97)  circle (5pt) ;

\draw  [dash pattern={on 4.5pt off 4.5pt}][line width=0.4]  (0,0) -- (181,0) -- (181,100) -- (0,100) -- cycle ;

\draw 	(120,50) node[anchor=west] {\tiny $\Omega_{t_0+h_n}$} ;
\draw 	(50,90) node[anchor=west] {\tiny $\Omega_{t_0}$} ;

\end{tikzpicture}
\caption{Two possible scenarios for $x + h_n v_0$ when $v_0$ points outward from $\Omega_t$.}
\label{fig:domain_flows}
\end{figure}

\end{proof}

\begin{remark}
The $C^1$ assumption on the domain $\Om_t$ can be relaxed to Lipschitz, which appears to be necessary for the only if part.
For example, if we consider a closing cusp on the boundary $\p \Om_t$ as in Figure \ref{fig:inward_cusp},   informal calculations suggest that the point on the tip of the cusp does not need to flow with the velocity field.
\end{remark}

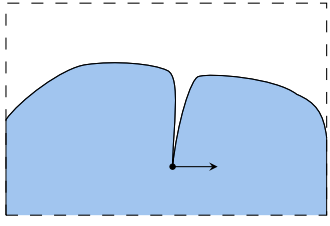
\begin{figure}

\begin{tikzpicture}[x=1pt,y=1pt,yscale=-1,xscale=1]

\draw [line width=0.4]    (0.0,43.98) .. controls (4.86,36.88) and (21.21,24.27) .. (30.01,23.2) .. controls (40.81,21.6) and (55.21,22.84) .. (60.35,25.41) .. controls (65.78,27.69) and (62.35,48.09) .. (62.35,61.69) .. controls (64.06,45.12) and (69.21,28.27) .. (72.06,27.69) .. controls (76.44,26.2) and (100.92,28.21) .. (108.81,34.4) .. controls (115.78,37.36) and (118.92,41.36) .. (120,50.84) ;

\draw [line width=0.0]  [fill={rgb, 255:red, 74; green, 144; blue, 226 }  ,fill opacity=0.52 ] (0,80)--(0.0,43.98) .. controls (4.86,36.88) and (21.21,24.27) .. (30.01,23.2) .. controls (40.81,21.6) and (55.21,22.84) .. (60.35,25.41) .. controls (65.78,27.69) and (62.35,48.09) .. (62.35,61.69) .. controls (64.06,45.12) and (69.21,28.27) .. (72.06,27.69) .. controls (76.44,26.2) and (100.92,28.21) .. (108.81,34.4) .. controls (115.78,37.36) and (118.92,41.36) .. (120,50.84) --(120,80);

\draw  [dash pattern={on 4.5pt off 4.5pt}][line width=0.4]  (0,0) -- (120,0) -- (120,80) -- (0,80) -- cycle ;
\draw  [-stealth]   (62.35,61.69)   -- (79.21,61.69) ;

\filldraw (62.35,61.69) circle (1pt) ;

\end{tikzpicture}
\caption{A moving domain with an inward cusp.}
\label{fig:inward_cusp}
\end{figure}

If $\Om_t$ is a $W^{2,p}$ $\alpha$-patch, to show $ \Om$ flows with the velocity that it generates, by Lemma \ref{lemma:Eulerian_to_geoflow} we just need to show the velocity $ v $ computed according to \eqref{eq:aSQG_BS} is uniformly continuous on $\RR^2  $. This is well-known and we refer to  \cite{MR3666567} for its proof.

\begin{lemma}[{\cite[Lemma 3.1]{MR3666567}}]
Let $ 0 < \alpha < \frac{1}{2}$. If $\om \in  L^1 \cap L^\infty(\RR^2)$, then $ v = K_\alpha *  \om (x,t)$ according to \eqref{eq:aSQG_BS} satisfies
\begin{equation}
| v|_{C^{0,1-2\alpha} (\RR^2)} \lesssim |\om|_{ L^1 \cap L^\infty}.
\end{equation}

\end{lemma}

\subsection{Intrinsic regularity of the velocity}\label{subsec:Eu_to_Lagran_2}
The next result concerns the intrinsic $C^{1,\ep}$ H\"older estimate of the Biot-Savart law.
\begin{lemma}\label{lemma:estimate_v_on_boundary}
Let $0< \alpha< \frac{1}{2}$ and $p>  \frac{1}{ 1 -2 \alpha } $. Suppose  $\Omega $ is a $W^{2,p}$ domain and $v$ is the velocity field of $\Omega $ according to \eqref{eq:aSQG_BS}.

For any arc-length parameterization of $\p \Omega $, the velocity field satisfies
\begin{equation}\label{eq:estimate_v_on_boundary}
		\p_s v (s)    =     P.V.  \int_{ \g    } \T (s' )   \frac{   (\g(s  )-\g( s' )) \cdot \T(s ) }{  |\g(s )-\g(s' )|^{2  +2\alpha }}  \,ds',
\end{equation}
and we have
\begin{equation}\label{eq:estimate_v_on_boundary2}
\big|   \p_s v   \big|_{C^{0,\sigma} } \leq C(\alpha,p, \Omega)
\end{equation}
with $\sigma= 1-\frac{1}{p} - 2\alpha >0.$
\end{lemma}
\begin{proof}
Let us consider an arc-length parameterization $\g  : \RR \to \p \Omega \subset \RR^2 $, where $\g$ is $L$-periodic with $L$   the length of $\p \Omega $. To show \eqref{eq:estimate_v_on_boundary}, we can start with \eqref{eq:aux_BS_arc-length_gamma} and use a distributional argument as in~\cite[Proposition 2.4]{c2euler}. We omit further details.

We now prove \eqref{eq:estimate_v_on_boundary2}. By a translation in \eqref{eq:estimate_v_on_boundary}, we have
\begin{align*}
\p_s v (s)  & =     P.V.  \int_{ \g    } \T (s'+s )   \frac{   (\g(s  )-\g( s'+s )) \cdot \T(s ) }{  |\g(s )-\g(s'+s )|^{2  +2\alpha }}  \,ds',
\end{align*}
where $P.V.$ denotes $ \lim_{\ep \to 0+} \int_{ \ep \leq |s'| \leq \frac{L}{2} }$.
It is not hard to show $|\p_s v| \leq C(\alpha, p, \Omega) $, so we focus on the H\"older continuity below.

For any $\delta >0$, we denote by $\Delta_\delta f(s) =f(s+\delta) -f(s)  $  (not to be confused with the Littlewood-Paley projections) and consider the split
\begin{align*}
 \Delta_\delta [ \p_s v   ]  & = \p_s v  (s  + \delta) - \p_s v    (s) = I_1 + I_2
\end{align*}
where
\begin{align*}
I_1 =  \lim_{\ep\to 0+}  & \left( \int_{ \ep< |  s' |    \leq 2\delta } \T ( s+  \delta+  s'  )    \frac{   (\g( s +\delta  )-\g(   s+  \delta+  s' )) \cdot \T( s +\delta ) }{  |\g( s+\delta )-\g(   s+  \delta+  s' )|^{2  +2\alpha }}    \,d   s' \right. \\
&\quad  \left. -\int_{ \ep< |  s' | \leq 2\delta   }\T (s'+s )   \frac{   (\g(s  )-\g( s'+s )) \cdot \T(s ) }{  |\g(s )-\g(s'+s )|^{2  +2\alpha }}  \,ds' \right)
\end{align*}
and
\begin{align*}
I_2 =    \int_{   |  s'| \geq 2\delta   }  \Delta_\delta \left[ \T (s'+s )   \frac{   (\g(s  )-\g( s'+s )) \cdot \T(s ) }{  |\g(s )-\g(s'+s )|^{2  +2\alpha }}    \right] \,ds'.
\end{align*}
We will show the estimates $|I_1|,|I_2|  \lesssim \delta^{ \sigma } =  \delta^{1  - \frac{1}{p}  -2\alpha  }    $.

\noindent
\textbf{\underline{Estimate of $I_1$:}}

By the estimates   \eqref{eq:arc_length_f}  and \eqref{eq:arc_length_linear} in Section \ref{subsec:curves_estimates}, for any $s,s \in \RR$
\begin{equation}\label{eq:aux_estimate_v_on_boundary}
\begin{aligned}
 \T(s' ) &= \T (s)   + O(| s  - s' |^{1 - \frac{1}{p}  } ) \\
 (\g(s      )-\g(  s' )) \cdot \T(s ) &= (s  - s' ) + O(|s  - s'|^{ 3- \frac{2}{p} } ) \\
\frac{|s  - s' |^{2  +2\alpha } }{|\g(s)-\g(  s')|^{2  +2\alpha }}   &  =1 + O(|s  - s'|^{  2(1 - \frac{1}{p}) }  ) .
\end{aligned}
\end{equation}

Here we note the constants in the big-O terms are independent of $s,s'$.
Using \eqref{eq:aux_estimate_v_on_boundary}, we have
\begin{align*}
I_1 = \lim_{\ep\to 0+} &\int_{ \ep< |  s' |    \leq 2\delta } \T ( s+  \delta  )    \frac{ -   s'   }{  |  s'  |^{2  +2\alpha }} \,ds'  + \int_{ \ep< |  s' |    \leq 2\delta }  O(| s'|^{- 2\alpha - \frac{1}{p} })  \,d   s' \\
&\quad -\int_{ \ep< |  s' | \leq 2\delta   }\T ( s )   \frac{   - s'   }{  |  s'  |^{2  +2\alpha }}\,ds'   + \int_{ \ep< |  s' |    \leq 2\delta }  O(| s'|^{ - 2\alpha - \frac{1}{p}})  \,ds'    ,
\end{align*}

The first terms on the right-hand side of the above two lines vanish by the oddness, and we have
\begin{align*}
|I_1| & \lesssim     \int_{  | s'| \leq 2\delta   }   |    s' |^{   - 2\alpha  - \frac{1}{p}  }      \,ds'.
\end{align*}
Since $ p> \frac{1}{ 1- 2\alpha }$, the exponent $ -2\alpha - \frac{1}{p} > -1 $, and integrating we obtain
\begin{align*}
I_1 &= O(\delta^{1  - \frac{1}{p}  -2\alpha  }   ).
\end{align*}

\noindent
\textbf{\underline{Estimate of $I_2$:}}

In this regime, we first consider the finite difference of the integrand
\begin{align}\label{eq:aux_proofholder}
 \Delta_\delta \left[ \T (s'+s )   \frac{   (\g(s  )-\g( s'+s )) \cdot \T(s ) }{  |\g(s )-\g(s'+s )|^{2  +2\alpha }}    \right] .
\end{align}
To reduce notations, for each fixed $|s'| \geq \ep$, let us denote by $K_{s'}(s)$, the function
$$
s \mapsto  K_{s'}(s):=  \T (s'+s )   \frac{   (\g(s  )-\g( s'+s )) \cdot \T(s ) }{  |\g(s )-\g(s'+s )|^{2  +2\alpha }} .
$$
From the regularity $\T, \N \in W^{1,p}$, it follows that $ K_{s'}$ is weakly differentiable with $L^p$ derivative:
\begin{equation}
\begin{aligned}
\p_s K_{s'} (s) = & - \k(s'+s )\N (s'+s )   \frac{   (\g(s  )-\g( s'+s )) \cdot \T(s ) }{  |\g(s )-\g(s'+s )|^{2  +2\alpha }}  \\
& +\T (s'+s )   \frac{   (\T(s  )-\T( s'+s )) \cdot \T(s ) }{  |\g(s )-\g(s'+s )|^{2  +2\alpha }} \\
& -\k(s) \T (s'+s )   \frac{   (\g(s  )-\g( s'+s )) \cdot \N(s ) }{  |\g(s )-\g(s'+s )|^{2  +2\alpha }}\\
& -(2+2\alpha )\T (s'+s )   \frac{   (\g(s  )-\g( s'+s )) \cdot \T(s )     }{  |\g(s )-\g(s'+s )|^{4  +2\alpha }}\\
& \qquad \times (\g(s  )-\g( s'+s )) \cdot (\T(s )  -\T( s'+s  ) )
\end{aligned}
\end{equation}
Using standard arguments and $W^{2,p}$ estimates of $\g$ in Section \ref{subsec:curves_estimates}, it is straightforward to show that $K_{s'}' $ satisfies the point-wise estimate
\begin{equation}\label{eq:aux_K'_mean_value_0}
|\p_s K_{s'} (s)| \lesssim      |\k(s'+s )| |s'|^{-1-2\alpha }  + \mathcal{M}\k(s'+ s)  |s'|^{-\frac{1}{p}-2\alpha }  +|\k( s )| |s'|^{-\frac{1}{p}-2\alpha } ,
\end{equation}
where as before $\mathcal{M}\k$ is the periodic maximal function of $\k$. So by the fundamental theorem of calculus for $W^{1,p} $ functions and the bound \eqref{eq:aux_K'_mean_value_0}, we have
\begin{equation}\label{eq:aux_K'_mean_value}
\begin{aligned}
|K_{s'}(s+\delta) - K_{s'}(s ) |  & \leq \int_{s}^{s+\delta}  |\p_s K_{s'} (\tau )|\, d\tau \\
& \lesssim  |s'|^{-1-2 \alpha } \delta^{ 1-\frac{1}{p}}
\end{aligned}
\end{equation}
where in the last inequality we have used $\mathcal{M}\k, \k \in L^p$ and $p>1$.

Now we apply the estimate \eqref{eq:aux_K'_mean_value} for $K_{s'}(s)$ to $I_2$, obtaining
\begin{align*}
| I_2| \leq     \int_{ \frac{L}{2}  \geq |  s'| \geq 2\delta   }  \Delta_\delta \left[  K_{s'}(s) \right] \,ds' \lesssim \int_{    \frac{L}{2}  \geq |  s'| \geq 2\delta   }   |s'|^{-1-2 \alpha } \delta^{ 1-\frac{1}{p}} \lesssim \delta^{ 1- \frac{1}{p}  -2\alpha } .
\end{align*}
where in the last step we have assumed $ \delta < \frac{L}{4}$.

\end{proof}

\subsection{Counterexamples with non-Lipschitz velocity}\label{subsec:Eu_to_Lagran_3}

We present a simple example of non-Lipschitz velocity for domains below the regularity threshold $W^{2 , \frac{1}{1  -2\alpha }}$.

\begin{lemma}\label{lemma:non-Lipschitz velocity}
Let $ 0 < \alpha < \frac{1}{2}$. For any $1\leq p < \frac{1}{1- 2 \alpha}$, there exist a $W^{2, p}$ bounded domain $\Omega$ such that the velocity $v$ given by \eqref{eq:aSQG_BS} is not Lipschitz on $\p \Omega$.
\end{lemma}
\begin{proof}

Let $ \ep>0 $ be sufficiently small such that
\begin{equation}\label{eq:aux_non-Lipschitz velocity_1}
\ep < \min\{  2 \alpha,  \frac{1}{p} - (1 -2\alpha ) \}.
\end{equation}
From \eqref{eq:aux_non-Lipschitz velocity_1}, we have that the function $  |s|^{- (1 -2 \alpha)  -\ep } \in L^p([ -1,1]) $. It follows that there are $W^{2 , p }$ bounded domains $\Omega$ whose curvature in the arc-length parameterization satisfies
$$
 \k(s) = - |s|^{- (1 -2 \alpha)  - \ep }
$$ near $s =0$, and is otherwise smooth.

Fix such a $W^{2 , p }$ domain $\Omega$ and let $\delta>0$ be such that
\begin{equation}\label{eq:aux_non-Lipschitz velocity_2}
  \N (s' ) \cdot \N ( s ) \in [\frac{1}{2} , 1]  \quad \text{ for any $s, s'  \in [-\delta , \delta ]$}.
\end{equation}

We need to show the velocity $
v (s) =  -\frac{1}{2\alpha } \int_{\g} \frac{\T(s')}{|\g(s) - \g(s')|^{2\alpha }}  \, ds'  $ is not a Lipschitz function.

Since $v$ is smooth away from $s=0$ by the construction of $\Om$, we will show $\p_s v \cdot \N(s) \to \infty $ as $s \to 0$. Similarly to Lemma \ref{lemma:estimate_v_on_boundary}, we introduce the split
\begin{align*}
|\p_s v (s)  \cdot \N(s)| & =     \Big| \int_{ 0 \leq |s'| \leq \delta   } \T (s' +s ) \cdot \N ( s )  \frac{   (\g(s  )-\g( s' +s )) \cdot \T(s ) }{  |\g(s )-\g(s' +s )|^{2  +2\alpha }}  \,ds' \\
&\qquad  +  \int_{ \delta  \leq |s'| \leq \frac{L}{2}  } \T (s' +s ) \cdot \N ( s )  \frac{   (\g(s  )-\g( s' +s )) \cdot \T(s ) }{  |\g(s )-\g(s' +s )|^{2  +2\alpha }}  \,ds'  \Big|\\
& \geq \Big| \int_{ 0 \leq |s'| \leq \delta    } \T (s' +s ) \cdot \N ( s )  \frac{   (\g(s  )-\g( s' +s )) \cdot \T(s ) }{  |\g(s )-\g(s' +s )|^{2  +2\alpha }}  \,ds'\Big| - M_\delta
\end{align*}
where $M_\delta < \infty$ is independent of the label $s$. Now it suffices to show
\begin{equation}\label{eq:aux_non-Lipschitz velocity_3}
\int_{ 0 \leq |s'| \leq \delta    } \T (s' +s ) \cdot \N ( s )  \frac{   (\g(s  )-\g( s' +s )) \cdot \T(s ) }{  |\g(s )-\g(s' +s )|^{2  +2\alpha }}  \,ds'
\end{equation}
goes to $\infty$  as $s \to 0$.

Since $\k(s) < 0$, by \eqref{eq:aux_non-Lipschitz velocity_2} and the fundamental theorem of calculus, we have that $\T (s' +s ) \cdot \N ( s )  $ always has the sign opposite to $s' $. The same holds for $(\g(s  )-\g( s' +s )) \cdot \T(s )$.

It follows that the integrand in \eqref{eq:aux_non-Lipschitz velocity_3} is always non-negative:
\begin{align*}
 \T (s' +s ) \cdot \N ( s )    (\g(s  )-\g( s' +s )) \cdot \T(s ) \geq 0 \quad \text{for any $ s'\in [ -\delta ,\delta]$} .
\end{align*}
By this non-negativity, we can reduce the domain of the integral to $ [0,s]$:
\begin{align}\label{eq:aux_non-Lipschitz velocity_33}
|\p_s v (s)  \cdot \N(s)| & \geq  \int_{  0\leq  s'  \leq s } \T (s' +s ) \cdot \N ( s )  \frac{   (\g(s  )-\g( s' +s )) \cdot \T(s ) }{  |\g(s )-\g(s' +s )|^{2  +2\alpha }}  \,ds'   - M,
\end{align}
where $M$ is a constant that does not depend on $s$ (and may only depend on $\delta$, $\epsilon$, $\alpha$ and choice of $\Omega$).

For $s' \in [0,s]$, we compute that
\begin{equation}\label{eq:aux_non-Lipschitz velocity_4}
 \T (s' +s ) \cdot \N ( s ) \geq   - \frac{1}{2} \int_0^{  s'  }  \k(s +\tau )   \, d\tau  \geq - C (|s+ s'|^{   2\alpha  -\ep }  -|s |^{   2\alpha  -\ep })
\end{equation}
if $\delta$ is sufficiently small.
By the mean value theorem and $s' \in [0,s]$ we have
\begin{equation}\label{eq:aux_non-Lipschitz velocity_5}
 \T (s' +s ) \cdot \N ( s ) \geq   - C (|s+ s'|^{   2\alpha  -\ep }  -|s |^{   2\alpha  -\ep })  \geq -C  s^{-1 + 2\alpha   -\ep} s'  .
\end{equation}
Similar computation also shows
\begin{equation}\label{eq:aux_non-Lipschitz velocity_6}
(\g(s  )-\g( s' +s )) \cdot \T(s ) \geq -C s' .
\end{equation}
Putting together \eqref{eq:aux_non-Lipschitz velocity_33}, \eqref{eq:aux_non-Lipschitz velocity_5}, and  \eqref{eq:aux_non-Lipschitz velocity_6} we have
\begin{align*}
|\p_s v (s)  \cdot \N(s)| & \geq  C s^{-1 + 2\alpha -\ep }\int_{  0\leq  s'  \leq s }  |s'|^{-2\alpha}  \,ds'   - M \\
&\geq C s^{-\ep} -M \to \infty \quad \text{as $s \to 0$.}
\end{align*}

\end{proof}

\subsection{A Lipschitz estimate of the normal velocity}\label{subsec:Eu_to_Lagran_4}

We present a lemma regarding the normal velocity, similar to the ones in \cite{asqg_nosplash} and \cite{MR3181769}; hence we only present a sketch of the argument. Note that here the bounded domain $\Omega$ does not need to be a patch solution and $p$ can be equal to $\frac{1}{ 1 -2 \alpha } $.

\begin{lemma}\label{lemma:normal_lipschitz}
Let $0< \alpha< \frac{1}{2}$ and $p \geq  \frac{1}{ 1 -2 \alpha } $. For any   $W^{2,p}$  bounded domain $\Omega$, there exist constants $\ep >0$ and $C(\ep)>0$ such that the following holds.

Let  $v:\RR^2 \to \RR^2$ be the velocity field generated by $\Om$ according to \eqref{eq:aSQG_BS}. Then  for any $x\in \RR^2 $ with $\dist(x, \p \Om) \leq \ep$, there is a unique $y=y(x) \in \p \Omega   $   such that  $\dist (x, \p\Omega )  = |x-y|$ and there holds
\begin{equation}\label{eq:lemma_normal_lipschitz}
|(v(x) - v(y))\cdot n(y)| \leq
 C(\ep) |x-y|  ,
\end{equation}
where $n(y) $ denote the unit normal vector of $\p \Om$ at $y$.
\end{lemma}
\begin{proof}
Given a $W^{2,p}$ bounded domain $\Omega$, we first take $0<\ep<1$ sufficiently small such that
\begin{enumerate}
	\item if $ \dist(x, \p \Omega) \leq \ep$ there is only one unique $y\in \p\Omega$ with $\dist(x, \p \Omega) = |y-x|$, and
	 \item  $\ep < L/100$ where $L$ is the length of $\p \Om$, and
	\item the tangent vector on $\p \Omega$ only changes by at most $1/100 $ if arc-length changes by $\ep$.
\end{enumerate}

Denote $d = |x-y|  \leq \ep $ the distance from $x$ to $\p \Omega$. Without loss of generality we assume $y =(0,0)$, $x = (0, d)$ and $n(y) = (0,1)$.

We first arc-length parameterize $\p \Om$ by $\g : [-\frac{L}{2} , \frac{L}{2}] \to \RR$ so that $\g (0) = y =(0,0)$ and $ \N(0) = n(y) = (0,1) $. Using the Biot-Savart law \eqref{eq:aux_BS_arc-length_gamma}, we have
\begin{align*}
(v(x) - v(y))\cdot n(y) = \int_{- \frac{L}{2} }^{\frac{L}{2} } \T(s) \cdot \N(0) \Big[ \frac{1 }{| x- \g (s )|^{2\alpha }} - \frac{1}{| \g (s )|^{2\alpha }} \Big] d s.
\end{align*}
Following \cite{asqg_nosplash}, we introduce a split in the domain of integration: $| s|\leq d$, $ d \leq| s |\leq \ep $, and $\ep \leq |s |\leq \frac{L}{2}$.

In the region $|s |\leq d$, the tangent of the boundary barely changes its value and we have the bounds
\begin{align*}
|x -\g(s)|^{-1}  & \lesssim |s|^{-1} ,\\
|\g(s)|^{-1} & \lesssim |s|^{-1} .
\end{align*}
These together with the estimate $| \T(s) \cdot \N(0) | \leq |s|^{1 -\frac{1}{p} } $ from the $W^{2,p}$ regularity of $\Omega $ allows to estimate the integrand by its absolute value,
$$
 \int_{   |s| \leq d } \T(s) \cdot \N(0) \Big[ \frac{1 }{| x- \g (s )|^{2\alpha }} - \frac{1 }{|   \g (s )|^{2\alpha }} \Big] d s \leq C \int_{s \leq d} |s|^{1-\frac{1}{p} -2\alpha } \leq \int_{s \leq d} |s|^{0} \leq  d.
$$

Next, in the region $ d \leq| s |\leq \ep $, we use the mean value theorem to obtain
$$
\Big| \frac{1 }{| x- \g (s )|^{2\alpha }} - \frac{1}{| \g (s )|^{2\alpha }} \Big| \leq C_\alpha   d\frac{ z_2  }{|   z  -\g (s )|^{2\alpha +2 }}  \leq C  \frac{d^2 }{|   \g (s )|^{2\alpha +2 }}
$$
where $z =(0,z_2) $ is a point between $x =  (0 , d ) $ and the origin. Since $  |s| \leq \ep$, we have $|   z  -\g (s )|^{-1} \lesssim |\g_1(s)|^{-1} \lesssim |\g(s)|^{-1}$ in this regime. Hence
\begin{align*}
  \int_{ d \leq |s| \leq \ep } \T(s) \cdot \N(0) \Big[ \frac{1 }{| x- \g (s )|^{2\alpha }} - \frac{1 }{|   \g (s )|^{2\alpha }} \Big] \, d s & \leq C d^2 \int _{ d \leq  |s| \leq \ep} s^{ - \frac{1}{p}  -1-2\alpha }\\
&  \leq C d  .
\end{align*}

Finally, the last integral
$$
  \int_{     \ep \leq  |s| \leq \frac{L}{2} } \T(s) \cdot \N(0) \Big[ \frac{1 }{| x- \g (s )|^{2\alpha }} - \frac{1 }{|   \g (s )|^{2\alpha }} \Big] d s
$$
can be estimated easily since the kernel is away from the singularity.

\end{proof}

\subsection{Existence and uniqueness of the trajectories}\label{subsec:Eu_to_Lagran_5}

We now use the estimates in Lemma \ref{lemma:estimate_v_on_boundary} and Lemma \ref{lemma:normal_lipschitz} to show the existence and uniqueness of trajectories for the patch boundary.

\begin{lemma}\label{lemma:unique_path_on_boundary}
Let $0< \alpha< \frac{1}{2}$ and $\Omega_t$ be a $W^{2,p}$ $\alpha$-patch solution for some $p \geq \frac{1}{1- 2\alpha}  $ with the initial data $\Omega_0$. Let $v :\RR^2 \times [0,T]$ be the velocity of $\Omega_t$. Then for every $X_0 \in \p \Omega_0 \subset \RR^2 $, the Cauchy problem for $X_t :[0,T] \to \RR^2$
\begin{equation}\label{eq:trajectory}
    \begin{cases}
     \frac{d X_t }{dt} = v(X_t ,t) & \\
     X_t |_{t=0} = X_0
    \end{cases}
\end{equation}
has a unique solution $X_t \in C^{1}$, and in addition $X_t \in \p \Omega_t$ for all $t \in [0,T]$.
\end{lemma}
\begin{proof}
We first prove the  existence of trajectories  and then that the solutions remain on the patch boundary and finally their uniqueness.

\textbf{\underline{Step 1: Existence of $C^1$ trajectories.}}

Observe that the vector field $v$ is uniformly continuous and bounded on $\RR^2\times [0,T]$. We can apply the  Peano existence theorem and this gives a solution $X_t :[0,T] \to \RR^2 $ to the ODE \eqref{eq:trajectory} on $ t\in [0,T]$.

\textbf{\underline{Step 2: No exit from $\p \Omega_t $.}}

We now prove that when $\Omega_t$ is a $W^{2,p}$ patch solution with $p \geq \frac{1}{1- 2\alpha }$, all $C^1$ solutions $X_t$ on $[0,T]$ with $X_0  \in \p \Omega_0$ satisfy $X_t  \in \p \Omega_t$ for all time $t \in [0,T]$. This is the only place in the proof that we need $\Omega_t$ to be a patch solution.

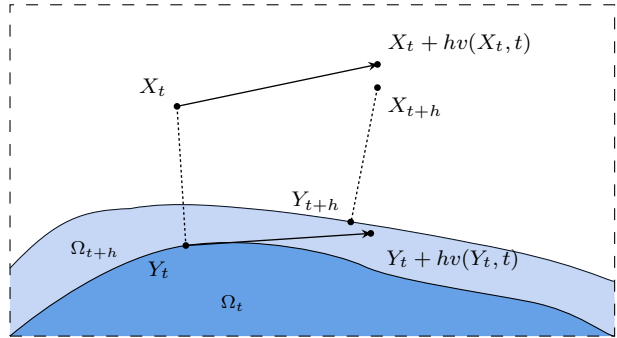
\begin{figure}[h]
\tikzset{every picture/.style={line width=0.5pt}} 
\centering
\begin{tikzpicture}[x=1.25pt,y=1.25pt,yscale=-1,xscale=1]

\draw [line width=0.4pt]    (0,29.5+ 50) .. controls (13.91  ,24.61 + 40 ) and (21.55,22.02 + 40 ) .. (36.05,21.52 + 40  ) .. controls (49.05,19.02 + 40 ) and (66.05,21.02 + 40 ) .. (76.5,22 + 40 ) .. controls (87.55,23.02 + 40 ) and (119.55,28.02 + 40 ) .. (138.5,32 + 40 ) .. controls (168.36,37.34 + 40 ) and (178.47,43.3 + 40  ) .. (181,43.5 + 40 );

\draw [dash pattern={on 0.0pt off 2.5pt}][fill={rgb, 255:red, 74; green, 130; blue, 226 }  ,fill opacity=0.32 ]   (0,29.5+ 50) .. controls (13.91  ,24.61 + 40 ) and (21.55,22.02 + 40 ) .. (36.05,21.52 + 40  ) .. controls (49.05,19.02 + 40 ) and (66.05,21.02 + 40 ) .. (76.5,22 + 40 ) .. controls (87.55,23.02 + 40 ) and (119.55,28.02 + 40 ) .. (138.5,32 + 40 ) .. controls (168.36,37.34 + 40 ) and (178.47,43.3 + 40  ) .. (181,43.5 + 40 ) -- (181,100)--(0,100) --(0,0) ;

\draw [line width=0.4pt][fill={rgb, 255:red, 74; green, 144; blue, 226 }  ,fill opacity=0.77 ]   (0,100) .. controls (9.1,92.2) and (31.64,75.66) .. (52.64,72.66) .. controls (73.64,69.66) and (97.62,74.93) .. (106.08,78.97) .. controls (114.55,83.02) and (140.91,86.75) .. (155.64,89.66) .. controls (170.36,92.57) and (177.97,99.8) .. (181,100) ;

\draw [-stealth]   (52.64,72.66)  -- (108 ,69) ;

\filldraw
		(108 ,69 ) circle (1pt) node[below right=1mm] {\scriptsize $Y_t + h  v(Y_t, t) $};

\filldraw		(52.64,72.66) circle (1pt) node[below left=1mm] {\scriptsize $Y_t $};

\draw [dash pattern={on 1pt off 1pt}] (50,30.66) --(52.64,72.66) ;

\draw [dash pattern={on 1pt off 1pt}] (110,25) --(102,66) ;

\filldraw		(102,65.55) circle (1pt) node[above left=0mm] {\scriptsize $Y_{t+h} $};

\filldraw		(50,30.66) circle (1pt) node[above left =0mm] {\scriptsize $X_t $};
\filldraw		(110,25) circle (1pt) node[below right=0mm] {\scriptsize $X_{t+h} $};

\draw [-stealth]   (50,30.66)  -- (110 ,18);

\filldraw   (110 ,18) circle (1pt) node[above right=0mm] {\scriptsize $X_t + h v(X_t , t ) $} ;

\draw  [dash pattern={on 4.5pt off 4.5pt}][line width=0.4]  (0,0) -- (181,0) -- (181,100) -- (0,100) -- cycle ;

\draw 	(15,74) node[anchor=west] {\tiny $\Omega_{t + h }$} ;
\draw 	(60,90) node[anchor=west] {\tiny $\Omega_{t }$} ;

\end{tikzpicture}
\caption{Relative positions for $X_t, Y_t$ and their approximate evolution $X_t + h  v(X_t, t ), Y_t + h  v(Y_t, t )$.}
\label{fig:no_exit_proof}
\end{figure}

Consider the distance $ m(t) : = \dist(X_t  , \p \Omega_t )$. We first claim that $m(t)$ is Lipschitz on $[0,T]$. Indeed, for any $t,t+h \in [0,T]$, without loss of generality assume $ m(t+h ) \geq  m(t)$ and $h>0$. Let $Y_t  \in \p \Omega_t $ be such that $|X_t  - Y_t| = \dist (X_t, \p \Omega_t)$, then
\begin{align}\label{eq:aux_no_exit_distance_0}
  m(t +h ) - m(t )  & =  \dist(X_{t + h }, \p \Omega_{t+h}  ) -|X_t - Y_t|.
\end{align}
Since $\Omega_t $ is an $\alpha$-patch, by Lemma \ref{lemma:Eulerian_to_geoflow}, it flows with $v$. Therefore, by $Y_t \in \p \Om_t$ we have \begin{align}\label{eq:aux_no_exit_distance_0b}
\dist(Y_t + h v(Y_t , t) , \p \Omega_{t+h}) = o(h)  .
\end{align}
It follows from \eqref{eq:aux_no_exit_distance_0b} that
\begin{equation}\label{eq:aux_no_exit_distance_0c}
\dist(X_{t+ h }, \p \Omega_{t+h}  ) \leq \Big | X_ { t+h } -  \big(Y_{t} + h v(Y_t , t)\big) \Big| + o(h) ,
\end{equation}
and hence by \eqref{eq:aux_no_exit_distance_0} and the triangle inequality
\begin{equation}
\begin{aligned}\label{eq:aux_no_exit_distance_1}
 m(t +h ) - m(t )& \leq \Big | X_{t+h} -  \big(Y_t + h v(Y_t , t)\big)   - \big(  X_t - Y_t  \big)\Big| +o(h) \\
& \leq \big | X_{ t+h } -  X_t \big |  + \big |   h v(Y_t , t)   \big| +o(h) \leq C h,
\end{aligned}
\end{equation}
where the constant $C$ depends on the patch $\Omega$ and $\alpha$. So $m$ is a.e. differentiable on $[0,T]$ by the Rademacher theorem.

Finally, we claim that $\frac{d m}{dt} \leq C m(t)   $ on $[0,T']$ provided $T'>0$ is sufficiently small. It would follow that  $m(t) = 0$ for all $t\in [0,T]$ by repeated applications of Gronwall's inequality. So any $C^1$ solution $X_t$ of \eqref{eq:trajectory} with $X_0  \in \p \Omega_0$ remains on $\p \Omega_t$ for all time $t \in [0,T]$.

Now it remains to prove the claim $\frac{d m}{dt} \leq C m(t)   $ to close the argument. Without loss of generality, we assume $T'>0 $ is sufficiently small so that $m(t)\leq \ep $ with $\ep>0$ from Lemma \ref{lemma:normal_lipschitz}. Again, it suffices to show that for any fixed $t \in [0,T')$ and for all   sufficiently small $h>0$,
 \begin{equation}\label{eq:aux_distance_lip}
|m(t +h )|^2 - |m(t )|^2 \leq  C  h |m(t )|^2 + o(h) .
\end{equation}

To show \eqref{eq:aux_distance_lip}, we just need to strengthen the previous augment for the Lipschitz continuity. Since $X_t$ solves \eqref{eq:trajectory} and $v$ is continuous on $\RR^2 \times [0,T]$, we have $X_{ t+h} -X_t = h v(X_t , t) + o(h) $. As in Figure \ref{fig:no_exit_proof}, it follows from \eqref{eq:aux_no_exit_distance_0c} that
\begin{align*}
|m(t +h )|^2  \leq  | X_t - Y_t + h( v(X_t , t) - v(Y_t , t)) |^2 + o(h).
\end{align*}
So the difference of the squares gives
\begin{equation}
\begin{aligned}\label{eq:aux_distance_lip_2}
& |m(t +h )|^2 - |m(t   )|^2 \\
&=  2h ( X_t - Y_t) \cdot   ( v(X_t , t) - v(Y_t , t))  + O(h^2)+ o(h)
\end{aligned}
\end{equation}
where $  \cdot  $ denotes the standard inner product on $\RR^2$.

Now $\sup_{[0,T']} m(t) \leq C T'  $ can be as small as we want by taking $T'>0$ small. Also since $Y_t$ is the closest point to $X_t$, we have $  X_t - Y_t  $ is parallel with the normal at $Y_t$. Thus we can apply Lemma \ref{lemma:normal_lipschitz} to the first term on the right-hand side of \eqref{eq:aux_distance_lip_2} and obtain that
\begin{align*}
|m(t +h )|^2 - |m(t   )|^2 & \leq   C |m(t)|^2 + o(h)
\end{align*}
for all sufficiently small $h>0$.

\textbf{\underline{Step 3: Uniqueness }}

Suppose $X_t$ and $Y_t$ are two solutions of \eqref{eq:trajectory} with $X_0 = Y_0 $. We aim to prove $X_t = Y_t$.

First of all, $X_t, Y_t \in \p \Omega_t$ for $t \in [0,T]$ from Step 2. Consider the quantity $|X_t - Y_t|^2$ which is Lipschitz and satisfies
\begin{equation}\label{eq:aux_distance_lip_3}
\frac{d|X_t -Y_t|^2}{dt} \leq 2 |X_t-Y_t| |v(X_t,t) -v(Y_t,t)|.
\end{equation}
The fundamental theorem of calculus applied to $v $ restricted to $\p\Om$ yields
\begin{align*}
|v(X_t ,t) - v(Y_t,t)| \leq  \int_{\gamma_{X_t,Y_t}} | \p_s v | \, d s
\end{align*}
where $\g_{X_t,Y_t}$ is the distance arc from $X_t$ to $Y_t$ on $\p \Omega_t$. Since the moving domain $\Omega_t$ is uniformly of class $W^{2,p}$, there exists a constant $ 0< C <\infty $ independent of time such that
$$
|\gamma_{X_t,Y_t}| \leq C |X_t- Y_t| \quad \text{for any $ X_t,Y_t \in \p \Omega_t $} .
$$
By Lemma \ref{lemma:estimate_v_on_boundary}, $\p_s v$ is uniformly bounded on $\cup_{t\in [0,T]} \p \Omega_t \times \{t\}$, and therefore we have for all $t\in [0,T]$ and all $X_t , Y_t \in \p \Omega_t$,
\begin{align}\label{eq:aux_distance_lip_4}
|v(X_t ,t) - v(Y_t ,t)| \leq  C  |\gamma_{X_t, Y_t}| \lesssim  |X_t- Y_t |  .
\end{align}
It follows from \eqref{eq:aux_distance_lip_3} and \eqref{eq:aux_distance_lip_4} that
$$
\frac{d|X_t -Y_t |^2}{dt} \leq C |X_t-Y_t|^2 ,
$$
and the Gronwall inequality implies $|X_t - Y_t| =0$ on $[0,T]$.

\end{proof}

\subsection{The flow of \texorpdfstring{$\alpha$}{a}-patches}\label{subsec:flow_asqg_patches}\label{subsec:Eu_to_Lagran_6}

We have thus far established the wellposedness of trajectories on the patch boundary for a given $W^{2,p}$ $\alpha$-patch. We now ``bundle together'' the trajectories into a flow map. To reduce the notation, let us introduce the flow of a patch solution.

\begin{definition}
Let $ 0< \alpha < \frac{1}{2}$ and $\Omega_t$ be a $\alpha$-patch on $[0,T]$. We say $\g:\TT\times [0,T] \to \RR^2$ is a flow of the patch $\Omega_t$ if
\begin{itemize}
    \item The map $ x  \mapsto \g( x ,t)$ is a  parameterization (i.e. $|\g'|>0$) of $\p \Om_t$ for every $t\in [0,T]$;
    \item The flow equation $\p_t \g = v(\g,t )$ holds on $\TT \times [0,T]$, where $v(x,t) :\RR^2 \times [0,T] \to \RR^2$ is the velocity field generated by the patch $\Omega_t$.
\end{itemize}

If in addition $\g(\cdot, t) $ is of class $C^{1,\sigma}$ for some $\sigma>0$, we say $\g$ is a $C^{1,\sigma}$ flow of $\Omega_t$.
\end{definition}

The main result of this section is the following.
\begin{theorem}\label{thm:flow_existence}
Let $0< \alpha< \frac{1}{2}$ and    $p > \frac{1}{1-2 \alpha}$. Suppose  $\Omega_t$ is a $W^{2,p}$ $\alpha$-patch with the initial data $\Omega_0$ and $v =K_\alpha  *  \chi_{\Omega_t}:\RR^2\times [0,T] \to \RR^2$ is the velocity field of $\Omega_t$.

For any parametrization of $ \p \Omega_0$, $\g_0\in W^{2,p}(\TT) $,  the Cauchy problem for the flow $\gamma: \TT \times [0,T] \to \RR^2$
\begin{equation}
\begin{cases}
\p_t \gamma(x,t) = v(\gamma(x,t) , t) &\\
\gamma |_{t=0} = \gamma_0&
\end{cases}
\end{equation}
has a unique solution $\gamma \in  C([0,T]; C^{1,\sigma }(\TT))$ for  $\sigma = 1 -\frac{1}{p} -2\alpha >0 $.
\end{theorem}

\begin{proof}[Proof of Theorem \ref{thm:flow_existence}]
\hfill

We define the parameterization $\g  $ as follows. For each $x  \in \TT$ define $\g ( x  ,t)  $ to be the unique $C^1$ solution of \eqref{eq:trajectory} with the initial data $ \g_0(
 x ) \in \p \Omega_0 $. Then $\g( x , t) $ is well-defined by Lemma~\ref{lemma:unique_path_on_boundary}. Now let us show that $ x  \mapsto \g( x , t)$ is Lipschitz in $x$.

\noindent
\textbf{\underline{Step 1: Lipschitz regularity of the flow}}

We first show that $x \to \gamma(x, t)$ is Lipschitz for each $t\in [0,T]$. Let $\Delta_\delta$ be the difference operator in $x$ with spacing $\delta>0$.
\begin{align*}
\Delta_\delta \gamma(x, t) = \Delta_\delta  \gamma_0( x ) + \int_0^t \Delta_\delta  v(\gamma(x, t'),t'  ) \, dt' .
\end{align*}
By the fundamental theorem of calculus and Lemma \ref{lemma:estimate_v_on_boundary},
\begin{align*}
\big|  \Delta_\delta \gamma(x, t)  \big| \leq C \delta  + \int_0^t     \big|  \p_s   v\big|_{L^\infty } \dist_{\p \Omega(t')}(\gamma(x +\delta, t'),\gamma(x, t') ) \, dt' .
\end{align*}
Since $ \Omega_t$ is a $W^{2,p}$ $\alpha$-patch, we have   $ \dist_{\p \Omega(t')}(\gamma(x +\delta, t'),\gamma(x, t') ) \leq C(\Omega,\alpha) | \Delta_\delta \gamma(x, t') |$. It follows that
\begin{align*}
\big|  \Delta_\delta \gamma(x, t)  \big| \leq C \delta  + C(\Omega,\alpha) \int_0^t   | \Delta_\delta \gamma(x, t') | \, dt' .
\end{align*}
Then the Gronwall inequality implies that
$$
\big|  \Delta_\delta \gamma(x, t)  \big| \leq C \delta \quad \text{for all $ t\in[0,T]$} .
$$
Thus $x \mapsto \gamma(x,t)$ is Lipschitz for each fixed $t \in [0,T]$. Moreover, repeating the argument above shows that $ c \leq |\p_x {\g}(x )| \leq C$ for a.e. $x \in\TT$ at each $t \in [0,T]$.

\textbf{\underline{Step 2: Higher regularity of the flow}}

Next, we show the higher-order H\"older regularity. This is expected as the velocity satisfies $\p_s v \in C^\sigma$ in the arc-length. However, we need to be careful about the differentiability, as the map $x \mapsto \g(x,t)$ has only been shown to be Lipschitz for each fixed $t$ so far.

For each $t \in [0,T]$ let $\mathcal{N}_t \subset \TT $ be a null set such that $ x \mapsto \g(x,t)$ is differentiable  on $ x \in \TT\setminus \mathcal{N}_t $. Let us consider the (measurable and bounded) metric
\begin{equation}\label{eq:aux_g_for_lip_gamma}
g( x ,t )=
\begin{cases}
|\p_x \gamma ( x ,t ) | \quad \text{ if $  x \not\in \mathcal{N}_t $ } & \\
0 \quad \text{otherwise.} &
\end{cases}
\end{equation}
Let $\overline{v}(s,t)$ be the velocity in the arc-length $ s= \ell( x ,t) = \int_0^x   g( y,t )\, d  y   $, namely $ \overline{v}(\ell( x ,t ) ,t): = v(\g(x,t ), t )$.

Since $\g(x,t)$ is defined point-wise, by trajectories of the flow equation $\p_t \g (x, t) = v (\g (x,t) ,t)$, we have
\begin{equation}\label{eq:aux_g_integral}
\g (x,t) = \g_0(x) + \int_0^t \overline{v} (\ell(x, t'), t') \, dt' \qquad \text{for all $(x,t) \in \TT \times [0,T]$}.
\end{equation}
 We claim that for each fixed $t \in [0,T]$, the function $x \mapsto \int_0^t \overline{v} (\ell(x, t'), t') \, dt'$ is $a.e.$ differentiable on $\TT$ and
\begin{equation}\label{eq:aux_g_integral_derivative}
\p_x \int_0^t \overline{v} (\ell(x, t'), t') \, dt' = \int_0^t \p_s \overline{v} (\ell(x, t'), t') g(x,t') \, dt'.
\end{equation}
To show \eqref{eq:aux_g_integral_derivative}, consider for any smooth $\varphi \in C^\infty(\TT)$ the integral
$$
\int_{\TT}  \p_x \varphi(x)  \int_0^t \overline{v} (\ell(x, t'), t') \, dt' \, dx
$$
By Fubini, the regularity of $\overline{v} $, and Lipschitzness of $\ell$, we may integrate by parts in the $x$-variable and hence conclude that \eqref{eq:aux_g_integral_derivative} holds $a.e.$ on $\TT$.

From \eqref{eq:aux_g_integral} and \eqref{eq:aux_g_integral_derivative}, by redefining $ \mathcal{N}_t$ if necessary, it follows that for each   $ t\in [0,T]$ and all $  x\in \TT \setminus \mathcal{N}_t$
\begin{equation}\label{eq:aux_g'_integral}
\begin{aligned}
 \p_x \g (x, t ) & = \p_x \g_0  ( x) + \int_0^t \p_s \overline{ v } (\ell( x , t') , t')  g(x, t' )\, dt' .
\end{aligned}
\end{equation}
We first show that $\g \in C^{1,\sigma}(\TT)$ uniformly in time. We claim that it suffices to show that
\begin{equation}\label{eq:aux_g_holder_claim}
\text{for all}\,\,  x,y \in \TT \setminus \mathcal{N}_t \,\, \text{ there holds} \,\, |g(x ,t ) - g(y ,t )| \leq C| x -y|^\sigma  .
\end{equation}
Indeed, this claim would imply that the maps $x \mapsto \ell( x,t)$ are $C^{1,\sigma}$ diffeomorphisms from $ \TT$ to $ \frac{L \TT}{2\pi}  $. Since $s \mapsto \gamma( \ell^{-1}(s,t ),t )$ is an arc-length parametrization of $\p \Omega$ which is $W^{2,p} \subset C^{1,1-\frac{1}{p}}$  and $\sigma = 1 - \frac{1}{p} -2\alpha < 1- \frac{1}{p}$, we can conclude that $x \mapsto \g (x ,t)$ is $C^{1,\sigma}$ in this case.

The task then reduces to showing the claim \eqref{eq:aux_g_holder_claim}. By \eqref{eq:aux_g_for_lip_gamma}, \eqref{eq:aux_g'_integral}, and the triangle inequality, for any fixed $t$ and $x,y \in \TT\setminus \mathcal{N}_t$,  we obtain that
\begin{align*}
|g(x ,t ) - g(y ,t )| & \leq  | \g_0' ( x)- \g_0' ( y ) |\\
& \quad + \int_0^t | \p_s \overline{ v } (\ell( x , t') , t') -\p_s \overline{ v } (\ell( y, t') , t')  |  g(x, t' )  \, dt' \\
& \qquad + \int_0^t  | \p_s \overline{ v } (\ell( y , t') , t')    | | g(x, t' ) -g( y , t' )  |\, dt'.	
\end{align*}
Since the initial data $\g_0 \in W^{2,p}(\TT) \subset C^{1 , 1 - \frac{1}{p}}$, it follows that for any $x,y \in \TT\setminus \mathcal{N}_t$,
\begin{align*}
|g(x ,t ) - g(y ,t )|  & \lesssim |x-y|^\sigma  +   C\int_0^t | \ell( x , t')   - \ell( y, t')   |^{\sigma}   \, dt' \\
& \qquad + C \int_0^t   | g(x, t' ) -g( y , t' )  |\, dt' \\
& \lesssim |x-y|^\sigma    + C \int_0^t   | g(x, t' ) -g( y , t' )  |\, dt' .
\end{align*}
It follows from the Gronwall inequality that for any fixed $t$ and $x,y \in \TT \setminus \mathcal{N}_t $, we have $ |g(x ,t ) - g(y ,t )|    \lesssim |x-y|^\sigma $. This proves the claim and hence the uniform $C^{1,\sigma}$ regularity of $ \g$.

Finally, let us show $t\mapsto \g(\cdot,t)$ provides a continuous mapping from $[0,T]$ to $C^{1,\sigma}(\TT)$. By the integral equation \eqref{eq:aux_g_integral}, we have
$$
| \g(\cdot, t_1) - \g(\cdot, t_0) |_{C^{1,\sigma}} \leq |t_1 -t_0 | \sup | v (\g(\cdot, t ) ) |_{C^{1,\sigma}} .
$$
Therefore, the continuity in time $\g \in C( [0,T]; C^{1,\sigma}(\TT)) $ follows directly from the regularity $ \g \in C^{1,\sigma}$ and $\p_s  v \in C^{\sigma }$.

\end{proof}

Note that we have proven the flow of a given $W^{2,p}$ $\alpha$-patch is unique, but we do not prove the $W^{2,p}$ patch itself (in the sense of Definition \ref{def:asqg_patch}) is unique.

\section{The curvature flow of \texorpdfstring{$\alpha$}{a}-patches}\label{sec:curvature_flow}

In this section, we study the curvature dynamics of the $\alpha$-patches. These equations will be understood in a distributional sense using the flow that we obtained in the previous section.

 To rigorously derive the curvature dynamics, we organize this section in the following order:

\begin{enumerate}
\item The $C([0,T];C^{1,\sigma}(\TT))$ flow allows us to derive classically the evolution of arc-length, tangent, and angle.

\item Then we adopt a distributional formulation of the curvature equation \eqref{eq:curvature_eq} based on the quantities arc-length and tangent, both of which are function-valued.

\item Next, we show that the curvature of a given $W^{2,p}$ $\alpha$-patch, under the parameterization of the flow, satisfies the distributional formulation when $ p> \frac{1}{1 -2\alpha }$.

\item Lastly, we show that the   leading order dynamics of the curvature flow has certain dispersive characteristics, which will then be studied in detail in the next two sections.
This leading term is identified in \eqref{eq:aux_curvature_kernal_4} below. 

\end{enumerate}

\subsection{Difficulties}\label{subsec:curvature_idea}

Let us first explain the issues that necessitated the use of a distributional curvature dynamics.

In the smooth setting, when the parameterization $\g: \TT \times [0,T] \to \RR^2$ and the velocity field are smooth, the (Lagrangian) curvature equation reads
\begin{equation}\label{eq:curvature_eq}
\p_t \k = - \k \p_s  v \cdot \T -\p_s( \p_s v \cdot \N )
\end{equation}
(see, e.g. \cite{c2euler}). From our previous analysis of the Euler patch case~\cite{c2euler}, the main term in \eqref{eq:curvature_eq} should be the last term on the right-hand side (see \eqref{eq:aux_curvature_kernal_4} below). However, when dealing with \eqref{eq:curvature_eq} in the $\alpha$-patch problem, we face two substantial problems compared to the smooth setting or even the Euler case:
\begin{enumerate}
\item the flow $\g$ is less regular, only $C^{1,\sigma}(\TT) $, not even twice differentiable; and
\item the velocity field does not allow us to differentiate $\p_s v \cdot \N$, cf. Lemma \ref{lemma:estimate_v_on_boundary}.
\end{enumerate}

The first problem can be overcome by the strategy explained in Section \ref{subsec:rough_gamma_regular_Omega} by using the $C^{1,\sigma} $ diffeomorphism induced by the flow $\g$ and the intrinsic regularity of $\p \Omega_t$.

However, the second problem is more severe and one needs to make sense of \eqref{eq:curvature_eq} rigorously in the $\alpha$-patch dynamics.  Our strategy to overcome this problem is to use the classical idea of distributional formulation. One has to be careful, as the differentiation in \eqref{eq:curvature_eq} is in the arc-length variable.

\subsection{The evolution of arc-length, tangent, and angle of the flow} \label{subsec:curvature_flow_1}

Since the flow $\g$ is $C([0,T];C^{1,\sigma}(\TT))$, by Lemma \ref{lemma:geo_quant_rough_gamma} the arc-length metric $g= |\p_x \g| :\TT \times [0,T] \to \RR^+ $,  tangent vector $\T: \TT \times [0,T] \to \RR^2$, and normal vector $\N: \TT \times [0,T] \to \RR^2$ are well-defined in this parameterization.
In what follows, these letters always refer to the quantities according to a given flow of a $W^{2,p}$ $\alpha$-patch.

We now derive evolution equations for $g$ and $\T$, necessary for defining the curvature dynamics.

By the integral form of the flow equation
\begin{equation}\label{eq:aux_int_flow}
\g(x,t) = \g_0 (x) + \int_0^t v( \g(x,t') , t') \,dt'
\end{equation}
and the velocity field regularity $ \p_s v \in C^{ \sigma}(\TT)$, we can differentiate \eqref{eq:aux_int_flow} to find that $ \p_x \g \in C([0,T];C^{ \sigma}(\TT)) \cap C(\TT;C^{1}([0,T] ) ) $  satisfies  the equation
\begin{equation}\label{eq:aux_eq_d_gamma}
\p_t (\p_x \g) = \p_s v g.
\end{equation}

Since $ g \T = \p_x \g $ and $g$ is uniformly bounded from below, from \eqref{eq:aux_eq_d_gamma} it follows that the metric $g \in C([0,T];C^{ \sigma}(\TT))$ satisfies the equation
\begin{equation}\label{eq:aux_eq_g}
\begin{cases}
\p_t g = g \p_s v \cdot \T   &\\
g |_{t =0} =  |\p_x \g_0|  &
\end{cases}
\end{equation}
By the identity $ g \T = \p_x \g $, \eqref{eq:aux_eq_d_gamma}, and \eqref{eq:aux_eq_g} again, we have the evolution for the unit tangent vector $\T \in C([0,T];C^{ \sigma}(\TT))$,
\begin{equation}\label{eq:aux_eq_T}
\begin{cases}
	\p_t \T = (\p_s v \cdot \N) \N &\\
\T |_{t =0} = |\g_0|^{-1} \p_x \g_0. &
\end{cases}
\end{equation}

We remark that these equations \eqref{eq:aux_eq_d_gamma}, \eqref{eq:aux_eq_g}, and \eqref{eq:aux_eq_T} are satisfied classically.

In what follows, we need to consider the angle $\theta\in [-\pi,\pi ]$ between the tangent vector $\T$ and the $x_1$-axis, define via
\begin{equation}\label{eq:def_angle_theta}
\T (x, t) =( \cos ( \theta(x,t ))  , \sin ( \theta(x,t )) ) .
\end{equation}

In general, $ \theta $ is not a continuous function on $\TT$, however we still have that $x \mapsto \theta(x,t) - x $ is $ C^{ \sigma}(\TT) $. In particular,  since the flow $\g$ parameterizes $\p \Om$ that is $ W^{2,p} $, by Lemma \ref{lemma:geo_quant_rough_gamma} and \eqref{eq:def_angle_theta}, the curvature $\k(x,t)$ of $\p \Om_t$ at $ \g(x,t)$ satisfies
\begin{equation}\label{eq:theta_kappa}
\p_s \theta (x,t )=\k  (x,t ) \quad \text{for any $ (x, t)\in (-\pi , \pi ) \times [0,T ]$}
\end{equation}
and by \eqref{eq:aux_eq_T} the angle itself satisfies the evolution equation
\begin{equation}\label{eq:aux_theta}
\begin{cases}
\p_t \theta =  -\p_s v \cdot \N &\\
\theta |_{t =0} =  \theta_0 .  &
\end{cases}
\end{equation}

\subsection{Definition of curvature flow for \texorpdfstring{$\alpha$}{a}-patches}

We are now ready to consider the curvature dynamics of  $\alpha$-patches.

Given a time-dependent metric $g : \TT \times [0,T] \to (0,\infty)$, for any $ f : \TT \to \RR$ recall its arc-length derivative $\p_s f (x ): = \frac{1}{g} f'(x ) $ and its integration with $ds (x) : = g dx $ is $ \int_{\TT} f \,ds = \int_{\TT} f  g \,dx$.

In what follows, the letter $g$ is always referring to the metric given by the flow $\g$ of an $\alpha$-patch $\Omega_t$, namely
$$
g : = |\g'|.
$$

As a standard practice in PDE, we derive a suitable distributional formulation of the curvature equation \eqref{eq:curvature_eq} in the non-smooth setting.

\begin{definition}\label{def:kappa_flow}
Let $0<\alpha < \frac{1}{2}$. Let $\Omega_t$ be a $W^{2,p}$ $\alpha$-patch. Let $\g \in C([0,T];C^{1,\sigma}(\TT)) $ be a flow of $ \p \Omega_t$ for some $\sigma>0$ with the initial data $\g_0$.

We say $\k: \TT\times [0,T] \to \RR $ is a curvature flow (of $\g$) if
\begin{itemize}
\item   $\k(\cdot , t )   \in L^p(\TT) $ denotes the curvature of $\p \Omega_t$ at $ \g(\cdot , t )$ for all $t \in [0,T]$;

\item $\k$ satisfies
\begin{equation}\label{eq:def:kappa_flow}
 \int_0^T \int_{\TT} \k \p_t \phi \,g dx dt  = -\int_0^T \int_{\TT } \p_s v\cdot \N \p_s \phi g\, dx  dt
\end{equation}
for any $\phi \in C^\infty(\TT \times \RR)$ supported in $\TT \times (0,T)$.
\end{itemize}

\end{definition}
In the remaining part of Section \ref{sec:curvature_flow}, one can always take $\sigma = 1-\frac1p-2\alpha,$ since this is the regularity of the flow corresponding to $W^{2,p}$ patch that we proved in Section \ref{sec:Eu_to_Lagran}.

We will use a more convenient version of the weak formulation. A standard argument that we sketch below  gives
\begin{lemma}\label{lemma:kappa_better_formulation}
Let $\Omega_t$ be a $W^{2,p}$ $\alpha$-patch and $\g \in C([0,T]; C^{1,\sigma} (\TT) ) $ be a flow of $ \p \Omega_t$ for some $\sigma>0$.

If $\k$ is a curvature flow of $\g$ then $\k \in L^\infty([0,T] ; L^p(\TT))$  and $\k$   is continuous as a distribution, i.e.
\begin{equation}\label{eq:kappa_weak_continuty}
t\mapsto \int_{\TT} \k(x,t) \varphi(x,t) \,dx \quad \text{is continuous for every $\varphi \in  C^\infty(\TT \times \RR))$}
\end{equation}
and also
\begin{equation}\label{eq:kappa_better_formulation}
\begin{aligned}
\int_0^T \int_{\TT} \k \p_t \phi g\, dx dt  + \int_0^T \int_{\TT }     \p_s v\cdot \N \p_s \phi g\, dx  dt\\
 =    \int_{\TT} \k(x, T) \phi(x, T)g(x,T)\, dx -\int_{\TT} \k(x, 0) \phi(x,0)g(x,0)\, dx
\end{aligned}
\end{equation}
for all test functions $\phi \in C^\infty(\TT \times \RR).$
\end{lemma}
\begin{proof}

To show \eqref{eq:kappa_weak_continuty}, it suffices to show the continuity for
\begin{equation}\label{eq:aux_kappa_weak_continuty}
t\mapsto \int_{\TT} \k(x,t) \varphi(x,t) g(x,t) \,dx  .
\end{equation}
Indeed, \eqref{eq:aux_kappa_weak_continuty}, the time-continuity and positivity of $g$ will then imply the continuity of $ \int_{\TT} \k(x,t) \varphi(x,t)   \,dx $ for all smooth $\varphi.$

To show \eqref{eq:aux_kappa_weak_continuty}, we observe that the right-hand side of \eqref{eq:aux_kappa_weak_continuty}   is $ \int   \p_x \theta    \varphi  \,dx  $.
We can integrate by parts and obtain
\begin{equation}\label{eq:aux_kappa_weak_continuty_2}
\int_{\TT} \k(x,t) \varphi(x,t) g(x,t) \,dx  = -\int_{\TT}    \theta  (x,t)   \p_x \varphi(x,t)  \,dx .
\end{equation}

Since the flow $\g \in C([0,T]; C^{1,\sigma} (\TT) ) $,   we have $\theta \in  C([0,T]; C^{ \sigma} (\TT) )$ and hence the continuity of \eqref{eq:aux_kappa_weak_continuty_2} follows.

Next, we prove \eqref{eq:kappa_better_formulation}. For $\ep>0 $, we consider a sequence of cutoffs $h_\ep  \in C^\infty_c((0,T))$ such that $h_\ep=1$ on $[\ep,T- \ep ]$ and   $|h'(t) | \leq C\ep$. We use $h_\ep(t) \phi(x,t)$ as a test function in the weak formulation of $\k$ to obtain

\begin{equation}\label{eq:aux_weak_kappa_1}
\int_0^T \int_{\TT} h_\ep \k \p_t \phi g\, dx dt  + \int_0^T \int_{\TT } h_\ep  \p_s v\cdot \N \p_s \phi g\, dx  dt = -\int_0^T  \int_{\TT} \k  \phi \p_t h_\ep g\, dx dt.
\end{equation}

Since $h_\ep$ converges to $1$ a.e. on $[0,T]$, by the dominated convergence the left-hand side of \eqref{eq:aux_weak_kappa_1} converge to that of \eqref{eq:kappa_better_formulation} as $\ep \to 0^+$.

We now focus on the right-hand side of \eqref{eq:aux_weak_kappa_1}. Let us denote $\k \phi g\in L^\infty([0,T];L^p(\TT)) \cap C ([0,T]; \mathcal{D}'(\TT))$ by $f;$ then it suffices to show that
\begin{align}\label{eq:aux_time_continuity}
\int_{0}^{\ep} \int_{\TT} f(x,t) \p_t h_\ep(t) \,dx dt \to  \int_{\TT} f(\cdot,0) \,dx.
\end{align}
and a similar statement for an integral near $T$.

To show \eqref{eq:aux_time_continuity}, we first use $\int_{0}^{\ep} \p_t h =1$ to obtain
\begin{align*}
   \int_{ 0 }^{\ep} \int_{\TT} f(x,t) \p_t h_\ep(t) \,dx dt&  = \int_{0 }^{\ep} \int_{\TT} \big(f(x,t)-f(x,0) \big) \p_t h_\ep(t) \,dx dt \\
&\qquad +  \int_{\TT}  f(x,0)     \,dx .
\end{align*}
Then by H\"older's inequality, we have
\begin{align*}
  &   \Big| \int_{ 0 }^{\ep} \int_{\TT}  ( f(x,t) -  f(x,0) )     \p_t h_\ep(t) \,dx dt\Big|  \\
& \leq  \sup_{t \in [0,\ep]} \left| \int_{\TT}  (f(x,t)-f(x,0))\,dx \right| |\p_t h_\ep    |_{L^1(\RR)}.
\end{align*}
Since $\k$ is continuous as a distribution, $\sup_{t \in [0,\ep]} \big| \int_{\TT} f(\cdot,t)-f(\cdot,0) \big|  \to 0 $  as $\ep \to 0^+$. Then \eqref{eq:aux_time_continuity} follows from the above bound and $ |\p_t h_\ep    |_{L^1(\RR)}  \leq C$.

Repeating the same argument we can also conclude that
$$
\int_{T-\ep }^{T} \int_{\TT} f(x,t) \p_t h_\ep(t) \,dx dt \to  \int_{\TT} f(\cdot,T ) \,dx \quad \text{as $\ep \to 0^+$.}
$$
Hence the right-hand side of \eqref{eq:aux_weak_kappa_1} also converges to that of \eqref{eq:kappa_better_formulation}.
\end{proof}

\subsection{Existence of curvature flow for \texorpdfstring{$\alpha$}{a}-patches}

 We now show the existence of a unique curvature flow for $W^{2,p}$ $\alpha$-patches with $ p> \frac{1}{1 -2\alpha }$.

\begin{proposition}\label{prop:kappa_existence}
Let $0 < \alpha < \frac{1}{2}$ and $ p> \frac{1}{1 -2\alpha }$. Suppose $ \g \in C([0,T];C^{1,\sigma} (\TT)) $ is the flow of  a $W^{2,p}$ $\alpha$-patch $\Omega_t$ with the initial data $\g_0$ given by Theorem \ref{thm:flow_existence}.

Then there is  a (unique) curvature flow $\k$  and it satisfies the regularity $\k \in L^\infty([0,T]; L^p(\TT)) \cap C_w([0,T],L^p(\TT)).$
\end{proposition}
\begin{proof}
Recall that we have defined $\k (x ,t)$ to be the curvature of $ \p \Omega_t$ at $\g(x ,t)$. Alternatively, let $\overline{\g}(s,t) : = \g(\ell^{-1}(s), t ) $ where $\ell^{-1}$ is the inverse of the $C^{1,\sigma}$ diffeomorphism $ x \mapsto \ell(x) = \int_0^x g(y) \,dy $. Then we also have $\k (x ,t) :=  \overline{\k}(\ell^{-1}(x) , t) $.

By \eqref{eq:aux_theta},  for any $\phi \in C^\infty(\TT \times \RR)$ supported in $(0,T)$, we have
\begin{align}\label{eq:aux_kappa_existenc_1}
\int_0^T \int_{\TT} \theta  \p_t \varphi' \,dxdt = \int_0^T \int_{\TT}  \p_s v \cdot \N  \varphi' \,dx dt.
\end{align}
Since $ \theta(x ,t) = \overline{\theta}( \ell(x) , t)$ and $ \ell$ is a $C^{1,\sigma}$ diffeomorphism for each $t \in [0,T]$, we have $ \p_x \theta (x,t)  = g (x,t) \overline{\k}( \ell(x) , t)  = g(x,t) \k(x,t) $. Integrating by parts on  the left-hand side of \eqref{eq:aux_kappa_existenc_1},  we have
\begin{align}
\int_0^T \int_{\TT}   \k  \p_t \varphi g \,dxdt = - \int_0^T \int_{\TT}  \p_s v \cdot \N  \p_s \varphi g \,dx dt
\end{align}
which is exactly
\begin{align}
\int_0^T \int_{\TT}   \k  \p_t \varphi  \,ds \, dt = - \int_0^T \int_{\TT}  \p_s v \cdot \N  \p_s \varphi   \,ds \,dt .
\end{align}

\end{proof}

\subsection{The dispersive operator \texorpdfstring{$\mathcal{L}_\alpha$}{La}}

Throughout the paper, the dispersive operator $\mathcal{L}_\alpha$ is defined through Fourier multiplier,
\begin{equation}\label{eq:def_L_a}
\mathcal{L}_\alpha f (x) := \sum_{k\in \ZZ}i k |k|^{2\alpha - 1} \hat{f}_k e^{i k x} .
\end{equation}

The following lemma provides a characterization of the convolution kernel of $\mathcal{L}_\alpha$. Recall that \begin{itemize}
\item Functions on $\TT$ are considered as $2\pi$-periodic functions on $\RR$;

\item on $\RR$ the principle value $P.V. \int_{\RR} f $ is understood as
$$
\lim_{\substack{\ep \to 0^+ \\
N\to \infty}}  \int_{\ep \leq y \leq N }  f(y) \, dy,
$$
whenever the limit exists.
\end{itemize}

We now state the characterization lemma that will be useful for analyzing the dispersion in the curvature dynamics.

\begin{lemma}\label{lemma:local_formula_La}
Let $ 0 < \alpha < \frac{1}{2}$. For any $f \in   C^\infty (\TT)$,
\begin{equation}
\mathcal{L}_\alpha f(x) : = c_\alpha^{ -1}  P.V.  \int_{  \RR } \frac{ (x-y) f(y)}{|x-y|^{2+2\alpha}} \, dy
\end{equation}
where $c_\alpha>0$ is a universal constant.

\end{lemma}
\begin{proof}
We need to prove
\begin{equation}
\mathcal{L}_\alpha f(x) : = c_\alpha^{ -1}  \lim_{\substack{\ep \to 0^+ \\
N\to \infty}}\int_{ \ep \leq |x-y| \leq N} \frac{ (x-y) f(y)}{|x-y|^{2+2\alpha}} \, dy.
\end{equation}
By the Fourier inversion and smoothness of $f$, for any $\ep , N>0$, the truncated integral satisfies
\begin{align*}
\int_{ \ep \leq |x-y| \leq N} \frac{ (x-y) f(y)}{|x-y|^{2+2\alpha}} \, dy= \sum_{k} \hat{f}_k \int_{ \ep \leq |x-y| \leq N} \frac{ (x-y) e^{iky }}{|x-y|^{2+2\alpha}} \, dy,
\end{align*}
where $\hat{f}_k$ are the Fourier coefficients of $f$.
Since the Fourier coefficients $ \hat{f}_k$ decay faster than any polynomial and the integrals
$$
I_k(x) : = \int_{ \ep \leq |x-y| \leq N} \frac{ (x-y) e^{iky }}{|x-y|^{2+2\alpha}} \,dy
$$
are uniformly bounded in approximations parameters $\ep$ and $N$, the summation is absolutely convergent and hence it suffices to prove
\begin{equation}\label{aux56a}
 \lim_{\substack{\ep \to 0^+ \\
N\to \infty}}\int_{ \ep \leq |x-y| \leq N} \frac{ (x-y) e^{iky }}{|x-y|^{2+2\alpha}} \, dy. = c_\alpha^{ -1}  i  k |k|^{2\alpha -1} e^{ikx}\quad \text{for all $k \in \ZZ $}.
\end{equation}
By oddness of the kernel $\frac{x-y}{|x-y|^{2+2\alpha }}$ and a shift in the variable, we just need to show the principle value integral $ \int_{\RR} z |z|^{-2-2\alpha} \sin(k z) \, dz = c_\alpha^{ -1} k |k|^{2\alpha -1} $ for $k \in \ZZ$. Since $0 < \alpha < \frac{1}{2}$, the integral converges in the principle value sense, and its value $c_\alpha^{ -1} k |k|^{2\alpha -1}$ follows by a change of variable due to the homogeneity.

\end{proof}

\subsection{Dispersion in the curvature dynamics of \texorpdfstring{$\alpha$}{a}-patches}

Next, we analyze in detail the curvature dynamics in its weak formulation. The goal is to derive the leading order dynamics and reveal its dispersive characteristics.

Below we use the Besov space $B^{s}_{p , q }(\TT)$  on the torus, see Definition \ref{def:BesovLP}. More  on functional setups is introduced in Section \ref{sec:disper_estimates} below.

\begin{proposition}\label{prop:dispersive_term}
Let $0 < \alpha < \frac{1}{2}$ and $ p> \frac{1}{1 -2\alpha }$. Suppose $ \g \in C([0,T];C^{1,\sigma}),$ $\sigma=1-\frac{1}{p}-2\alpha,$ is the flow of  a $W^{2,p}$ $\alpha$-patch $\Omega_t$ with the initial data $\g_0$. Let $v$ be the velocity field generated by the patch $\Omega_t$ parameterized by $\g$ and  $g = |\p_x \g |$ be the metric associated with the flow $\g$.

Then for any $t \in [0,T]$ and any $\phi \in C^\infty(\TT)$,
\begin{equation}\label{eq:lemma_dispersive_term_0}
\int_{\TT} \p_s v \cdot \N \p_s\phi g\,dx = -c_\alpha \int_{\TT}g^{1-2\alpha}  \k \mathcal{L}_\alpha \phi \, dx + \mathcal{E}_t(\phi)
\end{equation}
where $\mathcal{E}_t:C^\infty(\TT) \to \RR$ is a time-dependent distribution
satisfying the bound
\begin{equation}\label{eq:prop_dispersive_error_term}
|\mathcal{E}_t (\phi)| \leq C_{\alpha,\gamma} |\k|_{L^p} |\phi |_{B^{2\alpha -\sigma}_{p', 1 }}  \quad \text{for any $\phi \in C^\infty(\TT) $}
\end{equation}
uniformly in time, where     $p'$ is the H\"older dual of $p$ and the constant $C_{\alpha,\gamma}$ depends on $\alpha$, $p$, $\Om$, and $|\g|_{C^{1,\sigma}(\TT)}$.
\end{proposition}
\begin{proof}

Define the distribution $\mathcal{E}_t : C^\infty(\TT) \to \RR$ by
\begin{equation}\label{eq:aux_dispersive_term}
\mathcal{E}_t (\phi) : =  \int_{\TT}\p_s v\cdot \N \p_s \phi  \,g dx  + c_\alpha \int_{\TT}g^{1-2\alpha}  \k \mathcal{L}_\alpha \phi \, dx.
\end{equation}
It is clear that at each time $t\in [0,T]$ \eqref{eq:aux_dispersive_term} defines a bounded linear functional on $C^\infty(\TT)$, and it suffices to prove the estimate \eqref{eq:prop_dispersive_error_term}.

Consider the (time-dependent) arc-length variable $s : = \ell(\xi) = \int_0^{\xi} g(\tau )\, d\tau $. We define the arc-length counterparts $ \overline{\g},\overline{\T}, \overline{\N}, \overline{\k}$ and $\overline{v} $ by the formula $\overline{f}(\ell(\xi)) = f(\xi) $, i.e. $\overline{f}(s) : = f (\ell^{-1}(s))$ for $f \in \{ \g,\T,\N,\k,v\}$. Immediately by Lemma \ref{lemma:arc_diff} we have
\[ \p_s \overline{f} = \p_s f(\ell^{-1}(s)) = \frac{1}{g(x)}\p_x  f(x).\]

We drop the time variable and domain of integration $[-L(t)/2,L(t)/2]$ for simplicity in the proof.

In the new variable, by Lemma \ref{lemma:estimate_v_on_boundary} the left-hand side of \eqref{eq:lemma_dispersive_term_0} reads
 $$
\int  \p_s \overline{v} \cdot \overline{\N} \p_s\overline{\phi} \,ds = \int  \p_s\overline{\phi} \int \overline{\T}(s') \overline{\N}(s) \frac{( \oga(s) -  \oga(s'))\cdot \oT(s)  }{ |\oga(s) -  \oga(s')|^{2+2\alpha } }\, ds'  \, ds  .
$$

By estimates of Lemma \ref{lemma:arc_length_estimates} and the assumption $p>\frac{1}{1-2\alpha},$ one can check that the expression under integral is absolutely integrable, with a fixed bound for every $s.$
Thus we can use the Fubini theorem, and we can also change variable to obtain
\begin{equation}\label{eq:aux_dispersive_term_0}
\begin{aligned}
& \int  \p_s \overline{v} \cdot \overline{\N} \p_s\overline{\phi} \,ds \\
& = \int  \int  \p_s\overline{\phi}(s  )\overline{\T}(s' + s) \overline{\N}(s  ) \frac{( \oga(s ) -  \oga(s' + s ))\cdot \oT(s   )  }{ |\oga(s   ) -  \oga(s' + s )|^{2+2\alpha } }\,  ds  \, ds' .
\end{aligned}
\end{equation}
In what follows we consider the approximations,
\begin{equation}\label{eq:aux_dispersive_term_1}
 \int_{|  s' | \geq \ep }  \int  \p_s\overline{\phi}(s  )  \overline{\T}(s' + s) \overline{\N}(s  ) \frac{( \oga(s ) -  \oga(s' + s ))\cdot \oT(s   )  }{ |\oga(s   ) -  \oga(s' + s )|^{2+2\alpha } }\,  ds  \, ds',
\end{equation}
and in the end will take the limit $\ep \rightarrow 0.$

With the approximation, we can integrate by parts in the $s $ variable.  Notice that when $s'  \neq 0$, the function $s  \mapsto  \overline{\T}(s' + s) \overline{\N}(s  ) \frac{( \oga(s ) -  \oga(s' + s ))\cdot \oT(s   )  }{ |\oga(s   ) -  \oga(s' + s )|^{2+2\alpha } }$ is at least $W^{1,p}$ whose  derivative in $s$ is equal to the sum of
\begin{align}
& \underbrace{- \ok(s'+ s) \oN(s'+ s ) \oN( s ) \frac{( \oga(s) -  \oga(s'  + s ))\cdot \oT(s)  }{ |\oga(s) -  \oga(s' + s )|^{2+2\alpha } } }_{:= K_1(s,s')} \label{eq:aux_def_K_1}\\
&  \underbrace{+ \ok(s  )  \overline{\T}( s'+ s  ) \oT( s ) \frac{ ( \oga(s ) -  \oga(s' + s ))\cdot \oT(s   )  }{ |\oga(s) -  \oga(s' + s  )|^{2+2\alpha }}}_{:= K_2(s,s')}  \label{eq:aux_def_K_2} \\
& \underbrace{ \overline{\T}(s' + s) \overline{\N}(s  ) \frac{( \oT(s ) -  \oT(s' + s ))\cdot \oT(s   )  }{ |\oga(s   ) -  \oga(s' + s )|^{2+2\alpha } }   }_{:= K_3(s,s')} \label{eq:aux_def_K_3} \\
&  \underbrace{ -  \ok(s  )  \overline{\T}( s'+ s  ) \oN( s ) \frac{ ( \oga(s ) -  \oga(s' + s ))\cdot \oN(s   )  }{ |\oga(s) -  \oga(s' + s  )|^{2+2\alpha }}}_{:= K_4(s,s')}  \label{eq:aux_def_K_4} \\
&  - (2+2\alpha)\overline{\T}( s'+ s  ) \oN( s )  \frac{( \oga(s) -  \oga(s'+s ))\cdot \oT(s)  }{ |\oga(s) -  \oga(s'+s )|^{4+2\alpha } }  \nonumber\\
& \underbrace{\qquad\qquad \times  ( \oga(s) -  \oga(s'+s ))\cdot (\oT(s) - \oT(s'+s ) ) .\,\, }_{:= K_5(s,s')} \label{eq:aux_def_K_5}
\end{align}
We then integrate by parts in \eqref{eq:aux_dispersive_term_1} using the derivatives \eqref{eq:aux_def_K_1}--\eqref{eq:aux_def_K_5} and then reparameterize back to obtain
\begin{equation}
\begin{aligned}\label{eq:aux_dispersive_term_2}
 & \int  \p_s \overline{v} \cdot \overline{\N} \p_s\overline{\phi} \,ds  \\
&  = \lim_{\ep \to 0^+} \int_{|s'| \geq \ep } \int_{}  \overline{\phi}(s)\p_s \left( \overline{\T}(s' + s) \overline{\N}(s  ) \frac{( \oga(s ) -  \oga(s' + s ))\cdot \oT(s   )  }{ |\oga(s   ) -  \oga(s' + s )|^{2+2\alpha } }  \right) \,ds \, ds' \\
& =  \lim_{\ep \to 0^+} \sum_{n} \int_{| s'| \geq \ep  } \int    \overline{\phi}(s)  K_n(s,s') \, ds  \, ds' \\
& = \lim_{\ep \to 0^+} \sum_{1 \leq n \leq 5} I_n(\ep)
\end{aligned}
\end{equation}
where each integral $I_n(\ep)$ corresponds to the kernels $K_n$ from \eqref{eq:aux_def_K_1}--\eqref{eq:aux_def_K_5}.
\begin{align}
I_1 (\ep) & =  -\int \int_{|   s -s' | \geq \ep  } \overline{\phi}(s) \ok(s') \oN(s  ') \cdot \oN( s ) \frac{( \oga(s) -  \oga(s'    ))\cdot \oT(s)  }{ |\oga(s) -  \oga(s'   )|^{2+2\alpha } }   \, ds  \, ds ' \\
I_2(\ep) & =  \int \int_{|   s -s' | \geq \ep  } \overline{\phi}(s)  \ok(s ) \oT(s  ') \cdot \oT( s ) \frac{( \oga(s) -  \oga(s'    ))\cdot \oT(s)  }{ |\oga(s) -  \oga(s'   )|^{2+2\alpha } }   \, ds  \, ds ' \\
I_3(\ep) &  =  \int \int_{|   s -s' | \geq \ep  } \overline{\phi}(s)  \oT(s  ') \cdot \oN( s ) \frac{( \oT(s) -  \oT(s'    ))\cdot \oT(s)  }{ |\oga(s) -  \oga(s'   )|^{2+2\alpha } }   \, ds  \, ds '  \\
I_4(\ep) &  =  \int \int_{|   s -s' | \geq \ep  } - \overline{\phi}(s) \ok(s    ) \oT(s  ') \oN( s ) \frac{( \oga(s) -  \oga(s'    ))\cdot \oN(s)  }{ |\oga(s) -  \oga(s'   )|^{2+2\alpha } }   \, ds  \, ds '  \\
I_5(\ep)  & =  - (2+2\alpha )\int \int_{|   s -s' | \geq \ep  }  \overline{\phi}(s)  \overline{\T}( s'  )\oN( s )  \frac{( \oga(s) -  \oga(s'  ))\cdot \oT(s)  }{ |\oga(s) -  \oga(s'  )|^{4+2\alpha } }  \nonumber \\
&\qquad  \qquad \qquad  \qquad  \qquad  \times ( \oga(s) -  \oga(s'  ))\cdot (\oT(s) - \oT(s' ) ) \, ds  \, ds ' .
\end{align}

In the steps below, we will show each term on the right-hand side of \eqref{eq:aux_dispersive_term_2} is compatible with the thesis \eqref{eq:lemma_dispersive_term_0}.

\noindent
\textbf{\underline{Analysis of $I_1$:}}

We further split $I_1 = I_{11} + I_{12}$ with
\begin{align}
I_{11}(\ep) :   & =  -\int \int_{|   s -s' | \geq \ep  } (\overline{\phi}(s)  - \overline{\phi}(s' )) \ok(s') \oN(s  ') \cdot \oN( s ) \nonumber \\
& \qquad \qquad \qquad \qquad\times \frac{( \oga(s) -  \oga(s'    ))\cdot \oT(s)  }{ |\oga(s) -  \oga(s'   )|^{2+2\alpha } }   \, ds  \, ds' \label{eq:aux_curvature_kernal_I_1_1} \\
 I_{12}(\ep) :   & =  - \int \int_{|   s -s' | \geq \ep  } \overline{\phi}(s ') \ok(s') \oN(s  ') \cdot \oN( s )  \nonumber \\
& \qquad \qquad \qquad \qquad\times \frac{( \oga(s) -  \oga(s'    ))\cdot \oT(s)  }{ |\oga(s) -  \oga(s'   )|^{2+2\alpha } }   \, ds  \, ds'  .\label{eq:aux_curvature_kernal_I_1_2}
\end{align}

\noindent
\textbf{\underline{Analysis of $I_{11}$:}}

Let us first analyze $I_{11}$. We switch back to the Lagrangian labels via $s \mapsto x =\ell^{-1}(s)$ and $s' \mapsto y = \ell^{-1}(s')$ in the integral, obtaining
\begin{align*}
I_{11}(\ep) :   & =  -\int_{\TT} \int_{|   \ell (x)  -\ell (y) | \geq \ep  } ( {\phi}(x )  -  {\phi}( y )) \k(y ) \N( y) \cdot \N( x )\\
& \qquad\qquad \qquad \times  \frac{( \g(x ) -  \g(y    ))\cdot \T(x)  }{ |\g(x) -  \g( y   )|^{2+2\alpha } } g(x) g(y)  \, dx  \, d y  .
\end{align*}
Let us first consider the variant $ \widetilde{I}_{11}(\ep)$
\begin{equation}\label{eq:aux_curvature_kernal_I_1_3}
\begin{aligned}
\widetilde{I}_{11}(\ep) :   & =  -\int_{\TT} \int_{ g(y)|    x   - y  | \geq \ep  } ( {\phi}(x )  -  {\phi}( y )) \k(y ) \N( y) \cdot \N( x )\\
& \qquad\qquad \qquad \times  \frac{( \g(x ) -  \g(y    ))\cdot \T(x)  }{ |\g(x) -  \g( y   )|^{2+2\alpha } } g(x) g(y)  \, dx  \, d y  .
\end{aligned}
\end{equation}
We claim $\lim_{\ep \to 0^+}  \widetilde{I}_{11}(\ep) = \lim_{\ep \to 0^+} {I}_{11}(\ep)  $ where both limit clearly exist due to $\phi \in C^\infty(\TT)$. Since $x \mapsto \ell  (x)$ is a $C^{1,\sigma}$ diffeomorphism, the fundamental theorem of calculus implies that for each $y\in \TT$, the set $S_\ep(y)$ of the symmetric difference
\begin{equation}
 S_\ep(y): =  \{x \in \TT : | \ell(x) -\ell(y) | \geq \ep \} \Delta \{x \in \TT : g(y) |   x  - y  | \geq \ep \}
\end{equation}
has Lebesgue measure $|S_\ep| \lesssim \ep^{\sigma} $ and hence
\begin{equation}
\left|  \widetilde{I}_{11}( \ep) -  {I}_{11}( \ep)  \right| \lesssim  \int_{\TT} \int_{S_\ep(y) } | {\phi}(x )  -  {\phi}( y )| |\k(y )|  |x-y|^{-1-2\alpha }  \, dx  \, d y
\end{equation}
Since $\phi $ is smooth, we have that  $\left|  \widetilde{I}_{11}( \ep) -  {I}_{11}( \ep)  \right| \lesssim |\nabla  \phi |_{L^\infty }  \int \int_{S_\ep(y) }   |\k(y )|  |x-y|^{ -2\alpha }  \, dx  \, d y $ and  hence by H\"older's inequality
(due to $p>\frac{1}{1-2\alpha},$ $\k \in L^p$)
together with the fact $|S_\ep| \to 0 $ implies $\left|  \widetilde{I}_{11}( \ep) -  {I}_{11}( \ep)  \right| \to 0 $   as $\ep \to 0$.

So now we focus on \eqref{eq:aux_curvature_kernal_I_1_3}. Since the flow $\g \in C_t C^{1,\sigma}$, by Lemma \ref{lemma:curve_estimates} we have
(here we use $\N(y) \cdot \N(x) = \T(y) \cdot \T(x)$)
\begin{equation}\label{eq:aux_curvature_kernal_2}
\begin{aligned}
&{\T}( y   ) \cdot {\T}(x ) \frac{( \g (x ) -  \g ( y   ))\cdot \T(x  )  }{ |\g ( x ) -  \g ( y   )|^{2+2\alpha } } g(x ) g( y  ) \\
&   = \frac{1 }{ |g( y )|^{1- 2\alpha}}  \frac{ x - y }{ | x -y |^{2+2\alpha}}    + O(C_\gamma  |x  - y |^{-1-2\alpha + \sigma }).
\end{aligned}
\end{equation}

It follows from \eqref{eq:aux_curvature_kernal_I_1_3} and \eqref{eq:aux_curvature_kernal_2} that
\begin{equation}
\begin{aligned}\label{eq:aux_curvature_kernal_3}
 \widetilde{I}_{11}( \ep)  & = -\int_{\TT} \k( y )   |g( y )|^{1-  2\alpha}  \int_{g(y) | x  - y| \geq \ep  }    ({\phi}(x)  - {\phi}(y)  )      \frac{ x  - y }{ | x  - y|^{2+2\alpha}}  \,  d x   \,  d y    \\
& \qquad + O(C_\gamma \int_{\TT} \int_{\TT} |  {\phi}( x ) -   {\phi}( y  ) | | \k( y  )| |x  - y  |^{-1-2\alpha + \sigma }    \,   d x  \, d y )   .
\end{aligned}
\end{equation}
The first term in \eqref{eq:aux_curvature_kernal_3} is the main term, corresponding to the first term on the right-hand side of \eqref{eq:lemma_dispersive_term_0}. Indeed, by oddness of the kernel and Lemma \ref{lemma:local_formula_La}, there holds
\begin{align}\label{eq:aux_curvature_kernal_4}
& \lim_{\ep \to 0^+ } \int_{g(y) | x  - y| \geq \ep  }    ({\phi}(x)  - {\phi}(y)  )      \frac{ x  - y }{ | x  - y|^{2+2\alpha}}  \,  d x  \nonumber  \\
&  = P.V. \int_{\TT }     {\phi}(x)       \frac{ x  - y }{ | x  - y|^{2+2\alpha}}  \,  d x  \nonumber \\
& = c_\alpha \mathcal{L}_\alpha \phi(x) + O(|\phi|_{L^1}  ),
\end{align}
where the error term comes from extending integration to $\RR$.
Therefore,  by \eqref{eq:aux_curvature_kernal_4} and \eqref{eq:aux_curvature_kernal_3} we have
\begin{equation}
\begin{aligned}\label{eq:aux_curvature_kernal_5}
&\lim_{\ep \to 0^+}-\int_{\TT} \k( y )   |g( y )|^{1-  2\alpha}  \int_{g(y) | x  - y| \geq \ep  }    ({\phi}(x)  - {\phi}(y)  )      \frac{ x  - y }{ | x  - y|^{2+2\alpha}}  \,  d x   \,  d y\\
& =  c_\alpha \int_{TT} \mathcal{L}_\alpha \phi( y )
  \k( y ) |g( y )|^{1-  2\alpha}       \, d y  +    O( |\k|_{L^p} |\phi|_{L^1}  ).
\end{aligned}
\end{equation}

Now we focus on the error term in \eqref{eq:aux_curvature_kernal_3}. We first rewrite it by a translation  and then use H\"older's inequality in $\eta$ to obtain that
\begin{equation}\label{eq:aux_curvature_kernal_6}
\begin{aligned}
& \int_{\TT} \int_{\TT} |  {\phi}( x   ) -   {\phi}(  y   ) | | \k( y )| | x  - y |^{-1-2\alpha + \sigma }    \,   dx  dy\\
& =  \int_{\TT} \int_{\TT}  \frac{|{\phi}( x +  y ) -   {\phi}( y ) |  }{ |x    |^{ 1+2\alpha -  \sigma }} | \k(  y )|     \,   d y  d x\\
& \leq  |\k |_{L^p }    \int_{\TT} \Big[ \int_{\TT} \Big[ \frac{| {\phi}( x + y  ) -   {\phi}( y ) |  }{ | x   |^{ 1+2\alpha -  \sigma }}    \Big]^{p'}   \,   d y \Big]^\frac{1}{p' }  d x  .
\end{aligned}
\end{equation}

Since $ \phi \in C^\infty(\TT)$, by Lemma \ref{lemma:modulus_Besov} that appears in the next section we have
\begin{equation}\label{eq:aux_S4_phi_sobolev}
\int \Big[ \int \Big[ \frac{| {\phi}( x + y  ) -   {\phi}( y ) |  }{ | x   |^{ 1+2\alpha -  \sigma }}    \Big]^{p'}   \,   d y \Big]^\frac{1}{p' }  d x \leq  C |\phi |_{B^{2\alpha -\sigma}_{ p', 1} }.
\end{equation}
So the error term in \eqref{eq:aux_curvature_kernal_3} satisfies
\begin{align}\label{eq:aux_curvature_kernal_7}
  \int \int |  {\phi}(x ) -   {\phi}( y ) | | \k( y )|   | x - y|^{-1-2\alpha + \sigma }    \,   dx   d y   \lesssim  |\k |_{L^p } |\phi |_{B^{2\alpha -\sigma}_{p', 1 }} .
\end{align}

Finally putting together \eqref{eq:aux_curvature_kernal_3}, \eqref{eq:aux_curvature_kernal_5}, and \eqref{eq:aux_curvature_kernal_7}, we have
\begin{align*}
\lim_{\ep \to 0^+}\widetilde{I}_{11}( \ep)   & = -c_\alpha \int \mathcal{L}_\alpha \phi( y )
  \k( y )   [g( y)]  ^{1-2\alpha}       \, d  y   \\
& \qquad \qquad + O(C_{\g,\alpha } |\k |_{L^p}  |\phi |_{B^{2\alpha -\sigma}_{p', 1 }}  )    .
\end{align*}

\noindent
\textbf{\underline{Analysis of $I_{12}$:}}

By the estimates \eqref{eq:arc_length_a}, \eqref{eq:arc_length_f}, and \eqref{eq:arc_length_linear} in Section \ref{subsec:curves_estimates}, for any $s,s \in \RR$
\begin{equation}
\begin{aligned}
 \oN (s) \cdot \oN(s' ) &= 1  + O(| s  - s' |^{2(1 - \frac{1}{p})  } ) \\
 (\oga(s      )-\oga(  s' )) \cdot \oT(s ) &= (s  - s' ) + O(|s  - s'|^{ 3- \frac{2}{p} } ) \\
\frac{|s  - s' |^{2  +2\alpha } }{|\oga(s)-\oga(  s')|^{2  +2\alpha }}   &  =1 + O(|s  - s'|^{ 2 (1 - \frac{1}{p})  }  ) .
\end{aligned}
\end{equation}
It follows that
\begin{equation}\label{eq:aux_curvature_kernal_I_12}
\begin{aligned}
 I_{12}(\ep)   =  - & \int \int_{|   s -s' | \geq \ep  } \overline{\phi}(s ') \ok(s')    \bigg( \frac{ s  -   s'     }{ | s  -   s'   |^{2+2\alpha } }   \\
& \qquad \qquad     + O(|s-s'|^{1-2\alpha -\frac{2}{p}})      \bigg)    \, ds  \, ds'
\end{aligned}
\end{equation}
By the oddness of the kernel, the first term in \eqref{eq:aux_curvature_kernal_I_12} vanishes, and it follows that
\begin{align}
| I_{12}(\ep)|  \lesssim  - \int \int  |\overline{\phi}(s ')| |\ok(s')|  |s-s'|^{1-2\alpha -\frac{2}{p}})          \, ds  \, ds'
\end{align}
Since $1-2\alpha - \frac{2}{p} > -1$ by our assumptions, by Young's inequality we have $
| I_{12}(\ep)|  \lesssim  |  \overline{\phi} |_{L^{p'}} |\ok |_{L^p}$. Since the flow is $C^{1,\sigma}$, we have $| \overline{\k}|_{L^p } \lesssim |  {\k}|_{L^p }$ and $| \overline{\phi} |_{L^{p'} }  \lesssim | {\phi} |_{L^{p'} }$. So in the Lagrangian label, we have
\begin{align*}
|I_{12} (\ep)| \lesssim |  {\k}|_{L^p } |  {\phi} |_{L^{p'} }  ,
\end{align*}
which is compatible with the claimed estimate.

\noindent
\textbf{\underline{Analysis of $I_2$:}}

Recall that we need to estimate
\begin{align}\label{eq:aux_curvature_kernal_8}
I_2(\ep) &  =  \int \int_{|   s  - s' | \geq \ep  }  \overline{\phi}(s) \ok(s    ) \oT(s ' ) \oT( s ) \frac{( \oga(s) -  \oga(s'    ))\cdot \oT(s)  }{ |\oga(s) -  \oga(s'   )|^{2+2\alpha } }   \, ds  \, ds' .
\end{align}
Observe that by Fubini and Lemma \ref{lemma:estimate_v_on_boundary}
\begin{equation}\label{eq:aux_curvature_kernal_9}
\begin{aligned}
& \lim_{\ep \to 0^+} \int  \overline{\phi}(s) \ok(s    )  \int_{| s- s'| \geq \ep  }  \oT(s ' ) \oT( s ) \frac{( \oga(s) -  \oga(s'    ))\cdot \oT(s)  }{ |\oga(s) -  \oga(s'   )|^{2+2\alpha } }   \,  ds' \, ds \\
&  = \int  \overline{\phi}(s)\ok(s    ) \p_s \overline{v} \cdot \oT (s ) \, ds  .
\end{aligned}
\end{equation}
Estimating in the arc-length variable by H\"older's inequality, we have
\begin{equation}\label{eq:aux_curvature_kernal_10}
|I_2(\ep)| \leq | \overline{\k}|_{L^p } | \overline{\phi} |_{L^{p'} } | \p_s \overline{v} |_{L^\infty }.
\end{equation}
It follows that  in the Lagrangian label
\begin{align*}
|I_2(\ep)| \lesssim |  {\k}|_{L^p } |  {\phi} |_{L^{p'} }  ,
\end{align*}
which is compatible with the claimed estimate.

\noindent
\textbf{\underline{Analysis of $I_3$:}}

Recall that we need to estimate
\begin{align}\label{eq:aux_curvature_kernal_11}
I_3(\ep) &  =  \int \int_{|  s  -    s' | \geq \ep  } \overline{\phi}(s)  \oT(s  ') \cdot \oN( s ) \frac{( \oT(s) -  \oT(s'    ))\cdot \oT(s)  }{ |\oga(s) -  \oga(s'   )|^{2+2\alpha } }   \, ds  \, ds ' .
\end{align}
We apply \eqref{eq:arc_length_e} and \eqref{eq:arc_length_M_a} to find
\begin{align}\label{eq:aux_curvature_kernal_12}
| I_3(\ep) |  &  \leq C  \int \int  |\overline{\phi}(s)|    \mathcal{M}\overline{\k} (s') |s-s'|^{1+2(1 - \frac{1}{p}) -2 -2\alpha}   \, ds  \, ds '  .
\end{align}
The assumption on $p$ implies that $ 1 - \frac{2}{p}) -2\alpha > -1 $, so we obtain by Young's inequality that
\begin{align}\label{eq:aux_curvature_kernal_13}
| I_3(\ep) |  &  \leq C    |\overline{\phi} |_{L^{p'} }    | \mathcal{M}\overline{\k}   |_{L^{p } }.
\end{align}
From the $L^p$ boundedness of the maximal function with $p>1$ it follows that $ | I_3(\ep) |\leq  C    |\overline{\phi} |_{L^{p'} }    |  \overline{\k}  |_{L^{p } } $. Once again, the $C^{1,
\sigma}$ regularity of the flow allows us to conclude
$$
 | I_3(\ep) | \leq C    | {\phi} |_{L^{p'} }    |   {\k}  |_{L^{p } }  .
$$

\noindent
\textbf{\underline{Analysis of $I_4$:}}

Recall that we need to estimate
\begin{align}\label{eq:aux_curvature_kernal_14}
I_4(\ep) &  =  \int \int_{|  s  - s' | \geq \ep  } - \overline{\phi}(s) \ok(s    ) \oT(s  ') \cdot \oN( s ) \frac{( \oga(s) -  \oga(s'    ))\cdot \oN(s)  }{ |\oga(s) -  \oga(s'   )|^{2+2\alpha } }   \, ds  \, ds ' .
\end{align}

By \eqref{eq:arc_length_d} and \eqref{eq:arc_length_M_a},
\begin{align*}
|I_4(\ep)| &  \leq C     \int \int  |\overline{\phi}(s)|    \mathcal{M}\overline{\k} (s') |s-s'|^{1+2(1 - \frac{1}{p}) -2 -2\alpha}   \, ds  \, ds ' .
\end{align*}
So $ I_4$ satisfies the same bound as $I_3$

\noindent
\textbf{\underline{Analysis of $I_5$:}}

Recall that we need to estimate
\begin{equation}\label{eq:aux_curvature_kernal_15}
\begin{aligned}
I_5(\ep)   =  - (2+2\alpha )\int& \int_{| s -    s' | \geq \ep  } \overline{\phi}(s)  \overline{\T}( s'  )\oN( s )  \frac{( \oga(s) -  \oga(s'  ))\cdot \oT(s)  }{ |\oga(s) -  \oga(s'  )|^{4+2\alpha } }  \\
&\qquad \times ( \oga(s) -  \oga(s'  ))\cdot (\oT(s) - \oT(s' ) ) \, ds  \, ds ' .
\end{aligned}
\end{equation}

By \eqref{NT1} and \eqref{eq:arc_length_M_e},
\begin{align*}
|I_5(\ep) | &  \leq C \int \int | \overline{\phi}(s) | \mathcal{M}\overline{\k} (s')  |s -s'|^{\beta -3-2\alpha +2+\beta }    \, ds  \, ds '\\
 &  \leq C       \int \int  |\overline{\phi}(s)|    \mathcal{M}\overline{\k} (s') |s-s'|^{1 - \frac{2}{p}) -2\alpha}   \, ds  \, ds '  .
\end{align*}

So $ I_5$ also satisfies the same bound as $I_3.$

\end{proof}

\subsection{Leading order dynamics of the curvature flow}

We put together results in this section to obtain the following characterization of the curvature flow for the $\alpha$-patch problem.
\begin{theorem}\label{prop:weak_k_flow}
Let $0 < \alpha < \frac{1}{2}$ and $p > \frac{1}{1 -2\alpha}$. Let $\Omega_t$ be a $W^{2,p}$ $\alpha$-patch on $[0,T]$. Suppose $ \g \in C([0,T];C^{1,\sigma}) $ with $\sigma=1-\frac{1}{p}-2\alpha,$ is the flow of $\p \Omega_t$  with the initial data $\g_0$. Let $v$ be the velocity field generated by the patch $\Omega_t$ parameterized by $\g$
 and  $g = |\p_x \g |$ be the metric associated with the flow $\g$.

Then the curvature flow $\k$ satisfies  for all $ \varphi \in C^\infty(\TT \times \RR )$
\begin{equation}
\begin{aligned}
&\int_{\TT} \k(\cdot, T) \phi(\cdot , T)\, g(\cdot,T ) \,dx -\int_{\TT} \k(\cdot, 0) \phi(\cdot ,0) g(\cdot,0)\, dx \\
&=    \int_0^T \int_{\TT} g \k \p_t \phi \,dx dt  - c_\alpha \int_0^T \int_{\TT} g^{1-2\alpha} \k \mathcal{L}_\alpha \phi \, dx dt \\
&\qquad \qquad + \int_0^T \mathcal{E}_t(\phi(\cdot, t))  \, dt
\end{aligned}
\end{equation}
where $c_\alpha>0$ is a universal constant and $\mathcal{E}_t:C^\infty(\TT   ) \to \RR$ is a distribution satisfying the bound
\begin{equation}
|\mathcal{E}_t (  \phi )| \leq C_{\alpha,\gamma} |\k|_{L^p} |\phi |_{B^{2\alpha -\sigma}_{p', 1 }}  \quad \text{for any $\phi \in C^\infty(\TT) $}
\end{equation}
uniformly in time where     $p'$ is the H\"older dual of $p$ and the constant $C_{\alpha,\gamma}$ depends on $\alpha$, $p$,   $\Om$, and $|\g|_{C^{1,\sigma}(\TT)}$.
\end{theorem}

\begin{proof}
By Proposition \ref{prop:kappa_existence}, the existence of the curvature flow $\k$ is guaranteed. Now thanks to Lemma \ref{lemma:kappa_better_formulation}, $\k$ satisfies for all $\phi \in C^\infty ( \TT \times \RR )$,
\begin{equation}
\begin{aligned}
&\int_{\TT} \k(\cdot, T) \phi(\cdot , T)\, g(\cdot,T ) \,dx - \int_{\TT} \k(\cdot, 0) \phi(\cdot ,0) g(\cdot,0)\, dx    \\
&=    \int_0^T \int_{\TT} \k   \p_t \phi  \,g dx dt  + \int_0^T \int_{\TT }     \p_s v\cdot \N \p_s \phi  \,g dx  dt .
\end{aligned}
\end{equation}

We then define the distribution $\mathcal{E}$ by
\begin{equation}
\mathcal{E}_t (\phi) : =  \int_{\TT}\p_s v\cdot \N \p_s \phi  \,g dx   +c_\alpha \int_{\TT}g^{1-2\alpha}  \k \mathcal{L}_\alpha \phi \, dx.
\end{equation}
The conclusion follows directly from Proposition \ref{prop:dispersive_term}.

\end{proof}

\section{Dispersive estimates for \texorpdfstring{$e^{t\mathcal{L}_\alpha}$}{eLt}}\label{sec:disper_estimates}

In this section, we derive useful dispersive estimates for the evolution group $e^{t\mathcal{L}_\alpha}$ generated by $\mathcal{L}_\alpha$ on $\TT$.

The results in this and the next sections are independent of the rest of the paper, with a focus on functional spaces and distributions.

As before, we fix $0<\alpha < \frac{1}{2}$ and identify the torus $\TT = (- \pi,\pi]$.

We first consider the Fourier multiplier $ e^{i \xi |\xi|^{2\alpha -1} t}$ on the whole line $\RR$. Later we will use the Poisson summation formula to transfer estimates to the periodic setting.

\subsection{Littlewood-Paley decomposition}

In this subsection, we introduce a Littlewood-Paley decomposition on $\TT$. This is needed for results in this and the next sections.

Throughout the paper, $ \mathcal{D}'$ denotes  the dual of $C^\infty(\TT)$  as the set of periodic distributions  and $\lambda_q := 2^q$ for any $q\in \RR$.

Given $f : \TT \to \RR$, its Fourier series coefficients are defined by
\begin{equation}
	\widehat{f}(k) = \int_{- \pi }^{\pi } f(x) e^{-ikx} \, dx \quad \text{for any $k\in \ZZ$}.
\end{equation}
The Fourier transform on $\RR$ is defined similarly.

\begin{definition}\label{def:LP}
Let even $\psi \in C^\infty_c(\RR) $ be such that $\psi =1 $ on $B_{\frac{3}{4}} (0)$ and $ \Supp \psi \subset B_{ 1 }(0)$. For any integer $q \geq -1$, define
\begin{equation}\label{aux57a}
\psi_q (\xi  ) =
\begin{cases}
 \psi( \xi \l_{q +1 }^{-1} )   - \psi( \xi   \l_{q }^{-1} )& \text{if $q \geq 0$} \\
\psi( \xi \l_{q+1 }^{-1} )  &\text{if $q = -1 $}
\end{cases}
\end{equation}
such that $\sum_{ q \geq -1} \psi_q( \xi ) =1 $.

 For each $f \in  \mathcal{D}'$, the Littlewood-Paley decomposition of $f$ is
$$
f = \sum_{q \geq -1} \Delta_q f, \quad \text{with $ \Delta_q f  : = \sum_{k\in \NN} \widehat{f}(k)   \psi_q(k)  e^{i k x}  $},
$$
where the equality holds in the sense of distributions.
\end{definition}
We will use the following facts:
\begin{align*}
\psi_q(\xi) &= \psi_0 (\l_q^{-1} \xi ) \quad \text{for any $q \in \NN^+$}\\
\Supp \psi_q & \subset \{ \xi \in \RR :  \frac{1}{2} \l_q  \leq | \xi | \leq  2\l_q  \}
\end{align*}

Later in the paper, we also use $\widetilde{\Delta}_q : =\Delta_{q-1 } +\Delta_{q   }+\Delta_{q+1 }$ which satisfies $ \widetilde{\Delta}_q  \Delta_{q   } = \Delta_{q   } $.

We will also use the Littlewood-Paley decomposition on $\RR$. The same set of cutoffs as in Definition \ref{def:LP} will also be used to define such decompositions on $\RR$. In particular if $f  \in L^1(\RR)$ then its $2\pi$-periodization $f^{per}: = \sum_{n\in \ZZ} f(x + 2\pi n)$ is a function on $\TT$ and it satisfies
\begin{equation}
\sum_{n\in \ZZ} \Delta_q f(x + 2\pi n ) = \Delta_q f^{per}(x),
\end{equation}
where the Littlewood-Paley projection $\Delta_q$ is acting in the $\RR$ setting on the left-hand side and in the $\TT$ setting on the right-hand side.

To streamline the proof, we use Besov spaces a few times and we recall their definitions below.
\begin{definition}\label{def:BesovLP}
For any $s\in \RR$, $1 \leq p,q \leq \infty $, the Besov space $B^{s}_{p,q}(\TT)$ is defined by
$$
|f|_{B^{s}_{p,q}} := \left| \lambda_q^s |\Delta_q f|_{L^p} \right|_{\ell^{q}(\NN)}, \quad \text{and} \quad  B^{s}_{p,q}(\TT): =\{ f\in \mathcal{D}' :|f|_{B^{s}_{p,q}}< \infty  \}
$$
where $\ell^{q}(\NN)$ denotes the $\ell^q$ norm for sequences.
\end{definition}

The following  Bernstein's inequality is well-known.

\begin{lemma}[Bernstein's inequality]\label{lemma:bernstein}
Let $1 \leq p \leq r \leq \infty$. For any $f \in \mathcal{D}'$ one has
\begin{equation}
|\Delta_q f |_{L^r(\TT) } \lesssim  \l_q^{\frac{1}{p}  - \frac{1}{r} }|\Delta_q f |_{L^p(\TT)}
\end{equation}
and  for any $0<\alpha < \frac{1}{2}$ the linear operator $\mathcal{L}_\alpha$ defined by \eqref{eq:def_L_a} satisfies
\begin{equation}
| \mathcal{L}_\alpha \Delta_q f   |_{L^p (\TT)} \lesssim  \l_q^{2\alpha}|\Delta_q f |_{L^p(\TT)}
\end{equation}
\end{lemma}
\begin{proof}
The first statement is classical and follows directly from Young's inequality. The second statement is surely known but we are not aware of a convenient reference.
In the whole space case, the bound follows from a scaling argument.

In the periodic case we first write $ \Delta_q f = \Delta_q f * \varphi_q$ where $\varphi_q \in C^\infty (\TT) $ is the convolution kernel for $\widetilde\Delta_q $. Then from Young's inequality again we have
$$
| \mathcal{L}_\alpha \Delta_q f   |_{L^p }  \leq | \mathcal{L}_\alpha \varphi_q   |_{L^1 }  |   \Delta_q f   |_{L^p } .
$$
We can use the Poissson summation to write
\begin{equation}\label{eq:aux_bernstein}
\begin{aligned}
 \mathcal{L}_\alpha \varphi_q & =\sum_{k \in \ZZ } \widehat{\varphi_q}(k) i|k|^{2\alpha -1 } ke^{ikx}  \\
&=\sum_{k \in \ZZ } \widehat{\varphi_0}(\l_q^{-1} k) i|k|^{2\alpha -1 } ke^{ikx} \\
& = \sum_{k \in \ZZ }\widetilde \psi_q (x+ 2\pi k)
\end{aligned}
\end{equation}
where for any $q \in \NN$ the non-periodic $\widetilde \psi_q :\RR  \to \RR $ is a Schwartz function and has Fourier transform $  \widehat{\varphi_0}(\l_q^{-1} \xi ) i|\xi |^{2\alpha -1 } \xi $.
Since we are on $\RR$ now, a standard scaling argument shows that $| \widetilde \psi_q(x)| \lesssim \l_q^{1+2\alpha} \langle \l_q x \rangle^{-100}$ where $\langle  x \rangle : = \sqrt{x^2 +1 } $ is the Japanese bracket.  Plugging this bound and summing up in \eqref{eq:aux_bernstein} give the desired estimate.
\end{proof}

\subsection{Moduli of smoothness of some functions}

In our illposedness proof, we need to consider the modulus of continuity of functions in various settings.

The next lemma concerning the moduli of smoothness of Besov functions is well-known, cf. \cite[Chapter 17]{MR3726909}. A proof is included for readers' convenience.
\begin{lemma}\label{lemma:modulus_Besov}
For any $ 0 < s < 1 $, $1 \leq  p \leq \infty$, there exist a constant $C(s,p)$ such that for any $ f \in B^{s}_{p,1}(\TT)$,
$$
 \int_{0}^1  \frac{1 }{h^{ 1+ s }}  \left[ \int_{\TT} | f(x +h) - f(x)|^p \,dx  \right]^\frac{1}{p }\,dh \leq C     |f |_{B^{s}_{p,1}} .
$$
\end{lemma}
\begin{proof}
It suffices to prove $\int_{0}^1  \frac{1 }{h^{1+ s }} | \Delta_q f(\cdot + h ) -\Delta_q f(\cdot  ) |_{L^p}  \,dh \leq C    \lambda_q^s  | \Delta_q f |_{L^p } $.

We consider the split of the integral into two regions $ 0< h< \l_q^{-1}$ and $ \l_q^{-1}   < h< 1 $. For the region $ 0< h< \l_q^{-1}$, we have $| \Delta_ q f(x +h) -  \Delta_q f(x)| \leq \int_0^h | \nabla  \Delta_q f (x+ t )| \,dt  $ by the mean value theorem, and so by the Minkowski inequality,
\begin{align*}
 | \Delta_q f(\cdot + h ) -\Delta_q f(\cdot  ) |_{L^p} \leq h | \nabla \Delta_q f  |_{L^p} .
\end{align*}
From this estimate as well as the bound $ | \nabla  \Delta_q f  |_{L^p}  \lesssim \l_q  | \Delta_q f  |_{L^p}  $, it follows that
\begin{align*}
 &\int_{0<h<\l_q^{-1} }  \frac{1 }{h^{1+ s }}    | \Delta_q f(\cdot + h ) -\Delta_q f(\cdot  ) |_{L^p} \,dh
  \\
& \leq  \int_{0<h<\l_q^{-1} }  \frac{1 }{h^{1+ s }}  \left[ \int_{\TT} \Big|   \int_0^h |\nabla \Delta_q f(x+ t )| \,dt \Big|^p \,dx  \right]^\frac{1}{p }\,dh\\
& \leq C     |\Delta_q f |_{L^p } \l_q^{s }.
\end{align*}
For the region $ h> \l_q^{-1}$, we have
\begin{align*}
	\int_{ \l_q^{-1} <h < 1  }  \frac{1 }{h^{1+ s }}  | \Delta_q f(\cdot + h ) -\Delta_q f(\cdot  ) |_{L^p}  \,dh & \leq C     |\Delta_q f |_{L^p } \int_{ \l_q^{-1} <h < 1  }  h^{-1 - s } \,dh \\
& \leq C \lambda_q^{s} |\Delta_q f |_{L^p }.
\end{align*}

\end{proof}

Let us also recall that the classical H\"older spaces $C^{k,\beta}(\TT)$ can be characterized as
\begin{equation}\label{eq:holder_LP_condition}
f \in C^{k,\beta }(\TT) \Leftrightarrow f  \in B^{k+\beta}_{\infty,\infty} \Leftrightarrow \,\, \sup_q \l_q^{k+ \beta} |\Delta_q f|_\infty < \infty
\end{equation}
when $0< \beta <1$.

\subsection{The multiplier \texorpdfstring{$e^{i |\xi |^{2\alpha -1} \xi t}$}{} on \texorpdfstring{$\RR$}{R}}

We first derive estimates of oscillatory integrals related to the Littlewood-Paley decomposition of the multiplier $e^{i |\xi |^{2\alpha -1} \xi t}$.

\begin{lemma}[Stationary phase]\label{lemma:stationary_phase}
Let $ 0 < \alpha < \frac{1}{2}$. Let $\psi \in C^\infty_c(\RR)$ be even and such that $\Supp \psi \subset \{ \frac{1}{2} \leq |\xi | \leq 2\}$. Consider for $t \geq 0$ and $q \in \NN$ the oscillatory integral
\begin{equation}
I_q(x,t) =\int_{\RR} e^{i \l_q \varphi(x,\xi)} \psi(\xi) \, d \xi
\end{equation}
where the phase $ \varphi(x,\xi) : =  x \xi + \l_q^{2\alpha -1} |\xi|^{  2\alpha -1} \xi t $.

Then there exists constants $    C_\psi> c_\psi >0  $ depending only on $\psi$ and $\alpha$  such that for any $  q \in \NN$ with $\l_q^{2\alpha} t \geq 1$
\begin{equation}\label{eq:upper_stationary_phase}
|I_q (x , t)|\lesssim_{\alpha, \psi,n }
\begin{cases}
   t^{-1 /2}\l_q^{    -   \alpha }  \quad \text{for $ x \in  [- C_\psi \l_q^{ 2\alpha -1} t, -c_\psi \l_q^{ 2\alpha -1} t ] $} &\\
     \big\langle   \l_q x+ t \l_q^{ 2\alpha   }  \big\rangle^{-n}     \quad \text{otherwise} &
\end{cases}
\end{equation}
where $\langle  x \rangle : = \sqrt{x^2 +1 } $ is the Japanese bracket.

Moreover, for any $  q \in \NN$ with $\l_q^{2\alpha} t \geq 1$ and any $    x \in  [- C_\psi \l_q^{ 2\alpha -1} t, -c_\psi \l_q^{ 2\alpha -1} t ] $,
there exist constants $A_\alpha,$ $a_\alpha>0$ such that
\begin{equation}\label{eq:lower_stationary_phase}
\begin{aligned}
 |  I_q(x,t) | &\geq a_\alpha t^{-1 /2} \l_q^{  -   \alpha } \Big|\psi( \l_q^{-1} |\frac{x}{2\alpha t }|^{ \frac{1}{2\alpha - 1}})\cos(A_\alpha   t^{  \frac{ 1}{ 1 -2\alpha } }  | x |^{- \frac{2\alpha }{ 1 -2\alpha } } + \frac{\pi}{4})  \Big| \\
&\qquad   - o( \l_q^{- \alpha } t^{-\frac{1}{2}}) .
\end{aligned}
\end{equation}

\end{lemma}
\begin{proof}

By the   rescaling $y = x t^{-1} \l_q^{1 -2\alpha }$ we will consider
\begin{equation}\label{eq:aux_stationary_phase_1}
I_q(y,t) =\int_{\RR} e^{i \l_q^{2\alpha} t \varphi(y,\xi)} \psi(\xi) \, d \xi
\end{equation}
with the new phase $  \varphi(y,\xi) = y \xi +   |\xi|^{  2\alpha -1} \xi   $. This transformation allows us to get rid of the extra dependencies in the phase function.

First of all, the phase $\xi \mapsto \varphi(y ,\xi)$ is smooth on the support of $\psi(\xi)$. Concerning stationary points defined by
$$
\p_\xi \varphi= y  +  2\alpha    |\xi|^{2\alpha -1}   =0,
$$
we have two solutions $\xi_{s\pm }(y )$ given by
\begin{align}\label{eq:aux_stationary_phase_2}
\xi_{s\pm }(y ) = \pm  ( -\frac{y }{2\alpha  })^{ \frac{1}{2\alpha -1} } .
\end{align}

Since the support of $\psi$ lies in $[1/2,2],$ we can consider two cases: $ \frac{1}{4} \leq |\xi_{s\pm }( y )| \leq 4 $ and $|\xi_{s\pm }( y )| \leq  \frac{1}{4} $ or $|\xi_{s\pm }( y )| \geq  4 $. We discuss these two cases below.

\noindent
 \textbf{Case 1: $ \frac{1}{4} \leq |\xi_{s\pm }(y)| \leq 4 $.}

By \eqref{eq:aux_stationary_phase_2}, we have $ y <0$ and $| y |\sim 1 $ with the specific values depending on $\alpha$. At the stationary points, we have the bound
\begin{align*}
 |\p_\xi^2 \varphi |  \sim    1 ,
\end{align*}
where the  constant depends on $\alpha$, and hence
\begin{align}\label{eq:aux_stationary_phase_3}
  \varphi(y ,\xi_{s\pm } ) &=  \pm A_\alpha   | y |^{- \frac{2\alpha }{ 1 -2\alpha }}  .
\end{align}
We substitute these estimates in the classical stationary phase formula (see for instance~\cite[Proposition 3, p.334]{MR1232192}):
\begin{align*}
I_q(y  ,t ) & =   \sum_{\pm}  \psi(\xi_{s\pm } ) e^{i \l_q^{2\alpha }t \varphi( y , \xi_{s\pm } )  \pm i  \frac{\pi}{4 } }\Big( \frac{1}{ \l_q^{2\alpha }t |\p_\xi^2 \varphi |} \Big)^{\frac{1}{2}} + o(\l_q^{ - \alpha }t^{ -\frac{1}{2}} )      \\
& =   2\Re \Big (   \psi(\xi_{s+ } ) e^{i \l_q^{2\alpha }t \varphi(y, \xi_{s+ } )  + i \frac{\pi}{4 } } \Big( \frac{1}{ \l_q^{2\alpha }t  |\p_\xi^2 \varphi |} \Big)^{\frac{1}{2}} \Big ) + o(\l_q^{ - \alpha }t^{ -\frac{1}{2}} )
\end{align*}
where $\Re $ denotes the real part of a complex number. We remark that the implicit constant in the error term depends on some higher order derivative bounds of $\psi$ and $\varphi$ which are independent of $x$, $t$, and $q$.

It follows that for some fixed $A_\alpha, a_\alpha >0$ and all $   y  \sim 1    $
\begin{equation}\label{eq:F_q_bound1}
\begin{aligned}
I_q(y  ,t )
= a_\alpha \psi(  |\frac{y}{2\alpha   }|^{ \frac{1}{2\alpha - 1}}) &t^{-1 /2}\l_q^{  -   \alpha }  \cos(A_\alpha  \l_q^{2\alpha }t  | y |^{- \frac{2\alpha }{ 1 -2\alpha }} + \frac{\pi}{4}) \\
&  \quad + o(\l_q^{ - \alpha }t^{ -\frac{1}{2}} )  .
\end{aligned}
\end{equation}

The upper bound \eqref{eq:upper_stationary_phase} and the lower bound \eqref{eq:lower_stationary_phase} immediately follow from \eqref{eq:F_q_bound1}.

\noindent
\textbf{Case 2 : $ \frac{1}{4} > |\xi_{s\pm }(y)| $ or  $|\xi_{s\pm }(y)|  > 4 $.}

In this case, the derivative of the phase does not vanish: $ \p_\xi \phi (y ,\xi )  = y +     2\alpha   |\xi|^{2\alpha -1}     \neq 0,$ and in fact $ |y +     2\alpha   |\xi|^{2\alpha -1}| \gtrsim 1 $ for all $\xi \in \Supp  \psi$ (with constants depending on $\alpha$). Let us define for $n \in \NN$ a sequence of bump functions inductively
\begin{equation}
\eta_n =
\begin{cases}
 \psi & \text{when $n =0$}\\
 \p_\xi \left(  \frac{\eta_{n-1}}{ \varphi'} \right) &\text{when $n \geq  1$}.
\end{cases}
\end{equation}

Repeated integration by parts gives
\begin{align}
I_q( y  ,t) &= \int_{\RR} (- i  \l_q^{  2 \alpha }t  )^{-1} e^{i  \l_q^{  2 \alpha }t  \varphi } \p_\xi \left(  \frac{\psi }{ \varphi'} \right) \, d \xi \nonumber \\
 & = \int_{\RR}  (-i  \l_q^{  2 \alpha }t  )^{-n} e^{i  \l_q^{  2 \alpha }t   \varphi } \eta_n d\xi .  \label{eq:aux_stationary_phase}
\end{align}

To reduce notation, let us write $z = y +       |\xi|^{2\alpha -1}     $.  We will prove by induction in $n $ the estimates
\begin{equation}\label{eq:aux_lemma:stationary_phase_1}
	| \p_\xi^k \eta_n  | \lesssim_{k,n} |z |^{-n }
\end{equation}
for all $k=0,1,2,\dots$ and all $n=0,1,2, \dots,$ on the support of $\psi$.
Here the implicit constant depends on $k,n$ and $\alpha$ but not on $t,x$ or $q$.

First, both $ |\p_\xi^k \varphi | \lesssim_k 1$ and $|  (\p_\xi \varphi)^{-1} | \lesssim |  y +      |\xi|^{2\alpha -1}   |^{-1} $ hold on $\Supp  \psi$. It follows that
 on $\Supp  \psi$
\begin{equation}
 |\p_\xi^k (\p_\xi \varphi)^{-1} | \lesssim_k
\begin{cases}
| z |^{-1}   & \quad \text{if $k =0$}\\
|z |^{-2}  & \quad \text{if $k \geq 1 $}.
\end{cases}
\end{equation}

When $ n=0$, we have $|\p_\xi^k \psi| \leq C  $ and \eqref{eq:aux_lemma:stationary_phase_1} holds. Assuming \eqref{eq:aux_lemma:stationary_phase_1} holds for some $n \in \NN$, at the level $n+1$ we have
\begin{align*}
\left| \p_\xi^k \eta_n  \right| & \leq    \left|\p_\xi^{k+1} \big( \frac{\eta_{n-1} }{\varphi' }   \big)  \right|  \\
& \lesssim_k \sum_{ 0 \leq  i \leq k+1 } \left|\p_\xi^{k+1 - i } \big(  \eta_{n-1}    \big)  \right| \left|\p_\xi^{i}  (\p_\xi \varphi )^{-1}      \right| \\
& \lesssim_{k,n} | z |^{-n+1}  \left(  |z  |^{  -1} + |z  |^{-2}  \right) \lesssim_{k,n} |z |^{-n }
\end{align*}
for all $\xi$ in the support of $\psi.$
We have thus proven \eqref{eq:aux_lemma:stationary_phase_1}.

Thus in the second case,
\begin{align*}
| I_q( y  ,t ) | &\lesssim_n (\l_q^{2\alpha } t)^{-n} \int_{\frac{1}{2} \leq |\xi | \leq 2} |y+  |\xi|^{2\alpha-1} |^{-n} \,d\xi.
\end{align*}
This translates to
\[ I_q(x,t) \lesssim_n   \langle  \frac{x}{t \l_q^{2\alpha -1 } } \rangle^{-n}  (\l_q^{2\alpha } t)^{-n} \lesssim_n \langle \l_q x + t \l_q^{2\alpha} \rangle^{-n} .    \]

\end{proof}

\subsection{Estimates for \texorpdfstring{$e^{t\mathcal{L}_\alpha}$}{} and Wainger's example}

We have the following brute-force estimate for $e^{t\mathcal{L}_\alpha}$.

\begin{lemma}\label{lemma:K_q}
Let $ 0 < \alpha < \frac{1}{2}$. Let $1 \leq p  \leq \infty$ and $ f\in L^p(\TT) $. For  any $t \in [0,1]$ and $q\in \NN$, there holds
\begin{equation}
\Big| \Delta_q   e^{t\mathcal{L}_\alpha}f   \Big|_{L^p (\TT)} \lesssim \max\{ 1,t^{ \frac{1}{2} } \lambda_q^{ \alpha} \}  | f|_{L^p (\TT) } .
\end{equation}
\end{lemma}
\begin{proof}
We first show the bound when $t^{ \frac{1}{2} } \lambda_q^{ \alpha}  \leq 1$. In this case, we have
\begin{equation}\label{eq:aux_exponentialsum}
\Delta_q   e^{t\mathcal{L}_\alpha}f  =   e^{t\mathcal{L}_\alpha}\Delta_q f = \sum_{ n\geq 0} \frac{(t\mathcal{L}_\alpha )^n }{n!} \Delta_q f .
\end{equation}
Due to the compact Fourier support of $\Delta_q f $, this equality \eqref{eq:aux_exponentialsum} can be shown by looking at the Fourier side. By Lemma \ref{lemma:bernstein}, \eqref{eq:aux_exponentialsum}, and the fact $ t\l_q^{2\alpha} \leq 1 $, it follows that
\begin{equation}
\begin{aligned}
\Big| \Delta_q   e^{t\mathcal{L}_\alpha}f  \Big|_{L^p (\TT)} &  \leq    \sum_{ n\geq 0}    \frac{|  (t\mathcal{L}_\alpha )^n \Delta_q f  |_{L^p (\TT)}}{n!}  \\
& \lesssim  |f  |_{L^p (\TT)} \sum_{ n\geq 0} \frac{   (t\l_q^{2\alpha} )^n   }{n!} \\
& \lesssim  |f  |_{L^p (\TT)} .
\end{aligned}
\end{equation}

It remains to prove the estimate when $  t\l_q^{2\alpha} \geq 1  $.
Consider $K_q:\RR \to \RR$ the (non-periodic) kernel of $\Delta_q e^{i|\xi|^{2\alpha - 1} \xi t}$ and denote by $k_q$ the periodic kernel $\Delta_q   e^{t\mathcal{L}_\alpha} $. The Poisson summation gives
\begin{equation}\label{eq:aux_poisson}
k_q(x) = \sum_{n \in \ZZ } \psi_q(n) e^{i(nx + |n|^{2\alpha -1} n t)} =  \sum_{n \in \ZZ } K_q(x + 2\pi n)
\end{equation}
where $\psi_q $ are the smooth cutoff functions as in Definition \ref{def:LP}.
So it suffices to show suitable estimates for the non-periodic $K_q$ and then take a summation.

The non-periodic kernel $K_q$ satisfies
\begin{equation}\label{eq:aux_poisson_1}
 K_q(x) = \int_{\RR} \psi_q (\xi ) e^{i(\xi x + |\xi|^{2\alpha -1} \xi t)} \,d\xi  ,
\end{equation}
and hence
\begin{equation}\label{eq:aux_poisson_2}
 K_q(x) =  \l_q \int_{\RR} \psi_0 ( \xi ) e^{i \l_q (\xi x + \l_q^{ 2\alpha -1 }|\xi|^{2\alpha -1} \xi t)} \,d\xi   ,
\end{equation}
where we note that $\Supp \psi_0 \subset \{ \frac{1}{2} \leq |\xi| \leq 2  \}$.
By Lemma \ref{lemma:stationary_phase} and $ \l_q^{2\alpha} t \geq 1$,   we have
\begin{enumerate}
\item  when $|x| \lesssim \l_q^{2\alpha -1}t $.
$$
|K_q(x)| \lesssim \l_q \l_q^{ -\alpha } t^{-\frac{1}{2}}
$$

\item when $
 \l_q^{2\alpha -1}t \lesssim |x|  $
$$
|K_q (x)| \lesssim  \l_q  \langle   \l_q x+ t \l_q^{ 2\alpha   }  \rangle^{-100}
$$

\end{enumerate}
Since $ \l_q^{1-2\alpha } t  \leq 1 $ due to $0\leq t\leq 1$, the number of intervals $ [ -\pi + 2n\pi, \pi + 2n\pi]$ where we have to use the first bound is bounded by some absolute constant. So, summation in $n$ over $\ZZ$ in \eqref{eq:aux_poisson}    gives
\begin{align*}
|k_q |_{L^1(\TT)} & \leq \sum_{n \in \ZZ  } |K_q (\cdot + 2\pi n) |_{L^1( [  -\pi,\pi ]  ) }  \\
& \lesssim  \l_q^{1-\alpha } t^{-\frac{1}{2}}  \l_q^{2\alpha -1}t + \l_q   \sum_{n \geq 1  }  \langle   \l_q n + t \l_q^{ 2\alpha   }  \rangle^{-100}   \\
& \lesssim t^{ \frac{1}{2} }  \l_q^{ \alpha}    .
\end{align*}

\end{proof}

The main building device for the proof of illposedness is the following example, which was considered by Waigner in \cite{MR182838} with $t =1$.

\begin{lemma}[Wainger's example] \label{lemma:bulding_f_pless2}
Let $0 < \alpha < \frac{1}{2}$. For any $ 1 \leq  p < 2$ and $\ep>0$, there exists a real-valued and even function $f \in L^p(\TT)$   such that the following holds.

For any  $ 0< t \leq 1 $ and any sufficiently large integer $ q  $  with $\l_q^{2\alpha} t \geq 1$,
\begin{equation}\label{eq:bulding_f_pless2}
	\Big| e^{t \mathcal{L}_\alpha } \Delta_q  f  \Big|_{L^p(\TT)} \geq   t^{ \frac{1}{ p }  - \frac{1}{2} } \lambda_q^{2\alpha(\frac{1}{ p }  - \frac{1}{2} )  - \ep}  .
\end{equation}

\end{lemma}
\begin{proof}

We first discuss this example in $\RR$ and then use the Poisson summation to transfer the result to $\TT$.

Fix an even cutoff function $\eta(\xi)$ such that $\eta(\xi)=1$ when $|\xi| \geq 2$ and $\eta(\xi) =0$ when $|\xi| \leq \frac{3}{2}$.  For small $\ep>0$, define the function $f_c :\RR \to \RR$ by its Fourier transform,
\begin{align}
\widehat{f_c}(\xi) = \eta(\xi) |\xi|^{-(1-\frac{1}{p} +\frac{\ep}{2})}.
\end{align}

Then we have the following properties for $f_c$:
\begin{itemize}
\item $f_c$ is real-valued and even.

\item $f_c  $ is continuous away from $x=0$, and $ |x^N f(x) | \to 0 $ as $x \to \infty$ for any $N\in \NN$.

\item $f_c \in L^p(\RR)  $.
\end{itemize}
The first follows directly from the definition, and the second and the third were proved in \cite{MR182838} by Wainger. These facts can also be proved by the Littewood-Paley decomposition used here.

We can use the Poisson summation to obtain the periodic example on $\TT$:
$$
  f(x): =\sum_{n \in  \ZZ} f_c(x + 2\pi n)
$$
whose Fourier series  is given by
\begin{align}
  f(x) = \sum_{n\neq 0} \eta(n) |n|^{-(1-\frac{1}{p} +\frac{\ep}{2}  )} e^{  i n x }.
\end{align}
It is clear that $f\in L^p(\TT),$ is also real-valued, and even, so we focus on the lower bound for $e^{t \mathcal{L}_\alpha } \Delta_q  f $.

Denote by $F (x)   $ the corresponding output from the multiplier on $\RR$, namely
$$
\widehat F (\xi )= e^{i |\xi|^{2\alpha -1} \xi t}   \eta(\xi) |\xi|^{-(1-\frac{1}{p} +\frac{\ep}{2})}.
$$
Consider the Littlewood-Paley projections $\Delta_q F (x)$,
and by the Poisson summation formula
$$
e^{t \mathcal{L}_\alpha } \Delta_q  f (x) =\sum_{n\in \ZZ } \Delta_q F(x +2\pi n)
$$
we can write
\begin{equation}\label{eq:aux_F_f_bounds}
\Big| e^{t \mathcal{L}_\alpha } \Delta_q  f  \Big|_{L^p(\TT)} \geq \Big|   \Delta_q  F  \Big|_{L^p([-\pi,\pi])} -\sum_{n\neq 0} \Big|   \Delta_q  F (\cdot +2\pi n) \Big|_{L^p([-\pi,\pi])} .
\end{equation}

We now derive estimates for $\Delta_q  F  $. Writing the phase $ \varphi(x,\xi) =  x \xi + \l_q^{2\alpha -1} |\xi|^{  2\alpha -1}\xi t $, a change of variable yields
\begin{align}\label{eq:aux_F_f_bounds_2}
\Delta_q F  (x) = \l_q^{\frac{1}{p} -\frac{\ep}{2} } \int_{\RR} e^{i \l_q \varphi(x,\xi)}  \widetilde{\psi}(\xi) \, d \xi
\end{align}
where $\widetilde{\psi} = \psi_0 (\xi)  | \xi|^{-(1-\frac{1}{p} +\frac{\ep}{2} ) }$ is a nontrivial bump function related to the cutoffs from the Littlewood-Paley decomposition.

As in the proof of Lemma \ref{lemma:K_q}, we can apply   Lemma \ref{lemma:stationary_phase} to \eqref{eq:aux_F_f_bounds_2}. Specifically, let $J_q = [-c_F \l_q^{ 2\alpha -1} t, -C_F \l_q^{ 2\alpha -1} t] $, $C_F >c_F >0,$ be an interval where we have
\begin{equation}\label{eq:aux_F_f_bounds_22}
|\Delta_q F (x)| \geq a_\alpha \l_q^{\frac{1}{p} -\frac{\ep}{2} } \Big( \l_q^{ -\alpha } t^{ \frac{1}{2}}  |\cos(A_\alpha   t^{  \frac{ 1}{ 1 -2\alpha } }  | x |^{- \frac{2\alpha }{ 1 -2\alpha } } + \frac{\pi}{4})|   - o( \l_q^{- \alpha } t^{-\frac{1}{2}})  \Big)
\end{equation}
For all sufficiently large $q$, $\l_q^{ 2\alpha -1} t$ can be sufficiently small by our assumption, so we may assume $J_q \subset [-\pi, \pi]$.

In this setting, we have the following outcomes of Lemma \ref{lemma:stationary_phase}.
\begin{enumerate}
\item   By \eqref{eq:aux_F_f_bounds_22} and a direct computation,
\begin{equation}\label{eq:aux_F_f_bounds_221}
| \Delta_q F  |_{L^p([-\pi, \pi ])} \gtrsim \l_q^{2\alpha (\frac{1}{p} - \frac{1}{2}) -\frac{\ep}{2} }   t^{\frac{1}{p}-\frac{1}{2}} .
\end{equation}

\item When $  x \not \in \RR\setminus [-\pi,\pi ] $,
\begin{equation}\label{eq:aux_F_f_bounds_222}
| \Delta_q F  (x)| \lesssim \l_q^{\frac{1}{p} -\frac{\ep}{2} }  \langle   \l_q x+ t \l_q^{ 2\alpha   }  \rangle^{-100}
\end{equation}

\end{enumerate}

For all sufficiently large $q$, we may use a factor of $\l_q^{- \frac{\ep}{2}}$ to absorb/enlarge the constant in \eqref{eq:aux_F_f_bounds_221} so that
\begin{equation}\label{eq:aux_F_f_bounds_3}
|\Delta_q F |_{L^p[-\pi , \pi]} \geq 2 t^{ \frac{1}{ p }  - \frac{1}{2} } \lambda_q^{2\alpha(\frac{1}{ p }  - \frac{1}{2} )  -  {\ep}  }.
\end{equation}
Similarly,  by \eqref{eq:aux_F_f_bounds_222},  for all sufficiently large $q$
\begin{equation}\label{eq:aux_F_f_bounds_4}
\sum_{n\neq 0} \Big|   \Delta_q  F (\cdot +2\pi n) \Big|_{L^p([-\pi,\pi])} \leq t^{ \frac{1}{ p }  - \frac{1}{2} } \lambda_q^{2\alpha(\frac{1}{ p }  - \frac{1}{2} )  -  {\ep}  }.
\end{equation}
The result follows by putting together \eqref{eq:aux_F_f_bounds_3} and \eqref{eq:aux_F_f_bounds_4} in \eqref{eq:aux_F_f_bounds}.

\end{proof}

\section{Illposedness of a dispersive equation}\label{sec:disper_eq_illposed}

In this section, we prove illposedness for a model equation that has the same structure as the curvature equation in Theorem \ref{prop:weak_k_flow} of the $\alpha$-patch.

\subsection{The model equation}
Recall that for any $0<\alpha< \frac{1}{2}$,  $\mathcal{L}_\alpha := \mathcal{H } (-\Delta)^{\alpha}$ denotes the composition of a Hilbert transform with a fractional Laplacian on the torus $\TT$. In other words, $\mathcal{L}_\alpha f (x) := \sum_{k\in \ZZ}  i k |k|^{2\alpha - 1} \hat{f}_k e^{i k x}$.

Let $T>0$. We consider the following distributional dispersive equation
\begin{equation}\label{eq:model_linear}
\begin{cases}
\p_t f =   \mathcal{L}_\alpha \overline{f} +  F_t & \\
f |_{t=0} = f_0 .&
\end{cases}
\end{equation}

In \eqref{eq:model_linear}, $f : \TT \times [0,T] \to \RR$ is the real-valued unknown, the mapping $ f \mapsto \overline{f}   $ is bounded in some space embedded in $L^1(\TT)$, and $F_t: C^\infty(\TT ) \to \RR$ is a given time-dependent distributional forcing.

We now specify what we mean by a solution of \eqref{eq:model_linear}.

\begin{definition}\label{def:wea_sol_model_eq}
Let $0< \alpha< \frac{1}{2}$. We say $f  \in L^\infty(0,T; L^1(\TT)) $ is a weak solution of \eqref{eq:model_linear} if all of the following are satisfied.
\begin{enumerate}
\item $ t \mapsto f(\cdot ,t) $ is continuous as a distribution on $\TT$, i.e.
\begin{equation}
t \mapsto \int_{\TT} f \varphi \, dx \quad \text{is continuous for any $\varphi \in C^\infty(\TT)$;}
\end{equation}

\item $f  $ solves \eqref{eq:model_linear} in the sense of distributions: for any $\varphi \in C^\infty(\TT \times \RR )$, there holds
\begin{equation}\label{eq:linear_weak}
\begin{aligned}
 \int_{\TT} f(\cdot, T) \varphi(\cdot, T) \, dx  -\int_{\TT} f_0  \varphi(\cdot,0) \,d x  = \int_{\TT \times [0,T]}& (f \p_t \varphi -    \overline{f} \mathcal{L}_\alpha  \varphi) \, dx dt  \\
&+ \int_0^T  F_t( \varphi(\cdot,t) )   \,dt
\end{aligned}
\end{equation}
where $F_t( \varphi) $ denotes the pairing for distributions.
\end{enumerate}

\end{definition}

\subsection{\texorpdfstring{$L^p$}{Lp} illposedness of the model equation}

Since we are mainly concerned with the $W^{2,p}$ $\alpha$-patches, we focus on the $L^p$ case. To streamline the presentation, given any $p \in [1,\infty]$ and $\delta>0$, we consider the following \textbf{$(L^p, \delta)$-assumptions}, \ref{assu:f_continuity} and \ref{assu:Forcing_bound}, on \eqref{eq:model_linear}.
\begin{figure}[!htb]
\begin{mdframed}[linewidth=1pt,skipabove=10pt,frametitle={$(L^p, \delta)$-assumptions:}]
\begin{enumerate}[label=(A{\arabic*})]

\item\label{assu:f_continuity} $ f \mapsto \overline{f}  $ is bounded in $L^p(\TT)$  with the estimate
\begin{equation}\label{eq:fbar_bound}
	|f(\cdot, t)  -\overline{f}(\cdot, t)|_{L^p(\TT)} \lesssim t^{\delta} \quad \text{uniformly in time for all $t \geq 0$ small;}
\end{equation}

\item\label{assu:Forcing_bound} the distributional forcing $F_t$ satisfies
\begin{equation}\label{eq:Forcing_bound}
|F_t(\varphi)| \lesssim |\varphi |_{B^{2\alpha -\delta}_{ p' ,1} }  \quad \text{for all $ \varphi \in C^\infty(\TT)$} .
\end{equation}
uniformly in time where $p'$ is the dual H\"older index of $p$.
\end{enumerate}
\end{mdframed}
\end{figure}

\hfill

 These two assumptions, \ref{assu:f_continuity} and \ref{assu:Forcing_bound}, roughly say that \eqref{eq:model_linear} is a small perturbation of the original equation $ \p_t f = \mathcal{L}_\alpha f$ in some suitable sense.

\begin{proposition}[Illposedness for $p<2$]\label{prop:model_p<2}
For any $0< \alpha< \frac{1}{2}$, $1\leq p  <2 $, and $\delta>0$, there exists an even initial data $f_0 \in L^p(\TT)$ such that the following holds.

Consider any model equation \eqref{eq:model_linear} satisfying the $(L^p,\delta)$-assumptions~\ref{assu:f_continuity}--\ref{assu:Forcing_bound}. If $f:\TT \times [0,T] \to \RR$
is  a weak solution of \eqref{eq:model_linear} with the initial data $ f_0  $ in the sense of Definition \ref{def:wea_sol_model_eq}, then it must  satisfy
\begin{equation}\label{eq:f_blowup_p<2}
	\sup_{t \in [0,T ]}  |f(\cdot, t)|_{L^p(\TT)} = \infty \quad \text{for any $T >0$.}
\end{equation}

\end{proposition}
\begin{proof}
We fix $1 \leq p<2$ and assume $0< \delta \leq 1$.
Before we choose the initial data, let us fix a small number $\ep =\ep(\alpha, p, \delta)>0$   such that
\begin{equation}\label{eq:aux_p<2_1}
0< \ep < \left( \frac{1}{p} - \frac{1}{2} \right) \frac{\delta \alpha }{4 } .
\end{equation}

Let us choose the initial data. Let the real-valued even function $f_0 \in L^p(\TT)$ be given by Lemma \ref{lemma:bulding_f_pless2} with $\ep$ replaced by $ \frac{\ep}{2} $   such that $ e^{t \mathcal{L}_\alpha } f_0 \not \in L^p(\TT)$ with the estimates
\begin{equation}\label{eq:aux_p<2_2}
	| e^{t \mathcal{L}_\alpha } \Delta_q f_0|_{L^p} \geq   t^{ \frac{1}{ p }  - \frac{1}{2} } \lambda_q^{2\alpha(\frac{1}{ p }  - \frac{1}{2} )  - \frac{\ep}{2}}
\end{equation}
for all sufficiently large $q\in \NN$  with $ t\l_q^{2 \alpha }  \geq 1 $.

We proceed with a proof by contradiction. Suppose there is a weak solution $f:\TT \times [0,T] \to \RR$ of \eqref{eq:model_linear} with the initial data $f_0 $ such that
\begin{equation}\label{eq:aux_p<2_3}
 \sup_{t \in [0,T] } |f(\cdot, t)|_{L^p(\TT)} \leq M
\end{equation}
for some   $ T > 0$  and $M>0$.

Next, we use a class of special test functions in weak formulation to exploit the illposedness of $e^{t \mathcal{L}_\alpha}$. For any $T>0  $, consider test functions of the form
\begin{equation}\label{eq:aux_p<2_4}
\varphi(x, t) = e^{-(T-t) \mathcal{L_\alpha }}\phi(x)
\end{equation}
with $\phi \in C^\infty (\TT)$ smooth.
Since these test functions \eqref{eq:aux_p<2_4} satisfy the equation $\p_t \varphi = \mathcal{L}_\alpha \varphi$ classically, we have the equality
\begin{equation}\label{eq:aux_p<2_5}
\p_t \varphi = \mathcal{L}_\alpha\varphi  = e^{-(T-t) \mathcal{L_\alpha }}  \mathcal{L}_\alpha \phi .
\end{equation}
From \eqref{eq:linear_weak} and \eqref{eq:aux_p<2_5} we deduce that the solution $f$ satisfies for any $\phi \in C^\infty(\TT)$ and $T>0$
\begin{equation}\label{eq:aux_p<2_6}
	\begin{aligned}
 & \int_{\TT} f(x, T) \phi(x) \, dx - \int_{\TT} f_0 (x )  e^{ - T \mathcal{L_\alpha }} \phi(x) \,d x  \\
& = \int_{\TT \times [0,T]} (f -\overline{f} )         e^{-(T-t) \mathcal{L_\alpha }}  \mathcal{L}_\alpha \phi  \, dx dt + \int_0^T  F_t ( e^{-(T-t) \mathcal{L_\alpha }}\phi    )   \,dt .
	\end{aligned}
\end{equation}

 We choose a family of test functions $ \phi_q \in C^\infty(\TT)$ and a sequence of times $\tau_q  \to 0$ for  $ q\in \NN$ as follows.
 \begin{enumerate}
 	\item First, we choose $ \tau_q \in (0, 1)$ such that
 	\begin{equation}\label{eq:def_tau_q}
 		 \tau_q \lambda_q^{2\alpha } =  \lambda_q^{\frac{ \delta \alpha  }{4}}
 	\end{equation}
 where we can assume $0< \delta \leq 1$. When $q$ is large, we have $\tau_q \leq T$ and $  \tau_q \l_q^{2 \alpha } \geq 1$.

 \item Second, by \eqref{eq:def_tau_q} and \eqref{eq:aux_p<2_2}, we can take a sequence of test functions $\phi_q \in C^\infty (\TT)$ such that
 \begin{equation}\label{eq:phi_duality_2}
 	\int_{\TT} \Delta_q \left[ e^{  \tau_q  \mathcal{L_\alpha }} f_0 \right] \phi_q \geq   \tau_q^{ \frac{1}{ p }  - \frac{1}{2} } \lambda_q^{2\alpha(\frac{1}{ p }  - \frac{1}{2} )  - \frac{\ep}{2}}
 \end{equation}
for all $q$ sufficiently large and
 \begin{equation}\label{eq:phi_duality_1}
| \phi_q|_{L^{p'}} = 1.
 \end{equation}
 The existence of such $\phi_q$'s follows from   the duality of $L^p$ norm: for any $ f\in L^p(\TT)$, one has
$$
| f |_{L^p(\TT)} =\sup_{ |g |_{L^{p'}} =1} \int_{\TT} g f.
$$

\item Since the sequence of test functions $\phi_q $ has Fourier support comparable  to  $\Delta_q$, by \eqref{eq:def_tau_q} and \eqref{eq:phi_duality_1} we have the brute-force estimate from Lemma \ref{lemma:K_q} and Lemma \ref{lemma:bernstein},
\begin{equation}\label{eq:aux_error_bruteforce}
	\sup_{ t\in [0, \tau_q] }|  e^{-(\tau_q- t) \mathcal{L}_\alpha } \mathcal{L}_\alpha \Delta_q  \phi_q  |_{L^{p'}} \lesssim \tau_q^{ \frac{1}{2} } \l_q^{3 \alpha } \qquad   \text{for all  $q \in \NN$}.
\end{equation}
Here the constant is independent of   $q$.

 \end{enumerate}

We are in a position to get a contradiction. Consider the formulation \eqref{eq:aux_p<2_6} with $  \Delta_q \phi_q$ and $ T= \tau_q$ for $q \in \NN $ large, namely
\begin{equation}\label{eq:aux_p<2_contradiction_1}
\begin{aligned}
& \underbrace{	   \int_{\TT} f(x, \tau_q) \Delta_q  \phi_q(x) \, dx   - \int_{\TT} f_0 (x )  e^{ - \tau_q \mathcal{L_\alpha }}  \Delta_q \phi_q (x) \,d x   }_{:= L_q} \\
	& =  \underbrace{ \int_{\TT \times [0,\tau_q]} (f -\overline{f} )         e^{-(\tau_q -t) \mathcal{L_\alpha }}  \mathcal{L}_\alpha \Delta_q \phi_q  \, dx dt  +    \int_0^{\tau_q }  F_t ( e^{-(\tau_q -t) \mathcal{L_\alpha }} \Delta_q \phi    )   \,dt  }_{:= R_q} .
\end{aligned}
\end{equation}
Here $L_q, R_q \in \RR $ are the sequences of real numbers corresponding to the left-hand and right-hand sides. We will eventually show that $L_q \neq R_q$ for some $q \in \NN$ large under the hypothesis \eqref{eq:aux_p<2_3}.

Starting with the left-hand side of \eqref{eq:aux_p<2_contradiction_1}, we use integration by parts, the H\"older inequality, and \eqref{eq:phi_duality_2} to derive a lower bound for $|L_q|$.
First, note that
\begin{equation}\label{eq:aux_p<2_contradiction_2}
  |L_q |  \geq   \int_{\TT}\Delta_q \big(  e^{  \tau_q \mathcal{L_\alpha }}   f_0   \big)     \phi_q (x) \,d x   - |f(\cdot, \tau_q)|_{L^{p}}  |\Delta_q  \phi_q  |_{L^{p'}} .
\end{equation}
 By \eqref{eq:phi_duality_2} and the hypothesis \eqref{eq:aux_p<2_3},  it follows that from \eqref{eq:aux_p<2_contradiction_2} that
\begin{align}\label{eq:aux_p<2_contradiction_L_q}
| L_q | \geq    \tau_q^{ \frac{1}{ p }  - \frac{1}{2} } \lambda_q^{2\alpha(\frac{1}{ p }  - \frac{1}{2} )  - \frac{\ep}{2} } - C   \quad \text{for all sufficiently large $q \in \NN$} .
\end{align}
However,  the exponent of the first term in \eqref{eq:aux_p<2_contradiction_L_q} is positive:
\begin{align}\label{eq:aux_p<2_contradiction_6}
\tau_q^{ \frac{1}{p} -\frac{1}{2}} \l_q^{ \frac{2\alpha}{p} -\alpha  -\frac{\ep}{2} } = ( \tau_q \l_q^{  {2\alpha}   } )^{ \frac{1}{p} -\frac{1}{2}}\l_q^{    -\frac{\ep}{2} } \geq \l_q^{(\frac{1}{p} -\frac{1}{2})\frac{\delta \alpha }{4}  -\frac{\ep}{2} }.
\end{align}
Since by \eqref{eq:aux_p<2_1}, we have $0< \ep < (\frac{1}{p } -\frac{1}{ 2})\frac{\delta \alpha }{4} $, the exponent is strictly positive, and hence $L_q$ blows up as $q \to \infty$.

Now let us consider the right-hand side value $R_q$. The goal is to derive an upper bound smaller than \eqref{eq:aux_p<2_contradiction_L_q}.  From \eqref{eq:aux_p<2_contradiction_1}, we  first use the H\"older inequality and \eqref{eq:phi_duality_1} to obtain that for any $q\in \NN$,
\begin{equation}\label{eq:aux_p<2_contradiction_3}
\begin{aligned}
|R_q |& \leq  \tau_q  \sup_{t\in [0,\tau_q]} |f(t) - \overline{f}(t) |_{L^p}        \sup_{t\in [0,\tau_q]}  \Big|  e^{-( \tau_q -t) \mathcal{L_\alpha }}  \mathcal{L}_\alpha \Delta_q \phi_q \Big|_{L^{p'}} \\
& \qquad \qquad + \tau_q \sup_{t\in [0,\tau_q]} \left|  F_t ( e^{-(\tau_q -t) \mathcal{L_\alpha }}\Delta_q \phi_q    ) \right|.
\end{aligned}
\end{equation}

For the first term in \eqref{eq:aux_p<2_contradiction_3} we use the assumption \ref{assu:f_continuity}  and \eqref{eq:aux_error_bruteforce}:
\begin{align}\label{eq:aux_p<2_contradiction_3a}
& \sup_{t\in [0,\tau_q]} |f(t) - \overline{f}(t) |_{L^p}        \sup_{t\in [0,\tau_q]}  \Big|  e^{-( \tau_q -t) \mathcal{L_\alpha }}  \mathcal{L}_\alpha \Delta_q \phi_q \Big|_{L^{p'}}  \nonumber \\
& \leq  C \tau_q^{\delta }  \tau_q^{\frac{1}{2}  }\l_q^{3\alpha}        \quad \text{for any sufficiently large $q \in \NN$}.
\end{align}
For the second term in \eqref{eq:aux_p<2_contradiction_3}, we first use the assumption \ref{assu:Forcing_bound}, the definition of Besov norms, and Lemma \ref{lemma:K_q}:
\begin{align}\label{eq:aux_p<2_contradiction_3b}
&\sup_{t\in [0,\tau_q]} \left|  F_t ( e^{-(T-t) \mathcal{L_\alpha }} \Delta_q \phi_q    )  \right| \nonumber\\
& \lesssim  \sup_{t\in [0,\tau_q]} |  e^{-(T-t) \mathcal{L_\alpha }} \Delta_q \phi_q |_{B^{2\alpha - \delta }_{p', 1}} \nonumber \\
& \lesssim \l_q^{2\alpha - \delta } \sup_{t\in [0,\tau_q]} |  e^{-(T-t) \mathcal{L_\alpha }} \Delta_q \phi_q |_{L^{p' }} \nonumber \\
& \leq  C \l_q^{2\alpha -\delta } \tau_q^{\frac{1}{2}  }  \l_q^{\alpha}        \quad \text{for any $q \in \NN$}.
\end{align}

It follows from \eqref{eq:aux_p<2_contradiction_3}, \eqref{eq:aux_p<2_contradiction_3a}, and \eqref{eq:aux_p<2_contradiction_3b} that
\begin{align}\label{eq:aux_p<2_contradiction_4}
| R_q | \leq  C \tau_q^{\frac{3}{2} +\delta }  \l_q^{3\alpha} + \tau_q^{\frac{3}{2}  }  \l_q^{3\alpha -\delta }         \quad \text{for any sufficiently large $q \in \NN$}.
\end{align}

Finally, we show that \eqref{eq:aux_p<2_contradiction_4} leads to a contradiction. Note that by \eqref{eq:def_tau_q}, we have $ \tau_q^\delta \leq \l_q^{- 2\alpha \delta + \frac{\alpha \delta}{4}}$. By \eqref{eq:aux_p<2_contradiction_4}, this implies the bound
\begin{align*}
 | R_q |  & \leq C (\tau_q \l_q^{2\alpha}  )^{\frac{3}{2}} \tau_q^{ \delta }  + C(\tau_q \l_q^{2\alpha}  )^{\frac{3}{2}} \l_q^{ -\delta } \\
& \leq C \l_q^{ \frac{\alpha \delta  }{2}  -2\alpha \delta +\frac{\alpha \delta  }{4} } +C \l_q^{ \frac{\alpha \delta  }{2}  -  \delta   }
\end{align*}
which is decaying when $q \to \infty$. In other words, $R_q$ is uniformly bounded for any $ q\in \NN $. This contradicts the fact that $ L_q$ blows up. So \eqref{eq:aux_p<2_3} can not hold and we are done.

\end{proof}

The next result is the extension of the $L^p$ illposedness to $p>2$. In general, the $L^p$ boundedness/unboundedness of a linear operator can be deduced from that of its adjoint operator. So one would expect the same equation is strongly ill-posed in $L^p$ for $p>2$. Here the situation is slightly more involved since we are dealing with time evolution, and the initial data needs to develop an ``instant blowup''.

We combine the usual duality approach with a patching argument to induce the desired blowups in the high-frequency.

\begin{proposition}[Illposedness for $p>2$]\label{prop:model_p>2}
For any $0< \alpha< \frac{1}{2}$, $2 < p  \leq \infty $, and $\delta>0$, there exists an even  initial data $f_0 \in L^p(\TT)$ such that the following holds.

Consider any model equation \eqref{eq:model_linear} satisfying the $(L^p,\delta)$-assumptions~\ref{assu:f_continuity}--\ref{assu:Forcing_bound}. If $f:\TT \times [0,T] \to \RR$
is  a weak solution of \eqref{eq:model_linear} with the initial data $ f_0  $ in the sense of Definition \ref{def:wea_sol_model_eq}, then it must  satisfy
\begin{equation}\label{eq:f_blowup}
	\sup_{t \in [0,T ]}  |f(\cdot, t)|_{L^p(\TT)} = \infty \quad \text{for any $T>0$.}
\end{equation}

\end{proposition}
\begin{proof}

First of all, let us assume $0<\delta \leq 1$ and fix $\ep =\ep(\alpha, p, \delta)>0$   such that
\begin{equation}\label{eq:small_ep_pbig2}
	0< \ep < \left( \frac{1}{2} - \frac{1}{p} \right) \frac{\delta \alpha }{8}.
\end{equation}

Since $p>2$, its H\"older dual $p' \in [1 , 2)$, for this $\ep$ we apply Lemma \ref{lemma:bulding_f_pless2} to obtain $ \phi \in L^{p'}$ such that
\begin{equation}\label{eq:aux_p>2_contradiction_1}
	|e^{ t \mathcal{L}_\alpha  } \Delta_q  \phi   |_{L^{p'}} \geq  t^{ \frac{1}{2} -\frac{1}{p}} \l_q^{   2\alpha (\frac{1}{2}  - \frac{1}{p})  -  \frac{\ep }{2} }
\end{equation}
for all sufficiently large $q$ with $ t\l_q^{2 \alpha }  \geq 1 $ .
In contrast to Proposition \ref{prop:model_p<2}, we will use this $\phi$ as a test function.

We take a sequence of times $\tau_q \in (0,T]$ for  such that
\begin{equation}\label{eq:def_tau_q_2}
	\tau_q \lambda_q^{2\alpha } =  \lambda_q^{\frac{ \alpha \delta   }{4}} .
\end{equation}
With the sequence $\tau_q$ chosen, we now construct the initial data. As in the previous proposition, the duality of $L^p$ norms implies that there exist a sequence of smooth functions $ h_q \in C^\infty (\TT)$, $q \in \NN$ such that
\begin{enumerate}
\item $h_q$ is even and $h_q$ has frequency support near $\lambda_q$;

	\item For any sufficiently large $q\in \NN$, $h_q$ satisfies $| h_q  |_{L^p} \leq \l_{q}^{-  \frac{\ep }{2}  }$ and
	\begin{equation}\label{eq:g_q_blowup}
		\int h_q e^{\tau_q \mathcal{L}_\alpha  } \Delta_q \phi  \geq         \tau_q^{ \frac{1}{2} -\frac{1}{p}}  \lambda_q^{2\alpha(\frac{1}{ 2 }  - \frac{1}{  p } )  -  \ep} .
	\end{equation}

\end{enumerate}
We can be sure to find such even $h_q$ due to the fact that any odd part does not contribute to the integral \eqref{eq:g_q_blowup} since $e^{\tau_q \mathcal{L}_\alpha  } \Delta_q \phi  $ is even,
and so can be discarded.

Now our initial data is
\begin{equation}\label{eq:aux_p>2_intial}
f_0 := \sum_{q\in 2\NN} h_q  .
\end{equation}
We take $q\in 2\NN$ to isolate the contribution from each frequency band $\l_q$ and any sufficiently lacunary sequence will work. It is clear that $f_0 $ is even and in $  L^p(\TT)$ since $|h_q |_{L^p} \leq \l_{q}^{- \frac{\ep }{2} }$. We note that when $p = \infty$, $f_0$ is in fact continuous due to this epsilon decay.

We proceed with a proof by contradiction. Suppose there is a weak solution $f :\TT \times [0,T ] \to \RR$ with the initial data $f_0 \in L^p(\TT)$ \eqref{eq:aux_p>2_intial} such that
\begin{equation}\label{eq:aux_p>2_assumption}
 \sup_{t \in [0,T] } |f(\cdot, t)|_{L^p(\TT)} \leq M
\end{equation}
for some   $ T > 0$  and $M>0$.

Now we take a look again at the weak formulation \eqref{eq:aux_p<2_6} used in the proof Proposition \ref{prop:model_p<2}. The weak solution $f$ satisfies for any  $\psi \in C^\infty (\TT)$  the following equation
\begin{equation}\label{eq:aux_p>2_contradiction_2}
\begin{aligned}
&			  \int_{\TT} f(x, T) \psi(x) \, dx -  \int_{\TT} f_0 (x )  e^{ - T \mathcal{L_\alpha }} \psi(x) \,d x  \\
& = \int_{\TT \times [0,T]} (f -\overline{f} )         e^{-(T-t) \mathcal{L_\alpha }}  \mathcal{L}_\alpha \psi  \, dx dt \\
&  \qquad +  \int_0^T F_t( e^{-(T-t) \mathcal{L_\alpha }}\psi ) \, dt.
\end{aligned}
\end{equation}

For all sufficiently large $q\in 2\NN$ consider the test functions
\begin{equation}\label{eq:aux_p>2_contradiction_3}
 \psi =  \Delta_q \phi \in  C^\infty (\TT)
\end{equation}
 over time intervals $[0,\tau_q]$ where $ \phi$ is as in \eqref{eq:aux_p>2_contradiction_1} and \eqref{eq:g_q_blowup}.

Since the frequency supports of the summands in $f_0 =\sum h_q$ lie in separated annuli, using \eqref{eq:aux_p>2_contradiction_3} in \eqref{eq:aux_p>2_contradiction_2} gives
\begin{equation}\label{eq:aux_p>2_contradiction_4}
\begin{aligned}
&  \underbrace{	\int_{\TT} f(x, \tau_q) \Delta_q   \phi(x) \, dx -\int_{\TT} h_q (x )  e^{ - \tau_q \mathcal{L_\alpha }} \Delta_q  \phi (x) \,d x    }_{:= L_q }\\
	& = \underbrace{\int_{\TT \times [0,\tau_q]} (f -\overline{f} )         e^{-(\tau_q -t) \mathcal{L_\alpha }}  \mathcal{L}_\alpha \Delta_q  \phi \, dx dt +   \int_0^{\tau_q} F_t( e^{-(\tau_q -t) \mathcal{L}_\alpha  } \Delta_q \phi  )\, dt}_{:= R_q} .\\
\end{aligned}
\end{equation}
As before, $L_q, R_q \in \RR $ are the sequences of real numbers corresponding to the left-hand and right-hand sides. We will show that $L_q \neq R_q$ for some sufficiently large $q \in 2\NN$ under the hypothesis \eqref{eq:aux_p>2_assumption}.

Let us start with $L_q$, the left-hand side of \eqref{eq:aux_p>2_contradiction_4}. Since $ \phi \in L^{p'}$, $1 \leq p'<2$ and $h_q$ satisfies \eqref{eq:g_q_blowup},  by \eqref{eq:aux_p>2_assumption}  we use H\"older to obtain that for all large $q \in 2\NN$
\begin{equation}
\begin{aligned} \label{eq:contradiction_eq_2}
|L_q| & \geq  \tau_q^{ \frac{1}{2} -\frac{1}{p}} \l_q^{ 2\alpha (\frac{1}{2} - \frac{1}{p})  -   \ep} - |f(\cdot,\tau_q) |_{L^p} | \Delta_q \phi|_{L^{p'} }  \\
& \geq  \tau_q^{ \frac{1}{2} -\frac{1}{p}} \l_q^{ 2\alpha (\frac{1}{2} - \frac{1}{p})  -   \ep} -  C  .
\end{aligned}
\end{equation}
Similarly to the proof of Proposition \ref{prop:model_p<2}, the exponent of the  first term is positive: by \eqref{eq:def_tau_q_2},
\begin{equation}\label{eq:aux_p>2_contradiction_5}
\tau_q^{ \frac{1}{2} -\frac{1}{p}} \l_q^{ 2\alpha (\frac{1}{2} - \frac{1}{p})  -   \ep}  = [\tau_q \l_q^{ 2\alpha } ]^{  \frac{1}{2} -\frac{1}{p} }      \l_q^{-\ep} =  \l_q^{  {\frac{ \alpha \delta   }{4}}(\frac{1}{2} -\frac{1}{p} )  - \ep }
\end{equation}
where ${\frac{ \alpha \delta   }{4}}(\frac{1}{2} -\frac{1}{p} )  - \ep>0 $  thanks to \eqref{eq:small_ep_pbig2}. Hence $L_q$ blows up as $q \to \infty$.

Finally we consider $R_q$, the right-hand side of \eqref{eq:aux_p>2_contradiction_4}. We will show the right-hand side $R_q$ remains uniformly bounded. By H\"older's inequality again,
\begin{equation}\label{eq:aux_p>2_contradiction_6}
\begin{aligned}
|R_q|  \leq  &  \tau_q  \sup_{t\in [0 ,\tau_q]}  |f -\overline{f} |_{L^p}  \sup_{t\in [0 ,\tau_q]} | e^{-(\tau_q -t) \mathcal{L_\alpha }}  \mathcal{L}_\alpha \Delta_q  \phi  |_{L^{p'}}  \\
 &\qquad  \qquad  + \tau_q  \sup_{t\in [0 ,\tau_q]} | F_t( e^{-(\tau_q -t) \mathcal{L}_\alpha  } \Delta_q \phi  )   |    .
\end{aligned}
\end{equation}
For the first term in \eqref{eq:aux_p>2_contradiction_6}, by the time-continuity of $f(t) -\overline{f}(t)$ in the assumption \ref{assu:f_continuity} and the bound by Lemma \ref{lemma:K_q} and Lemma \ref{lemma:bernstein}
\begin{equation}\label{eq:aux_p>2_contradiction_7}
\sup_{t\in [0 ,\tau_q]}  | e^{-(\tau_q -t) \mathcal{L_\alpha }}  \mathcal{L}_\alpha \Delta_q  \phi |_{L^{p'}}  \lesssim \tau_q^{ \frac{1}{2}} \l_q^{3\alpha },
\end{equation}
we have
\begin{equation}\label{eq:aux_p>2_contradiction_8}
\tau_q \sup_{t\in [0 ,\tau_q]}  |f -\overline{f} |_{L^p}  \sup_{t\in [0 ,\tau_q]} | e^{-(\tau_q -t) \mathcal{L_\alpha }}  \mathcal{L}_\alpha \Delta_q  \phi  |_{L^{p'}}  \lesssim    \big[ \tau_q\l_q^{2\alpha }  \big]^{ \frac{3}{2}   } \tau_q^{   \delta } .
\end{equation}
For the second term in \eqref{eq:aux_p>2_contradiction_6}, by the assumption \ref{assu:Forcing_bound} and Lemma \ref{lemma:K_q} again,
\begin{align}
& \tau_q \sup_{t\in [0 ,\tau_q]} | F_t( e^{-(\tau_q -t) \mathcal{L}_\alpha  } \Delta_q \phi  )   |   \nonumber \\
&   \lesssim  \tau_q \sup_{t\in [0 ,\tau_q]} |  e^{-(\tau_q -t) \mathcal{L}_\alpha  } \Delta_q \phi      |_{B^{2\alpha -\delta }_{p',1}} \nonumber  \\
& \lesssim  \tau_q \l_q^{2\alpha -\delta }  \sup_{t\in [0 ,\tau_q]} |  e^{-(\tau_q -t) \mathcal{L}_\alpha  } \Delta_q \phi      |_{L^{ p' }}  \nonumber \\
& \lesssim   \big[ \tau_q\l_q^{2\alpha }  \big]^{ \frac{3}{2}   } \l_q^{   - \delta }.\label{eq:aux_p>2_contradiction_9}
\end{align}
Combining \eqref{eq:aux_p>2_contradiction_8} and \eqref{eq:aux_p>2_contradiction_9}  in \eqref{eq:aux_p>2_contradiction_6}, we have
\begin{equation}\label{eq:aux_p>2_contradiction_10}
|R_q |\leq C \big[ \tau_q\l_q^{2\alpha }  \big]^{ \frac{3}{2}   } \tau_q^{   \delta } + C \big[ \tau_q\l_q^{2\alpha }  \big]^{ \frac{3}{2}   } \l_q^{-   \delta }
\end{equation}

Thanks to \eqref{eq:def_tau_q_2}, we have $ \tau_q^\delta \leq \l_q^{- 2\alpha \delta + \frac{\alpha \delta}{4}}$. It then follows that
$$
|R_q |\lesssim    \l_q^{ \frac{\alpha \delta  }{2}  - 2\alpha \delta + \frac{\alpha \delta }{4}} + \l_q^{\frac{\alpha \delta }{2}  - \delta }  \to 0
$$
as $q \to \infty$. So $R_q$ is uniformly bounded, a contradiction to the fact that $L_q$ blows up.

\end{proof}

From Lemma \ref{lemma:bernstein}, it is easy to see the initial data in Proposition \ref{prop:model_p<2} and Proposition \ref{prop:model_p>2} are stable under smooth perturbations. For later application to the curvature equation, we state a combined corollary suitable in that setting.

\begin{corollary}\label{cor:model_p_neq_2}
For any $0< \alpha< \frac{1}{2}$, $  p  \neq 2 $, and $\delta>0$,  there exists an even and zero-mean function  $f_{\sharp} \in L^p(\TT)$ such that the following holds.

Consider any model equation \eqref{eq:model_linear} satisfying the $(p,\delta)$-assumptions~\ref{assu:f_continuity}--\ref{assu:Forcing_bound}.

For any $f_g \in C^\infty(\TT)$ and any $\ep>0$, if  $f$ is a weak solution of \eqref{eq:model_linear} with the initial data $ f_0 = \ep f_{\sharp}  + f_g$ in the sense of Definition \ref{def:wea_sol_model_eq}, then $f$ must satisfy
\begin{equation}
	\sup_{t \in [0,T ]}  |f(\cdot, t)|_{L^p(\TT)} = \infty \quad \text{for any $T>0$.}
\end{equation}

\end{corollary}

\subsection{\texorpdfstring{$C^\beta$}{Cb} illposedness of the model equation}

In the last part of this section, we consider the model equation in $C^{\beta}$ H\"older spaces. The goal is to show that the illposedness holds for any $ 0 < \beta \leq 1$ since $\beta=0$ corresponds to the $L^\infty $ case that was covered before.  We use the duality approach with the patching argument in the $L^p$, $p>2$ case to prove the $C^\beta$ illposedness.

Given any $\beta \in (0 , 1]$ and $\delta>0$, we consider the \textbf{$(C^\beta, \delta)$-assumptions}, \ref{assu:f_continuity_Cbeta} and \ref{assu:Forcing_bound_Cbeta}, on \eqref{eq:model_linear}.  Similar to $L^p$ cases, these two assumptions roughly say that \eqref{eq:model_linear} is a small perturbation of the original equation $ \p_t f = \mathcal{L}_\alpha f$ in the $C^\beta$ scale.

\begin{figure}[!htb]
\begin{mdframed}[linewidth=1pt,skipabove=10pt,frametitle={$(C^\beta, \delta)$-assumptions:}]
\begin{enumerate}[label=(B{\arabic*})]

\item\label{assu:f_continuity_Cbeta} $ f \mapsto \overline{f}  $ is bounded in $C^\beta(\TT)$  with the estimate
\begin{equation}\label{eq:fbar_bound_Cbeta}
	|f(\cdot, t)  -\overline{f}(\cdot, t)|_{C^{0, \beta}(\TT)} \lesssim t^{\delta} \quad \text{uniformly in time for all $t \geq 0$ small;}
\end{equation}

\item\label{assu:Forcing_bound_Cbeta} the distributional forcing $F_t$   satisfies the bound
\begin{equation}\label{eq:Forcing_bound_Cbeta}
\begin{aligned}
| F_t  |_{L^\infty B^{\beta-2\alpha+\delta}_{\infty,\infty} }< \infty .
\end{aligned}
\end{equation}

\end{enumerate}
\end{mdframed}
\end{figure}

\hfill

\begin{proposition}[Illposedness for $C^\beta$]\label{prop:model_Cbeta}
For any $0< \alpha< \frac{1}{2}$, $0 < \beta   \leq 1 $, and $\delta>0$, there exists an even initial data $f_0 \in C^{  \beta}(\TT)$ such that the following holds.

Consider any model equation \eqref{eq:model_linear} satisfying the $(C^\beta,\delta)$-assumptions~\ref{assu:f_continuity_Cbeta}--\ref{assu:Forcing_bound_Cbeta}. If $f$
is  a weak solution of \eqref{eq:model_linear} with the initial data $ f_0  $ in the sense of Definition \ref{def:wea_sol_model_eq}, then it must  satisfy
\begin{equation}\label{eq:f_blowup_Cbeta}
	\sup_{t \in [0,T ]}  |f(\cdot, t)|_{C^\beta(\TT)} = \infty \quad \text{for any $T>0$.}
\end{equation}

\end{proposition}
\begin{proof}
First of all, let us  fix $\ep =\ep(\alpha, \delta)>0$   such that
\begin{equation}\label{eq:small_ep_Cbeta}
	0< \ep <  \frac{\delta \alpha }{16}
\end{equation}
and take a sequence of times $\tau_q \in (0,T]$ for  such that
\begin{equation}\label{eq:def_tau_Cbeta}
	\tau_q \lambda_q^{2\alpha } =  \lambda_q^{\frac{ \alpha \delta   }{4}} .
\end{equation}
For this $\ep>0$  and $p=1$  we apply Lemma \ref{lemma:bulding_f_pless2} to obtain the even real-valued function $ \phi \in L^{1}(\TT)$ such that
\begin{equation}\label{eq:aux_Cbeta_contradiction_1}
	|e^{ \tau_q  \mathcal{L}_\alpha  } \Delta_q  \phi   |_{L^{1 }} \geq  \tau_q^{ \frac{1}{2} }   \l_q^{ \alpha     -  \frac{\ep }{2} } \quad \text{for all sufficiently large $q \in \NN $}.
\end{equation}
As in the proof of Proposition \ref{prop:model_p>2}, we will use this $\phi$ as a test function.

Once again, the duality of $L^1$ norm implies that there exist a sequence $ h_q \in C^\infty (\TT)$, $q \in \NN$ such that
\begin{enumerate}
\item $h_q$ is even and  its frequency support is contained in that of $\Delta_q \phi$;

	\item For any sufficiently large $q\in \NN$, $h_q$ satisfies $| h_q  |_{L^\infty } \leq \l_{q}^{-  \frac{\ep }{2}  }$ and
	\begin{equation}\label{eq:g_q_blowup_Cbeta}
		\int h_q e^{\tau_q \mathcal{L}_\alpha  } \Delta_q \phi  \geq       \tau_q^{ \frac{1}{2}  }  \lambda_q^{ \alpha   -  \ep} .
	\end{equation}

\end{enumerate}
Here the evenness of $h_q$ follows from the same reasoning as in Proposition \ref{prop:model_p>2}.

Our initial data is
\begin{equation}\label{eq:aux_Cbeta_intial}
f_0 := \sum_{q\in 2\NN} \l_q^{-\beta } h_q  .
\end{equation}
We take $q\in 2\NN$ to isolate the contribution from each frequency band $\l_q$. It is clear that $f_0 $ is even and in $  C^\beta(\TT)$ since $|h_q |_{L^\infty } \leq \l_{q}^{- \frac{\ep }{2} }$  thanks to the characterization of H\"older continuous function \eqref{eq:holder_LP_condition}.

Suppose there is a weak solution $f$ with the initial data $f_0 $ \eqref{eq:aux_Cbeta_intial} such that
\begin{equation}\label{eq:aux_Cbeta_assumption}
 \sup_{t \in [0,T] } |f(\cdot, t)|_{C^{0, \beta}(\TT)} \leq M
\end{equation}
for some   $ T > 0$  and $M>0$.

Then as before this weak solution $f$ must satisfy for any $\zeta \in C^\infty(\TT)$,
\begin{equation}\label{eq:aux_Cbeta_contradiction_2}
\begin{aligned}
&		\int_{\TT} f(x, T) \zeta(x) \, dx - \int_{\TT} f_0 (x )  e^{ - T \mathcal{L_\alpha }} \zeta(x) \,d x   \\
& = \int_{\TT \times [0,T]} (f -\overline{f} )         e^{-(T-t) \mathcal{L_\alpha }}  \mathcal{L}_\alpha \zeta  \, dx dt \\
&  \qquad + \int_0^T F_t( e^{-(T-t) \mathcal{L_\alpha }}\zeta ) \, dt.
\end{aligned}
\end{equation}

For all sufficiently large $q\in 2\NN$ consider the test functions
\begin{equation}\label{eq:aux_Cbeta_contradiction_3}
 \zeta =  \l_q^{\beta }\Delta_q \phi \in   C^\infty (\TT)
\end{equation}
 over time intervals $[0,\tau_q]$ where $ \phi$ is as in \eqref{eq:aux_Cbeta_contradiction_1} and \eqref{eq:g_q_blowup_Cbeta}.

Due to the frequency supports of the summands defining $f_0 = \sum \l_q^{-\beta } h_q$ belonging to disjoint annuli, using \eqref{eq:aux_Cbeta_contradiction_3} in \eqref{eq:aux_Cbeta_contradiction_2} gives
\begin{equation}\label{eq:aux_Cbeta_contradiction_4}
\begin{aligned}
&  \underbrace{  \l_q^{\beta }\int_{\TT} f(x, \tau_q) \Delta_q   \phi(x) \, dx -  \int_{\TT} h_q (x )  e^{ - \tau_q \mathcal{L_\alpha }} \Delta_q  \phi (x) \,d x   }_{:= L_q }\\
	& = \l_q^{\beta } \underbrace{\int_{\TT \times [0,\tau_q]} (f -\overline{f} )         e^{-(\tau_q -t) \mathcal{L_\alpha }}  \mathcal{L}_\alpha \Delta_q  \phi \, dx dt + \int_0^{\tau_q} F_t( \l_q^{\beta } e^{-(\tau_q -t) \mathcal{L}_\alpha  } \Delta_q \phi  )\, dt}_{:= R_q} .\\
\end{aligned}
\end{equation}
As before, we will show that $L_q \neq R_q$ for some sufficiently large $q \in \NN$ under the hypothesis \eqref{eq:aux_Cbeta_assumption}.

Let us start with $L_q$, the left-hand side of \eqref{eq:aux_Cbeta_contradiction_4}. Since $ \phi \in L^{1}$,  and $h_q$ satisfies \eqref{eq:g_q_blowup_Cbeta}, we use H\"older to obtain that for all large $q \in 2\NN$
\begin{equation}
\begin{aligned} \label{eq:contradiction_eq_Cbeta}
|L_q| & \geq  \tau_q^{ \frac{1}{2} } \l_q^{  \alpha   -   \ep} -   \l_q^{\beta}| \widetilde\Delta_q  f(\cdot,\tau_q) |_{L^\infty} | \Delta_q \phi|_{L^{1} }  .
\end{aligned}
\end{equation}

By the assumption \eqref{eq:aux_Cbeta_assumption} and the fact $\phi \in L^1(\TT)$, we have
\begin{equation}\label{eq:contradiction_eq_Cbeta_1}
|L_q| \geq  \tau_q^{ \frac{1}{2}  } \l_q^{  \alpha    -   \ep} -  C  .
\end{equation}
As in Proposition \ref{prop:model_p<2}, the exponent of the term is positive: by \eqref{eq:def_tau_Cbeta},
\begin{equation}
\tau_q^{ \frac{1}{2}  } \l_q^{  \alpha    -   \ep}  = [\tau_q \l_q^{ 2\alpha } ]^{  \frac{1}{2}  }      \l_q^{-\ep} =  \l_q^{  {\frac{ \alpha \delta   }{8}}   - \ep }
\end{equation}
where $ \frac{ \alpha \delta   }{8}    - \ep>0 $  thanks to \eqref{eq:small_ep_Cbeta}. Hence $L_q$ blows up as $q \to \infty$.

Finally, we consider $R_q$, the right-hand side of \eqref{eq:aux_Cbeta_contradiction_4}. We will show that the right-hand side remains uniformly bounded.

By H\"older's inequality again,
\begin{equation}\label{eq:contradiction_eq_Cbeta_2}
\begin{aligned}
|R_q | \leq  &  \tau_q  \sup_{t\in [0 ,\tau_q]}    \l_q^{\beta }  |   \widetilde\Delta_q f -  \widetilde\Delta_q \overline{f} |_{L^\infty}     | e^{-(\tau_q -t) \mathcal{L_\alpha }}  \mathcal{L}_\alpha \Delta_q  \phi  |_{L^{1}}  \\
 &\qquad  \qquad  +  \tau_q  \sup_{t\in [0 ,\tau_q]} | F_t( \l_q^{\beta } e^{ -(\tau_q -t) \mathcal{L}_\alpha  } \Delta_q \phi  ) |        .
\end{aligned}
\end{equation}

For the first term in \eqref{eq:contradiction_eq_Cbeta_2}, by the   assumption \ref{assu:f_continuity_Cbeta} and the following bound by Lemma \ref{lemma:K_q}  and $\phi \in L^1(\TT)$:
\begin{equation}\label{eq:contradiction_eq_Cbeta_3}
\sup_{t\in [0 ,\tau_q]}  | e^{-(\tau_q -t) \mathcal{L_\alpha }}  \mathcal{L}_\alpha \Delta_q  \phi |_{L^{1}}  \lesssim \tau_q^{ \frac{1}{2}} \l_q^{3\alpha },
\end{equation}
we have
\begin{align}
& \sup_{t\in [0 ,\tau_q]}     \l_q^{\beta }  |   \widetilde\Delta_q  f -     \widetilde\Delta_q   \overline{f} |_{L^\infty}      | e^{-(\tau_q -t) \mathcal{L_\alpha }}  \mathcal{L}_\alpha \Delta_q  \phi  |_{L^{1}}  \nonumber \\
& \lesssim    \sup_{t\in [0 ,\tau_q]}  |f (t)-\overline{f} (t) |_{C^{0, \beta} }  \sup_{t\in [0 ,\tau_q]} | e^{-(\tau_q -t) \mathcal{L_\alpha }}  \mathcal{L}_\alpha \Delta_q  \phi |_{L^{1}}  \nonumber  \\
& \leq  \tau_q ^{ \frac{1}{2}    } \l_q^{3\alpha } \tau_q^{   \delta } . \label{eq:contradiction_eq_Cbeta_4}
\end{align}
For the second term in \eqref{eq:contradiction_eq_Cbeta_2}, by  \eqref{eq:Forcing_bound_Cbeta} and the distributional pairing
\begin{align*}
\sup_{t\in [0 ,\tau_q]} | F_t( \l_q^{\beta } e^{ -(\tau_q-t)   \mathcal{L}_\alpha  } \Delta_q \phi  ) |     &  \lesssim \l_q^{2\alpha- \delta}  | F_t |_{ L^\infty B^{ \beta - 2\alpha+\delta }_{\infty,\infty } } \sup_{t\in [0 ,\tau_q]} \left|
e^{t  \mathcal{L}_\alpha  }  \Delta_q \phi       \right|_{L^1}.
\end{align*}
It follows from the   assumption \ref{assu:Forcing_bound_Cbeta} and Lemma \ref{lemma:K_q} that
\begin{align}\label{eq:contradiction_eq_Cbeta_5}
\sup_{t\in [0 ,\tau_q]} | F_t( \l_q^{\beta } e^{\tau_q \mathcal{L}_\alpha  } \Delta_q \phi  ) |     &   \lesssim     \tau_q^{\frac{1}{2} } \l_q^{3 \alpha  -\delta }.
\end{align}

Combining \eqref{eq:contradiction_eq_Cbeta_2}, \eqref{eq:contradiction_eq_Cbeta_4}, and \eqref{eq:contradiction_eq_Cbeta_5} we have the estimate for $R_q$:
\begin{equation}
 |R_q| \leq C \left[ \tau_q \l_q^{2\alpha } \right] ^{ \frac{3}{2}    }  \tau_q^{   \delta }    + C \left[ \tau_q \l_q^{2\alpha } \right] ^{ \frac{3}{2}    } \l_q^{ -\delta}.
\end{equation}

Due to \eqref{eq:def_tau_Cbeta}, we have $ \tau_q^\delta \leq \l_q^{- 2\alpha \delta + \frac{\alpha \delta}{4}}$. It then follows that
\begin{align*}
& |R_q| \lesssim  \l_q^{ \frac{\alpha \delta  }{2}  - 2\alpha \delta + \frac{\alpha \delta }{4}} +  \l_q^{ \frac{\alpha \delta  }{2}  -   \delta   }  \to 0
\end{align*}
as $q \to \infty$. So $R_q$ is uniformly bounded, a contradiction to the fact that $L_q$ blows up.

\end{proof}

As in the $L^p$ cases, the initial data in Proposition \ref{prop:model_Cbeta} is also stable under smooth perturbation. We state a corollary suitable for later application to the curvature equation.

\begin{corollary}\label{cor:model_Cbeta}
For any $0< \alpha< \frac{1}{2}$, $  0 \leq \beta  \leq 1 $, and $\delta>0$,  there exists  an even and zero-mean function $f_{\sharp} \in C^\beta(\TT)$ such that the following holds.

Consider any model equation \eqref{eq:model_linear} satisfying the $(C^\beta ,\delta)$-assumptions~\ref{assu:f_continuity_Cbeta}--\ref{assu:Forcing_bound_Cbeta}.

For any $f_g \in C^\infty(\TT)$ and any $\ep>0$, if  $f$ is a weak solution of \eqref{eq:model_linear} with the initial data $ f_0 = \ep f_{\sharp}  + f_g$ in the sense of Definition \ref{def:wea_sol_model_eq}, then $f$ must satisfy
\begin{equation}
	\sup_{t \in [0,T ]}  |f(\cdot, t)|_{C^{0, \beta}(\TT)} = \infty \quad \text{for any $T>0$.}
\end{equation}

\end{corollary}

\section{Proof of Sobolev illposedness}\label{sec:proof_illposed}
In this section, we put together the two main ingredients, the curvature flow with its leading order dynamics and illposedness for the model equation, to complete the proof of the main theorem.

\subsection{Smooth bending of  curves}

To show illposedness for the $\alpha$-patch problem, we need to take the initial data $\k$ for the curvature equation in the form of our examples in Section \ref{sec:disper_eq_illposed}.
It is not immediately obvious that we can do so, even if we assume $\int \k ds =2\pi$, as the endpoints of the arc could mismatch.

The lemma below shows that after suitable symmetric bending by adding two smooth bump functions to the existing curvature,
an arc becomes a closed curve.

\begin{lemma}\label{lemma:curve_bending}
Let $\k_{\sharp} \in L^p(\TT)$, $1 \leq p \leq \infty,$ even and such that $\int \k_{\sharp}\,ds= 0 $. Then there exists a simple closed  curve $\g$ with length $2\pi$ of class $W^{2,p}$ such that its curvature $\k$ can be decomposed as
\begin{equation}\label{eq:curve_bending}
\k = \ep \k_{\sharp} + \k_{g}
\end{equation}
for some small $\ep>0$ and $ \k_{g} \in C^\infty(\TT)$.
\end{lemma}
\begin{proof}

Let $ \g_*: [-\pi, \pi ] \to \RR^2 $ be the even arc with curvature $\ep \k_{\sharp}$ and centered at the origin: $\gamma_*(0)=0.$
We fix $\ep>0$ sufficiently small (depending on $\k_\sharp$) so that the unit tangent of $\g_*$ only varies by at most $ 1000^{-1}$.

We define a continuum of even arcs $\g(\cdot, h)$ centered at the origin via the curvature formula
\begin{equation}\label{eq:aux_curve_bending}
\k_{h}(s) =  \ep \k_\sharp(s) + \phi(s - h ) +  \phi(s + h )
\end{equation}
where $\phi \geq 0$ is a fixed smooth bump function with  $\Supp \phi \subset [- \pi/8 , \pi/8 ]$ and $\int \phi = \pi  $.

Since $\int \k_h = 2\pi $, \eqref{eq:aux_curve_bending} defines a family of even arcs with length $2\pi$ and two end-points facing each other. As in Figure \ref{fig:bending}, the two bumps $\phi(s - h )$ and  $\phi(s + h ) $ correspond to two symmetric bends of $180^{\circ} $ and the parameter $h$ determines the positions of the bends.

For $h = \pi  - \pi/8$, the two endpoints of $\g(\cdot, h)$ are far apart; For $h =   \pi/4 $, the two endpoints of  $\g(\cdot, h)$ have reversed their relative position. If we define a quantity measuring the signed horizontal distance of the two endpoints, then this function must attain zero for some $h^* \in [ \pi/4, 7 \pi/8 ]$ by continuity.

The resulting arc $s \mapsto \g( s ,h^*)$ is a simple closed curve by symmetry--- $\g( -\pi ,h^*)$ and $\g( \pi ,h^*)$ are the same point and the tangent vectors also agree. In other words, $s \mapsto \g( s ,h^*)$ extends to a $2\pi$-periodic function on $\RR$ with $W^{2,p}$ regularity.

Finally \eqref{eq:curve_bending} follows from \eqref{eq:aux_curve_bending}.

\begin{figure}[ht]

\tikzset{every picture/.style={line width=0.75pt}} 

\begin{tikzpicture}[x=1pt,y=1pt,yscale=-1,xscale=1]

\draw [line width=0.4]    (50,120) -- (50,70) ;
\draw [shift={(50,68)}, rotate = 90] [color={rgb, 255:red, 0; green, 0; blue, 0 }  ][line width=0.4]    (4.37,-1.32) .. controls (2.78,-0.56) and (1.32,-0.12) .. (0,0) .. controls (1.32,0.12) and (2.78,0.56) .. (4.37,1.32)   ;
\draw [color={rgb, 255:red, 144; green, 19; blue, 254 }  ,draw opacity=1 ]   (0,111.6) .. controls (30,109.6) and (70,109.6) .. (100,111.6) ;

\draw [line width=0.4]    (250,120) -- (250,70) ;
\draw [shift={(250,68)}, rotate = 90] [color={rgb, 255:red, 0; green, 0; blue, 0 }  ][line width=0.75]    (4.37,-1.32) .. controls (2.78,-0.56) and (1.32,-0.12) .. (0,0) .. controls (1.32,0.12) and (2.78,0.56) .. (4.37,1.32)   ;
\draw [shift={(47,0)}, color={rgb, 255:red, 144; green, 19; blue, 254 }  ,draw opacity=1 ]   (171.03,105.51) .. controls (166.71,105.64) and (162.03,105.01) .. (160.53,107.17) .. controls (159.03,109.34) and (161.29,111.9) .. (167.58,111.54) .. controls (173.86,111.17) and (230.86,110.67) .. (237.17,111.43) .. controls (243.48,112.19) and (247.03,109.01) .. (245.86,106.84) .. controls (244.7,104.67) and (240.53,105.67) .. (235.35,105.37) ;

\draw [line width=0.4]    (50,200) -- (50,150) ;
\draw [shift={(50,148)}, rotate = 90] [color={rgb, 255:red, 0; green, 0; blue, 0 }  ][line width=0.5]    (4.37,-1.32) .. controls (2.78,-0.56) and (1.32,-0.12) .. (0,0) .. controls (1.32,0.12) and (2.78,0.56) .. (4.37,1.32)   ;
\draw [shift={(-40,10)}, color={rgb, 255:red, 144; green, 19; blue, 254 }  ,draw opacity=1 ]   (74.09,173.34) .. controls (65.03,173.17) and (61.09,173.34) .. (59.26,175.17) .. controls (57.42,177.01) and (57.59,180.17) .. (62.09,179.67) .. controls (66.59,179.17) and (112.42,179.51) .. (116.76,180.01) .. controls (121.09,180.51) and (122.42,177.51) .. (120.2,175.34) .. controls (117.97,173.17) and (112.62,173.63) .. (104.79,173.46) ;

\draw [line width=0.4]    (250,200) -- (250,150) ;
\draw [shift={(250,148)}, rotate = 90] [color={rgb, 255:red, 0; green, 0; blue, 0 }  ][line width=0.5]    (4.37,-1.32) .. controls (2.78,-0.56) and (1.32,-0.12) .. (0,0) .. controls (1.32,0.12) and (2.78,0.56) .. (4.37,1.32)   ;
\draw  [shift={(50,15)}, color={rgb, 255:red, 144; green, 19; blue, 254 }  ,draw opacity=1 ] (177.41,173.17) .. controls (174.91,175.01) and (178.08,177.67) .. (182.58,177.17) .. controls (187.08,176.67) and (212.58,177.01) .. (216.41,177.67) .. controls (220.24,178.34) and (225.3,176.51) .. (223.08,174.34) .. controls (220.85,172.17) and (211.74,172.34) .. (200.08,172.17) .. controls (188.41,172.01) and (179.91,171.34) .. (177.41,173.17) -- cycle ;

\end{tikzpicture}
\caption{Bending the arc $\g$ with curvature $\ep \k_\sharp$. The top right and the bottom two are the arcs $s \mapsto \g(s, h )$ with the two bends moving from near the endpoints to the center.}
\label{fig:bending}
\end{figure}
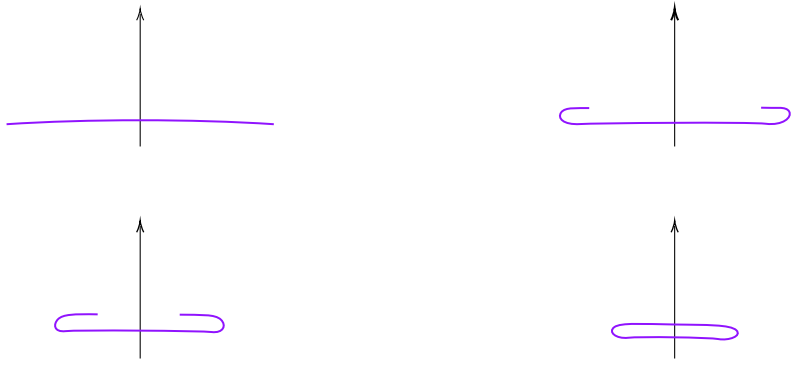

\end{proof}

\subsection{Illposedness in Sobolev spaces}

\begin{proof}[Proof of Theorem~\ref{thm:main_sobolev}]

Let us first choose the initial data. Given $p\neq 2$ and $0< \alpha < \frac{1}{2}  -\frac{1}{2 p} $, we fix
\begin{equation}\label{eq:aux_final_proof_delta}
\delta = \min \{ 1-\frac{1}{p} -2\alpha, 2\alpha \}.
\end{equation}
Since $\delta>0$, we let $\k_\sharp \in L^p(\TT)$ be given by Corollary \ref{cor:model_p_neq_2} with the parameters $\alpha,p$ and $\delta$ above.
Note that Corollary \ref{cor:model_Cbeta} states that this initial data fails to produce any $L^p$ weak solution to any system \eqref{eq:model_linear} satisfying the $(L^p,\delta)$ assumptions.

Then take the $W^{2,p}$ initial data $\Omega_0$ as the interior of a $W^{2,p}$ simple closed curve $\g_0$ that is arc-length parameterized with length $2\pi$ given by Lemma \ref{lemma:curve_bending} whose curvature is given by
\begin{equation}\label{eq:aux_final_proof_k}
\k_0  = \ep \k_{\sharp} + \k_{g}
\end{equation}
for some $\ep>0$ where the bad part $\k_{\sharp} \in L^p(\TT) $ is from Corollary \ref{cor:model_p_neq_2} and  the good part $\k_{g} \in C^\infty(\TT)$.

With the initial data chosen $\Omega_0$, we prove by contradiction. Suppose   there exists an $\alpha$-patch $\Om_t$ on some $[0,T]$ and $M >0 $ such that
\begin{equation}\label{eq:aux_final_proof_0}
\sup_{t \in [ 0, T ]  } \| \Om_t \|_{W^{2,p }} \leq M .
\end{equation}

Then we apply Theorem \ref{thm:flow_existence} with the $\alpha$-patch $\Om_t$ on $[0,T]$ and with the initial parameterization $\g_0 $ above to obtain a unique flow $\g \in C([0,T];C^{1,\sigma}(\TT))$ with the initial data $\g_0$ and $\sigma = 1- \frac{1}{p}  -2\alpha >0$.

Next, by \eqref{eq:aux_final_proof_0}, Theorem \ref{prop:weak_k_flow} shows that the curvature of $\p \Om_t$ under the $C^{1,\sigma}$ parameterization $\g$ satisfies the regularity
\begin{equation}\label{eq:aux_final_proof_1_0}
\k \in L^\infty( [0,T]; L^p(\TT)),
\end{equation}
and the equation
\begin{equation}\label{eq:aux_final_proof_1}
\begin{aligned}
&\int_{\TT} \k(\cdot, T) \phi(\cdot , T)\, g(\cdot,T ) \,dx -\int_{\TT} \k(\cdot, 0) \phi(\cdot ,0) g(\cdot,0)\, dx  \\
&  = \int_{\TT} g \k \p_t \phi \,dx dt  - c_\alpha \int_{\TT} g^{1-2\alpha} \k \mathcal{L}_\alpha \phi \, dx dt +\int_{0}^T \mathcal{E}_t (\phi) \, dt
\end{aligned}
\end{equation}
with  $\mathcal{E}_t $ satisfying
\begin{equation}\label{eq:aux_final_proof_2}
|\mathcal{E}_t(\phi)| \lesssim  |\phi|_{  B^{2\alpha -\sigma}_{  p' , 1}} \quad \text{for all $\phi \in C^\infty(\TT  ) $}
\end{equation}
uniformly on $t\in[0,T]$.

We now put \eqref{eq:aux_final_proof_1} into the setting considered in Section \ref{sec:disper_eq_illposed}. Let $f = \k g$ and define the mapping $ f \mapsto \overline{f} = g^{-2\alpha} f$. Let the distributional forcing $F_t = \mathcal{E}_t$ in the above. Then \eqref{eq:aux_final_proof_1} is  the weak formulation for the equation
\begin{equation}\label{eq:aux_final_proof_3}
\begin{cases}
\p_t f - c_\alpha \mathcal{L}_\alpha \overline{f} = F_t &\\
f|_{t =0} = f_0.
\end{cases}
\end{equation}
By rescaling the time, we can assume $c_\alpha =1$ as in Definition \ref{def:wea_sol_model_eq}.

To obtain the final contradiction, we need to verify the $(L^p,\delta)$ assumption for Corollary \ref{cor:model_p_neq_2}. Let us consider the first condition \ref{assu:f_continuity}. By H\"older's inequality, we have the estimate
\begin{align}\label{eq:aux_final_proof_4}
| f (\cdot, t) - \overline{f}(\cdot ,t )|_{L^p(\TT)} \lesssim |1 -g^{-2\alpha}  |_{L^\infty(\TT)} | f (\cdot, t) |_{L^p(\TT)} .
\end{align}
As explained in Section \ref{subsec:curvature_flow_1},  the flow equation for $\g$ implies the evolution of the metric $g:\TT\times [0,T] \to  \RR^+$
\begin{equation}\label{eq:aux_final_proof_5}
\begin{cases}
\p_t g = g \p_s v \cdot \T   &\\
g |_{t =0} = 1
\end{cases}
\end{equation}
where we have used that the initial data is arc-length parameterized. By Lemma \ref{lemma:estimate_v_on_boundary} and the regularity $g \in C([0,T]; C^{\sigma}(\TT))$ for $\sigma = 1- \frac{1}{p}  -2\alpha >0$ implies that in \eqref{eq:aux_final_proof_4} we have
\begin{align}
| f (\cdot, t) - \overline{f}(\cdot ,t )|_{L^p(\TT)} \lesssim   t  \leq t^{ \delta} .
\end{align}
So we have verified the first condition \ref{assu:f_continuity}.

The second condition \ref{assu:Forcing_bound} follows directly from \eqref{eq:aux_final_proof_2} by \eqref{eq:aux_final_proof_delta} and the fact   $ \delta  \leq 1$. So Corollary \ref{cor:model_p_neq_2} implies that
\begin{equation}\label{eq:aux_final_proof_6}
\sup_{t \in [0,\delta]} |f(t)|_{L^p(\TT)} =\sup_{t \in [0,\delta]} |g^{-2\alpha} \k (t) |_{L^p(\TT)} =  \infty
\end{equation}
which is a contradiction to \eqref{eq:aux_final_proof_1_0}.

\end{proof}

\section{Proof of H\"older illposedness}\label{sec:proof_holder}

In the last section, we consider the $C^{2,\beta}$ $\alpha$-patches. To finish the proof of illposedness, we need to strengthen a few main ingredients in this setting.

\subsection{Improved estimates in the \texorpdfstring{$C^{2,\beta}$}{C2b} case}

The $C^{2,\beta}$ H\"older case requires a higher regularity of the flow $\g$. To this end, we have the following improved analog of Lemma \ref{lemma:estimate_v_on_boundary} in the tangential direction.
\begin{lemma}\label{lemma:estimate_v_Cbeta}
Let $0< \alpha< \frac{1}{2}$ and $ 0 \leq \beta \leq 1 $. Suppose  $\Omega $ is a $C^{2,\beta}$ bounded domain and $v$ is the velocity field of $\Omega $ according to \eqref{eq:aSQG_BS}.

Then the velocity field satisfies
\begin{equation}\label{eq:estimate_v_Cbeta}
\big|   \p_s v \cdot \T   \big|_{C^{0,1}} \leq C(\alpha, \beta , \Omega)
\end{equation}
where $|\cdot |_{C^{0,1}}$ is the Lipschitz norm.
\end{lemma}
\begin{proof}
In view of Lemma \ref{lemma:estimate_v_on_boundary}, we focus on the $C^{0,1}$   continuity using the formula
\begin{align*}
\p_s v (s)  & =     P.V.  \int_{ \g    } \T (s'+s )   \frac{   (\g(s  )-\g( s'+s )) \cdot \T(s ) }{  |\g(s )-\g(s'+s )|^{2  +2\alpha }}  \,ds' .
\end{align*}

For any $\delta >0$, we denote by $\Delta_\delta f(s) =f(s+\delta) -f(s)  $   and consider the split
\begin{align*}
 \Delta_\delta [ \p_s v   \cdot \T ]  & = (\p_s v  \cdot \T) (s  + \delta) - (\p_s v \cdot \T )  (s) = I_1 + I_2
\end{align*}
where
\begin{align*}
I_1 =  \lim_{\ep\to 0+} &\int_{ \ep< |  s' |    \leq 2\delta } \T ( s+  \delta+  s'  ) \cdot \T ( s+  \delta  ) \\
& \qquad \times \frac{   (\g( s +\delta  )-\g(   s+  \delta+  s' )) \cdot \T( s +\delta ) }{  |\g( s+\delta )-\g(   s+  \delta+  s' )|^{2  +2\alpha }}    \,d   s' \\
&\quad - \lim_{\ep\to 0+} \int_{ \ep< |  s' | \leq 2\delta   }\T (s'+s )  \cdot \T ( s  ) \frac{   (\g(s  )-\g( s'+s )) \cdot \T(s ) }{  |\g(s )-\g(s'+s )|^{2  +2\alpha }}  \,ds'
\end{align*}
and
\begin{align*}
I_2 =    \int_{   |  s'| \geq 2\delta   }  \Delta_\delta \left[ \T (s'+s ) \cdot \T ( s  )  \frac{   (\g(s  )-\g( s'+s )) \cdot \T(s ) }{  |\g(s )-\g(s'+s )|^{2  +2\alpha }}    \right] \,ds'.
\end{align*}
We will show the estimates $|I_1|,|I_2|  \lesssim \delta  $.

\noindent
\textbf{\underline{Estimate of $I_1$:}}

By the estimates   \eqref{NT1}, \eqref{eq:arc_length_f}  and \eqref{eq:arc_length_linear} with $p=\infty$ in Section \ref{subsec:curves_estimates}, for any $s,s' \in \RR$
\begin{equation}\label{eq:aux_estimate_v_on_boundary_Cbeta}
\begin{aligned}
 \T(s' ) \cdot \T (s) &= 1   + O(| s  - s' |^{2 } ) \\
 (\g(s      )-\g(  s' )) \cdot \T(s ) &= (s  - s' ) + O(|s  - s'|^{ 3  } ) \\
\frac{|s  - s' |^{2  +2\alpha } }{|\g(s)-\g(  s')|^{2  +2\alpha }}   &  =1 + O(|s  - s'|^{     2   }  ) .
\end{aligned}
\end{equation}

Using \eqref{eq:aux_estimate_v_on_boundary_Cbeta}, we have
\begin{align*}
I_1 = \lim_{\ep\to 0+} & \left( \int_{ \ep< |  s' |    \leq 2\delta }    \frac{ -   s'   }{  |  s'  |^{2  +2\alpha }} \,ds'  + \int_{ \ep< |  s' |    \leq 2\delta }  O(| s'|^{1- 2\alpha  })  \,d   s' \right. \\
&\quad -\left. \int_{ \ep< |  s' | \leq 2\delta   }    \frac{   - s'   }{  |  s'  |^{2  +2\alpha }}\,ds'   + \int_{ \ep< |  s' |    \leq 2\delta }  O(| s'|^{1- 2\alpha  })  \,ds'\right).
\end{align*}

The first terms in the above two lines vanish by oddness, and we have
\begin{align*}
|I_1| & \lesssim     \int_{  | s'| \leq 2\delta   }   |    s' |^{    1 - 2\alpha     }      \,ds' \lesssim \delta^{2-2\alpha} \lesssim \delta.
\end{align*}

\noindent
\textbf{\underline{Estimate of $I_2$:}}

In this regime, we first consider the finite difference of the integrand
\begin{align}\label{eq:aux_proofholder_Cbeta}
 \Delta_\delta \left[ \T (s'+s )\cdot \T(s )   \frac{   (\g(s  )-\g( s'+s )) \cdot \T(s ) }{  |\g(s )-\g(s'+s )|^{2  +2\alpha }}    \right] .
\end{align}
To reduce notation, for each fixed $|s'| > 0$, let us denote by $K_{s'}(s)$ the function
$$
s \mapsto  K_{s'}(s):=  \T (s'+s ) \cdot \T(s )   \frac{   (\g(s  )-\g( s'+s )) \cdot \T(s ) }{  |\g(s )-\g(s'+s )|^{2  +2\alpha }} .
$$
From the regularity $\T, \N \in W^{1,p}$, we know that $ K_{s'}$ is weakly differentiable with $L^p$ derivative:
\begin{equation}
\begin{aligned}
\p_s K_{s'} (s) = & - \k(s'+s )\N (s'+s )\cdot \T(s )    \frac{   (\g(s  )-\g( s'+s )) \cdot \T(s ) }{  |\g(s )-\g(s'+s )|^{2  +2\alpha }}  \\
& -\k( s )\T (s'+s )\cdot \N(s )    \frac{   (\g(s  )-\g( s'+s )) \cdot \T(s ) }{  |\g(s )-\g(s'+s )|^{2  +2\alpha }} \\
& + \T (s'+s ) \cdot \T(s )     \frac{   (\T(s  )-\T( s'+s )) \cdot \T(s ) }{  |\g(s )-\g(s'+s )|^{2  +2\alpha }} \\
& -\k(s) \T (s'+s ) \cdot \T(s )     \frac{   (\g(s  )-\g( s'+s )) \cdot \N(s ) }{  |\g(s )-\g(s'+s )|^{2  +2\alpha }}\\
& -(2+2\alpha )\T (s'+s )  \cdot \T(s )    \frac{   (\g(s  )-\g( s'+s )) \cdot \T(s )     }{  |\g(s )-\g(s'+s )|^{4  +2\alpha }}\\
& \qquad \times (\g(s  )-\g( s'+s )) \cdot (\T(s )  -\T( s'+s  ) ).
\end{aligned}
\end{equation}
Since $ \g$ is $C^{2,\beta}$, using the   $W^{2,\infty}$ estimates in Section \ref{subsec:curves_estimates}, it is straightforward to show that $K_{s'}' $ satisfies the point-wise estimate
\begin{equation}\label{eq:aux_K'_mean_value_0_Cbeta}
|\p_s K_{s'} (s)| \lesssim       |s'|^{ -2\alpha }  .
\end{equation}
  So by the fundamental theorem of calculus for $W^{1,p} $ functions and the bound \eqref{eq:aux_K'_mean_value_0_Cbeta}, we have
\begin{equation}\label{eq:aux_K'_mean_value_Cbeta}
\begin{aligned}
|K_{s'}(s+\delta) - K_{s'}(s ) |  & \leq \int_{s}^{s+\delta}  |\p_s K_{s'} (\tau )|\, d\tau \\
& \lesssim  |s'|^{ -2 \alpha } \delta .
\end{aligned}
\end{equation}

Now we apply the estimate \eqref{eq:aux_K'_mean_value_Cbeta} for $K_{s'}(s)$ to $I_2$, obtaining
\begin{align*}
| I_2| \leq     \int_{ \frac{L}{2}  \geq |  s'| \geq 2\delta   }  \Delta_\delta \left[  K_{s'}(s) \right] \,ds' \lesssim \int_{    \frac{L}{2}  \geq |  s'| \geq 2\delta   }   |s'|^{ -2 \alpha } \delta \, ds' \lesssim \delta .
\end{align*}

\end{proof}

\subsection{The flow of  \texorpdfstring{$C^{2,\beta}$}{C2b} \texorpdfstring{$\alpha$}{a}-patches}

Based on the improved regularity in the tangential component of $\p_s v$, we have the improved wellposedness of the flow of $C^2$ $\alpha$-patches.
\begin{theorem}\label{thm:flow_existence_Cbeta}
Let $0< \alpha< \frac{1}{2}$ and $ 0 \leq \beta \leq 1 $. Suppose  $\Omega_t$ is a $C^{2,\beta}$ $\alpha$-patch with the initial data $\Omega_0$ and $v =K_\alpha  *  \chi_{\Omega_t}:\RR^2\times [0,T] \to \RR^2$ is the velocity field of $\Omega_t$.

For any $C^{2,\beta}$ parametrization of $ \p \Omega_0$, $\g_0\in C^{2,\beta}(\TT) $,  the Cauchy problem for $\gamma: \TT \times [0,T] \to \RR^2$
\begin{equation}
\begin{cases}
\p_t \gamma(x,t) = v(\gamma(x,t) , t) &\\
\gamma |_{t=0} = \gamma_0&
\end{cases}
\end{equation}
has a unique solution $\gamma \in  C([0,T]; C^{1,1 -2\alpha }(\TT)) \cap L^\infty([0,T]; C^{1,1   }(\TT))$.
\end{theorem}
\begin{proof}
By Theorem \ref{thm:flow_existence} with $p =\infty$, we have the existence of the flow $\g \in  C([0,T]; C^{1,1 -2\alpha }(\TT))$. Since $x \mapsto \g(x,t)$ is a parameterization of $C^{2,\beta}$ boundary $\p \Om$, it suffices to show that the arc-length $g  = |\p_x \g| \in L^\infty([0,T]; C^{0,1}(\TT)) $.

As discussed in Section \ref{subsec:curvature_flow_1}, the evolution of the metric $g:\TT\times [0,T] \to  \RR^+$ satisfies
\begin{equation}\label{eq:aux_arclength_Cbeta}
\begin{cases}
\p_t g = g \p_s v \cdot \T   &\\
g |_{t =0} = |\p_x \g_0|.
\end{cases}
\end{equation}
Since $s \mapsto \p_s v\cdot \T$ is $C^{0,1}$ in the arc-length variable by Lemma \ref{lemma:estimate_v_Cbeta}, we have
$$
  \p_s v \cdot \T(x,t)  \in L^\infty([0,T]; C^{0,1}(\TT))
$$ in the Lagrangian label as well. Then a standard Gronwall-type calculation for \eqref{eq:aux_arclength_Cbeta} implies that $g  \in L^\infty([0,T]; C^{0,1}(\TT)). $

\end{proof}

\subsection{The curvature flow of  \texorpdfstring{$C^{2,\beta}$}{C2b} \texorpdfstring{$\alpha$}{a}-patches}

The last improvement we need is the curvature dynamics in the weak formulation. Compared to the $W^{2,p}$ cases, we need to show the error terms are functions with some H\"older regularity.

\begin{proposition}\label{prop:dispersive_term_c2}
Let $0 < \alpha < \frac{1}{2}$ and $ 0 \leq \beta \leq 1$. Suppose $\gamma \in  C([0,T]; C^{1,1 -2\alpha }(\TT)) \cap L^\infty([0,T]; C^{1,1   }(\TT))$ is the flow of  a $C^{2,\beta}$ $\alpha$-patch $\Omega_t$ with the initial data $\g_0 \in C^{2,\beta } ( \TT)  $. Let $v$ be the velocity field generated by the patch $\Omega_t$ parameterized by $\g$ and  $g = |\p_x \g |$ be the metric associated with the flow $\g$.

Then for any $t \in [0,T]$ and any $\phi \in C^\infty(\TT)$,
\begin{equation}\label{eq:lemma_dispersive_term_0_Cbeta}
\int_{\TT} \p_s v \cdot \N \p_s\phi \,ds = -c_\alpha \int_{\TT}g^{1-2\alpha}  \k \mathcal{L}_\alpha \phi \, dx +\int_{\TT}   R \phi \, dx
\end{equation}
where $c_\alpha>0$ is a universal constant and $R: \TT \times [0,T] \to \RR$  satisfies
\begin{equation}\label{eq:prop_dispersive_error_term_Cbeta}
|R|_{L^\infty C^{0,\rho }} \leq C_{  \Omega,T}
\end{equation}
where $\rho = \min\{ 1-2\alpha,\beta\}>0 $ and the constant $C_{ \Omega,T} <\infty$ depends on $\alpha$, $\beta$, $\Om$, and $ T $.
\end{proposition}
\begin{proof}

As in the proof of Proposition \ref{prop:dispersive_term}, let us consider the (time-dependent) arc-length variable $s : = \ell(\xi) = \int_0^{\xi} g(\tau )\, d\tau $ and define the arc-length counterparts by the formula $\overline{f}(\ell(\xi)) = f(\xi) $, i.e. $\overline{f}(s) : = f (\ell^{-1}(s))$ for $f \in \{ \g,\T,\N,\k,v \}$.

In the new variable, by Lemma \ref{lemma:estimate_v_on_boundary} the left-hand side of \eqref{eq:lemma_dispersive_term_0_Cbeta} reads
 $$
\int  \p_s \overline{v} \cdot \overline{\N} \p_s\overline{\phi} \,ds = \int  \p_s\overline{\phi} \int \overline{\T}(s') \overline{\N}(s) \frac{( \oga(s) -  \oga(s'))\cdot \oT(s)  }{ |\oga(s) -  \oga(s')|^{2+2\alpha } }\, ds'  \, ds  .
$$

By the periodicity and the Fubini theorem, we can reparameterize to obtain
\begin{equation}\label{eq:aux_dispersive_term_0_Cbeta}
\begin{aligned}
& \int  \p_s \overline{v} \cdot \overline{\N} \p_s\overline{\phi} \,ds \\
& = \int  \int  \p_s\overline{\phi}(s  )\overline{\T}(s' + s) \overline{\N}(s  ) \frac{( \oga(s ) -  \oga(s' + s ))\cdot \oT(s   )  }{ |\oga(s   ) -  \oga(s' + s )|^{2+2\alpha } }\,  ds  \, ds' .
\end{aligned}
\end{equation}
As in Proposition \ref{prop:dispersive_term}, we consider the approximations,
\begin{equation}\label{eq:aux_dispersive_term_1_Cbeta}
\lim_{\ep \to 0^+}  \int_{|  s' | \geq \ep }  \int  \p_s\overline{\phi}(s  )  \overline{\T}(s' + s) \overline{\N}(s  ) \frac{( \oga(s ) -  \oga(s' + s ))\cdot \oT(s   )  }{ |\oga(s   ) -  \oga(s' + s )|^{2+2\alpha } }\,  ds  \, ds'.
\end{equation}

To integrate by parts in the $s $ variable, notice that when $s'  \neq 0$, the function $s  \mapsto  \overline{\T}(s' + s) \overline{\N}(s  ) \frac{( \oga(s ) -  \oga(s' + s ))\cdot \oT(s   )  }{ |\oga(s   ) -  \oga(s' + s )|^{2+2\alpha } }$ is at least $W^{1,p}$ whose derivative  is the sum of
\begin{align}
& \underbrace{- \ok(s'+ s) \oN(s'+ s ) \oN( s ) \frac{( \oga(s) -  \oga(s'  + s ))\cdot \oT(s)  }{ |\oga(s) -  \oga(s' + s )|^{2+2\alpha } } }_{:= K_1(s,s')} \label{eq:aux_def_K_1_Cbeta}
\end{align}
\begin{align}
&  \underbrace{+ \ok(s  )  \overline{\T}( s'+ s  ) \oT( s ) \frac{ ( \oga(s ) -  \oga(s' + s ))\cdot \oT(s   )  }{ |\oga(s) -  \oga(s' + s  )|^{2+2\alpha }}}_{:= K_2(s,s')}  \label{eq:aux_def_K_2_Cbeta}
\end{align}
\begin{align}
& \underbrace{ \overline{\T}(s' + s) \overline{\N}(s  ) \frac{( \oT(s ) -  \oT(s' + s ))\cdot \oT(s   )  }{ |\oga(s   ) -  \oga(s' + s )|^{2+2\alpha } }   }_{:= K_3(s,s')} \label{eq:aux_def_K_3_Cbeta}
\end{align}
\begin{align}
&  \underbrace{ -  \ok(s  )  \overline{\T}( s'+ s  ) \oN( s ) \frac{ ( \oga(s ) -  \oga(s' + s ))\cdot \oN(s   )  }{ |\oga(s) -  \oga(s' + s  )|^{2+2\alpha }}}_{:= K_4(s,s')}  \label{eq:aux_def_K_4_Cbeta}
\end{align}
\begin{align}
&  - (2+2\alpha)\overline{\T}( s'+ s  ) \oN( s )  \frac{( \oga(s) -  \oga(s'+s ))\cdot \oT(s)  }{ |\oga(s) -  \oga(s'+s )|^{4+2\alpha } }  \nonumber\\
& \underbrace{\qquad\qquad \times  ( \oga(s) -  \oga(s'+s ))\cdot (\oT(s) - \oT(s'+s ) ) .\,\, }_{:= K_5(s,s')} \label{eq:aux_def_K_5_Cbeta}
\end{align}
We then integrate by parts in \eqref{eq:aux_dispersive_term_1_Cbeta} using the derivatives \eqref{eq:aux_def_K_1_Cbeta}--\eqref{eq:aux_def_K_5_Cbeta} and then reparameterize back to obtain
\begin{equation}
\begin{aligned}\label{eq:aux_dispersive_term_2_Cbeta}
 & \int  \p_s \overline{v} \cdot \overline{\N} \p_s\overline{\phi} \,ds    =  \lim_{\ep \to 0^+} \sum_{i} \int_{| s'| \geq \ep  } \int    \overline{\phi}(s)  K_i(s,s') \, ds  \, ds' \\
& = \lim_{\ep \to 0^+} \sum_{1 \leq n \leq 5} I_n(\ep),
\end{aligned}
\end{equation}
where each integral $I_j(\ep)$ corresponds to the kernels $K_j$ from \eqref{eq:aux_def_K_1_Cbeta}--\eqref{eq:aux_def_K_5_Cbeta}.
\begin{align}
I_1 (\ep) & =  -\int \int_{|   s -s' | \geq \ep  } \overline{\phi}(s) \ok(s') \oN(s  ') \cdot \oN( s ) \frac{( \oga(s) -  \oga(s'    ))\cdot \oT(s)  }{ |\oga(s) -  \oga(s'   )|^{2+2\alpha } }   \, ds  \, ds '
\end{align}
\begin{align}
I_2(\ep) & =  \int \int_{|   s -s' | \geq \ep  } \overline{\phi}(s)  \ok(s ) \oT(s  ') \cdot \oT( s ) \frac{( \oga(s) -  \oga(s'    ))\cdot \oT(s)  }{ |\oga(s) -  \oga(s'   )|^{2+2\alpha } }   \, ds  \, ds '
\end{align}
\begin{align}
I_3(\ep) &  =  \int \int_{|   s -s' | \geq \ep  } \overline{\phi}(s)  \oT(s  ') \cdot \oN( s ) \frac{( \oT(s) -  \oT(s'    ))\cdot \oT(s)  }{ |\oga(s) -  \oga(s'   )|^{2+2\alpha } }   \, ds  \, ds '
\end{align}
\begin{align}
I_4(\ep) &  =  \int \int_{|   s -s' | \geq \ep  } - \overline{\phi}(s) \ok(s    ) \oT(s  ') \oN( s ) \frac{( \oga(s) -  \oga(s'    ))\cdot \oN(s)  }{ |\oga(s) -  \oga(s'   )|^{2+2\alpha } }   \, ds  \, ds '
\end{align}
\begin{equation}
\begin{aligned}
I_5(\ep)  & =  - (2+2\alpha )\int \int_{|   s -s' | \geq \ep  }  \overline{\phi}(s)  \overline{\T}( s'  )\oN( s )  \frac{( \oga(s) -  \oga(s'  ))\cdot \oT(s)  }{ |\oga(s) -  \oga(s'  )|^{4+2\alpha } }   \\
&\qquad  \qquad \qquad  \qquad  \qquad  \times ( \oga(s) -  \oga(s'  ))\cdot (\oT(s) - \oT(s' ) ) \, ds  \, ds ' .
\end{aligned}
\end{equation}

In the steps below, we will show each term on the right-hand side of \eqref{eq:aux_dispersive_term_2_Cbeta} is compatible with the thesis \eqref{eq:lemma_dispersive_term_0_Cbeta} and \eqref{eq:prop_dispersive_error_term_Cbeta}.
Specifically, we will show $I_1$ leads to the main linear term with a remainder of $C^{1-2\alpha}$ regularity while the rest of $I_i$ contribute to errors with either $C^{\beta}$ or $C^{1-2\alpha }$ regularity.

\noindent
\textbf{\underline{Analysis of $I_1$:}}

Let us first analyze $I_{1}$. We switch back to the Lagrangian labels via $s \mapsto x =\ell^{-1}(s)$ and $s' \mapsto y = \ell^{-1}(s')$ in the integral, obtaining
\begin{align*}
I_{1}(\ep) :   & =  -\int \int_{|   \ell (x)  -\ell (y) | \geq \ep  }  {\phi}(x )   \k(y ) \N( y) \cdot \N( x )\\
& \qquad\qquad \qquad \times  \frac{( \g(x ) -  \g(y    ))\cdot \T(x)  }{ |\g(x) -  \g( y   )|^{2+2\alpha } } g(x) g(y)  \, dx  \, d y  .
\end{align*}
Let us consider the variant $ \widetilde{I}_{1}(\ep)$
\begin{equation}\label{eq:aux_curvature_kernal_I_1_3_Cbeta}
\begin{aligned}
\widetilde{I}_{1}(\ep) :   & =  -\int \int_{ g(y)|    x   - y  | \geq \ep  }  {\phi}(x )   \k(y ) \N( y) \cdot \N( x )\\
& \qquad\qquad \qquad \times  \frac{( \g(x ) -  \g(y    ))\cdot \T(x)  }{ |\g(x) -  \g( y   )|^{2+2\alpha } } g(x) g(y)  \, dx  \, d y  .
\end{aligned}
\end{equation}
We claim $\lim_{\ep \to 0^+}  \widetilde{I}_{1}(\ep) = \lim_{\ep \to 0^+} {I}_{1}(\ep)  $ and both limits exist by an argument similar to Lemma \ref{lemma:estimate_v_on_boundary} due to $\phi \in C^\infty(\TT)$ .

Since $x \mapsto \ell  (x)$ is a $C^{1, 1}$ diffeomorphism, the fundamental theorem of calculus implies that for each $y\in \TT$, the set $S_\ep(y)$ of the symmetric difference
\begin{equation}\label{eq:aux_curvature_kernal_I_1_4}
 S_\ep(y): =  \{x \in \TT : | \ell(x) -\ell(y) | \geq \ep \} \Delta \{x \in \TT : g(y) |   x  - y  | \geq \ep \}
\end{equation}
has Lebesgue measure $|S_\ep| \lesssim \ep^{2 } $. Hence by the absolute value, the difference between $  \widetilde{I}_{1}( \ep) $ and ${I}_{1}$ satisfies
\begin{align}
& \left|  \widetilde{I}_{1}( \ep) -  {I}_{1}( \ep)  \right| \nonumber \\
& \lesssim \int \int_{S_\ep(y) } | {\phi}(x )    | | \k( y )|  |x-y|^{-1-2\alpha }  \, dx  \, d y \label{eq:aux_curvature_kernal_I_1_5a}
\end{align}
Since $\phi $ is smooth, the term \eqref{eq:aux_curvature_kernal_I_1_5a} vanishes at $\ep \to 0$ due to the small measure of the set $S_\ep(y)$ for all $\alpha< \frac{1}{2}$:
\begin{align*}
&\int \int_{S_\ep(y) } | {\phi}(x )    | | \k( y )|  |x-y|^{-1-2\alpha }  \, dx  \, d y \\
& \lesssim |\k|_{L^\infty} |  \phi |_{L^\infty} \ep^{-1 -2\alpha } \sup_{ y\in \TT} | S_\ep(y) | \\
& \lesssim |\k|_{L^\infty} | \phi |_{L^\infty} \ep^{1  -2\alpha} \to 0
\end{align*}
where we have used that $|x-y| \sim \ep $ on $S_\ep(y)$ in the first inequality. So we can conclude that    $\left|  \widetilde{I}_{1}( \ep) -  {I}_{1}( \ep)  \right| \to 0 $ as $\ep \rightarrow 0.$

So now we proceed to estimate
\begin{equation}\label{eq:aux_curvature_kernal_I_1_alter_Cbeta}
\begin{aligned}
\lim_{\ep \to 0^+} \widetilde{I}_1(\ep)
& =  -\int  P.V.\int   {\phi}(x )   \k(y ) \N( y) \cdot \N( x )\\
& \qquad\qquad \qquad \times  \frac{( \g(x ) -  \g(y    ))\cdot \T(x)  }{ |\g(x) -  \g( y   )|^{2+2\alpha } } g(x) g(y)  \, dx  \, d y  .
\end{aligned}
\end{equation}

Since the flow $\g \in  L^\infty_t C^{1, 1}$, by Lemma \ref{lemma:curve_estimates} we define a kernel $Q_{1}  : \TT \times \TT \to \RR$
\begin{equation}\label{eq:aux_curvature_kernal_2_Cbeta}
\begin{aligned}
Q_{1 }(x,y ) = &{\T}( y   ) \cdot {\T}(x ) \frac{( \g (x ) -  \g ( y   ))\cdot \T(x  )  }{ |\g ( x ) -  \g ( y   )|^{2+2\alpha } } g(x ) g( y  ) \\
&  -  |g( y )|^{1- 2\alpha}  \sum_{n \in \ZZ}\frac{ x - y-2\pi n }{ | x -y -2\pi n |^{2+2\alpha}}      .
\end{aligned}
\end{equation}
so that the periodic kernel $Q_{1}(x,y ) : \TT \times \TT \to \RR$  is jointly continuous away from $x \neq y$ with the estimate
\begin{equation}\label{eq:aux_curvature_kernal_Q1a}
\begin{aligned}
|Q_{1}(x,y )| &  \lesssim   |x  - y |^{ -2\alpha  }
\end{aligned}
\end{equation}
where we note that the exponent $ -2\alpha   > -1$. In what follows we will also use the fact that for any $0<h \leq \frac{|x-y|}{2}$
\begin{equation}\label{eq:aux_curvature_kernal_Q1b}
	\begin{aligned}
		|Q_{1}(x+h,y ) - Q_{1}(x,y ) | &  \lesssim  h |x  - y |^{ -1-2\alpha  }
	\end{aligned}
\end{equation}
which can be proved by considering the derivative of \eqref{eq:aux_curvature_kernal_2_Cbeta} using the $C^{2,\beta}$  regularity  of $\Om_t $ and $C^{1,1}$ regularity of $\g$.

It follows from   \eqref{eq:aux_curvature_kernal_2_Cbeta} that
\begin{equation}
\begin{aligned}\label{eq:aux_curvature_kernal_3_Cbeta}
 \widetilde{I}_{1}( \ep)  & = -\int \k( y )   |g( y )|^{1-  2\alpha}   \int_{g(y) | x  - y| \geq \ep  }     {\phi}(x)       \sum_{ \substack{m \in \ZZ   }}        \frac{ x  - y+ m }{ | x  -  y+m |^{2+2\alpha}}  \,  d x   \,  d y    \\
& \qquad -  \int \int     {\phi}( x )  \k( y  )   Q_1(x,y)   \,   d x  \, d y    .
\end{aligned}
\end{equation}

The first term in \eqref{eq:aux_curvature_kernal_3_Cbeta} is the main term, corresponding to the first term on the right-hand side of \eqref{eq:lemma_dispersive_term_0_Cbeta}. Indeed,  by Lemma \ref{lemma:local_formula_La}, there holds
\begin{align}\label{eq:aux_curvature_kernal_4_Cbeta}
& \lim_{\ep \to 0^+ }   \int_{ \substack{x \in \TT  \\ g(y) | x  - y| \geq \ep   }  }     {\phi}(x)       \sum_{ \substack{m \in \ZZ   }}        \frac{ x  - y+ m }{ | x  -  y+m |^{2+2\alpha}}  \,  d x  \nonumber  \\
&  = P.V. \int_{\RR }     {\phi}(x)       \frac{ x  - y }{ | x  - y|^{2+2\alpha}}  \,  d x  \nonumber \\
& = c_\alpha  \mathcal{L}_\alpha \phi( y ) .
\end{align}
Therefore,  by \eqref{eq:aux_curvature_kernal_3_Cbeta} and \eqref{eq:aux_curvature_kernal_4_Cbeta} we have
\begin{equation}\label{eq:aux_curvature_kernal_5_Cbeta}
\begin{aligned}
\lim_{\ep \to 0^+}  \widetilde{I}_{1}( \ep)  & = -c_\alpha \int \k( y )   |g( y )|^{1-  2\alpha}   \mathcal{L}_\alpha  {\phi}(y)        \,  d y    \\
& \qquad +   \int {\phi}( x )    \int      \k( y  )   Q_1(x,y)   \,   d y   \,  d x   .
\end{aligned}
\end{equation}

Since $ -2\alpha > -1$ in \eqref{eq:aux_curvature_kernal_Q1a}, the second integral in \eqref{eq:aux_curvature_kernal_5_Cbeta} is well-defined. We just need to show
\begin{equation}\label{eq:aux_curvature_kernal_R1_final_Cbeta}
R_1(x):=      \int_{\TT}       \k( y  )    Q_1(x,y)     \, d y \in C^{ \rho }(\TT) .
\end{equation}

We   repeat the standard splitting argument. Consider $h>0$ and the finite difference
\begin{equation}\label{eq:aux_curvature_kernal_R1_1}
\begin{aligned}
\Delta_h \left[ R_1(x) \right]
& =  \int_{|x -y| \leq  2 h } \k( y  )   Q_1(x+h,y)      \, d y  -\int_{|x- y | \leq  2 h } \k( y  )    Q_1(x,y)    \, d y   \\
&    + \int_{|x-y| \geq  2 h } \k( y  )   Q_1(x+h,y)      \, d y  -     \int_{|x-y| \geq  2 h } \k( y  )    Q_1(x,y)      \, d y
\end{aligned}
\end{equation}
In the nearby region $|x-y| \leq 2  h$, using \eqref{eq:aux_curvature_kernal_Q1a} and the fact $\k \in C^{\beta}(\TT)$   we can bound the terms individually by their absolute values and obtain
\begin{equation}
\begin{aligned}\label{eq:aux_curvature_kernal_R1_near}
&\Big|	\int_{|x -y| \leq  2 h } \k( y  )   Q_1(x+h,y)      \, d y  -\int_{|x - y| \leq  2 h } \k( y  )    Q_1(x,y)    \, d y   \Big| \\
& \lesssim     \int_{| x - y| \leq   2 h }     | x   -y   |^{   - 2\alpha   } \\
 & \lesssim h^{1 -2\alpha     } .
\end{aligned}
\end{equation}

Next, for the far-field region $|x-y| \geq 2 h $, by \eqref{eq:aux_curvature_kernal_Q1b} we have
\begin{align*}
& \Big| \int_{|x -y| \geq  2 h } \k( y  )   [ Q_1(x+h,y)      \, d y  -        Q_1(x,y)    ]  \, d y  \Big| \\
& \lesssim   h\int_{|x  -y| \geq  2 h } |\k( y  )  |  |x-y|^{-1 -2\alpha} \, d y  \Big| .
\end{align*}
Since $\k \in C^{\beta}(\TT)$, integrating the above we have
\begin{align}\label{eq:aux_curvature_kernal_R1_far}
  \Big| \int_{|x -y| \geq  2 h } \k( y  )   [ Q_1(x+h,y)      \, d y  -        Q_1(x,y)    ]  \, d y  \Big|
  \lesssim  h^{1-2\alpha }.
\end{align}
Collecting the estimates \eqref{eq:aux_curvature_kernal_R1_near} and \eqref{eq:aux_curvature_kernal_R1_far},  we see that
$$
\left| \int      \k( y  )  \Delta_h \left[  Q_1(x,y)  \right]    \, d y  \right| \lesssim h^{1-2\alpha } ,
$$
so  the $C^{1-2\alpha}$ H\"older regularity \eqref{eq:aux_curvature_kernal_R1_final_Cbeta} holds and $I_1$ is compatible with the conclusion \eqref{eq:lemma_dispersive_term_0_Cbeta}.

\noindent
\textbf{\underline{Analysis of $I_2$:}}

Recall that we need to show
\begin{align}\label{eq:aux_curvature_kernal_8_Cbeta}
\lim_{\ep \to 0^+}I_2(\ep) &  = \lim_{\ep \to 0^+} \int \int_{|   s  - s' | \geq \ep  }  \overline{\phi}(s) \ok(s    ) \oT(s ' ) \oT( s ) \frac{( \oga(s) -  \oga(s'    ))\cdot \oT(s)  }{ |\oga(s) -  \oga(s'   )|^{2+2\alpha } }   \, ds  \, ds'
\end{align}
defines a $C^{ \rho} (\TT)$ function in the original Lagrangian label.

Observe that by Fubini theorem and Lemma \ref{lemma:estimate_v_on_boundary}
\begin{equation}\label{eq:aux_curvature_kernal_9_Cbeta}
\begin{aligned}
& \lim_{\ep \to 0^+} I_2(\ep) =\lim_{\ep \to 0^+}    \int  \overline{\phi}(s) \ok(s    )  \int_{| s- s'| \geq \ep  }  \oT(s ' ) \oT( s ) \frac{( \oga(s) -  \oga(s'    ))\cdot \oT(s)  }{ |\oga(s) -  \oga(s'   )|^{2+2\alpha } }   \,  ds' \, ds \\
&  = \int  \overline{\phi}(s)  \ok(s    )  \p_s \overline{v} \cdot \oT (s ) \, ds  .
\end{aligned}
\end{equation}
Writing in terms of the Lagrangian label,
\begin{equation}\label{eq:aux_curvature_kernal_I_2}
 \lim_{\ep \to 0^+} I_2(\ep)= \int_\TT   {\phi}(x)  \k (x)   \p_s {v} \cdot \T  (x)  g(x)\, dx.
\end{equation}

Since $ \k \in C^{\beta}$ and $g \in C^{0,1}$, to show \eqref{eq:aux_curvature_kernal_I_2} is compatible with \eqref{eq:prop_dispersive_error_term_Cbeta}, it suffices to show $  \p_s {v} \cdot \T  $ is $C^{ \rho} (\TT)$. Since $\ell$ is a $C^{1,1}$ diffeomorphism
 and $\rho = \min \{ 1-2\alpha ,\beta \}$, this follows from  Lemma \ref{lemma:estimate_v_on_boundary} which shows $C^\sigma$ H\"older continuity of $\p_s \overline{v} \cdot \oT,$ and $\sigma = 1-2\alpha$ for $p=\infty.$

\noindent
\textbf{\underline{Analysis of $I_3$:}}

We need to show the distribution
\begin{align}\label{eq:aux_curvature_kernal_11_Cbeta}
\lim_{\ep \to 0^+} I_3(\ep) &  =  \lim_{\ep \to 0^+} \int \int_{|  s  -    s' | \geq \ep  } \overline{\phi}(s)  \oT(s  ') \cdot \oN( s ) \frac{( \oT(s) -  \oT(s'    ))\cdot \oT(s)  }{ |\oga(s) -  \oga(s'   )|^{2+2\alpha } }   \, ds '  \, ds
\end{align}
defines a $C^{\rho}(\TT)$ function in the original Lagrangian label.

By the reasoning similar to the treatment of $I_2$, we only need to show
\begin{equation}\label{eq:aux_curvature_kernal_I3}
s\mapsto  \int   \oT(s  ') \cdot \oN( s ) \frac{( \oT(s) -  \oT(s'    ))\cdot \oT(s)  }{ |\oga(s) -  \oga(s'   )|^{2+2\alpha } }   \, ds ' \quad \text{is $C^{\rho}$ continuous.}
\end{equation}
Let $h>0$ and $\Delta_h$ be the finite difference for the $s$ variable. As before, consider the split
\begin{equation}\label{eq:aux_curvature_kernal_I3_2}
\begin{aligned}
& \Delta_h  \left[   \int   \oT(s  ') \cdot \oN( s ) \frac{( \oT(s) -  \oT(s'    ))\cdot \oT(s)  }{ |\oga(s) -  \oga(s'   )|^{2+2\alpha } }   \, ds '  \right]   \\
&  =  \int_{|s-s'| \leq 2h }   \Delta_h  \left[ \oT(s  ') \cdot \oN( s ) \frac{( \oT(s) -  \oT(s'    ))\cdot \oT(s)  }{ |\oga(s) -  \oga(s'   )|^{2+2\alpha } }  \right] \, ds '  \\
& \qquad + \int_{|s-s'| \geq 2h }   \Delta_h  \left[ \oT(s  ') \cdot \oN( s ) \frac{( \oT(s) -  \oT(s'    ))\cdot \oT(s)  }{ |\oga(s) -  \oga(s'   )|^{2+2\alpha } }  \right] \, ds '  .
\end{aligned}
\end{equation}
The integral in region $|s- s'| \leq 2 h$ can be bounded by its absolute value:
\begin{align}\label{eq:aux_curvature_kernal_I3_3}
& \left|  \int_{|s-s'| \leq 2h }   \Delta_h  \left[ \oT(s  ') \cdot \oN( s ) \frac{( \oT(s) -  \oT(s'    ))\cdot \oT(s)  }{ |\oga(s) -  \oga(s'   )|^{2+2\alpha } }  \right] \, ds ' \right| \nonumber  \\
& \lesssim   \int_{|s-s'| \leq 2h }   |s+ h-s'|^{1 -2\alpha}   \, ds '  + \int_{|s-s'| \leq 2h }   |s -s'|^{1 -2\alpha}   \, ds '  \nonumber \\
&    \lesssim h^{2 -2\alpha } .
\end{align}

Next, we consider the integral in the region $|s -s'| \geq 2h$. Notice that the integrand $\oT(s  ') \cdot \oN( s ) \frac{( \oT(s) -  \oT(s'    ))\cdot \oT(s)  }{ |\oga(s) -  \oga(s'   )|^{2+2\alpha } } $ has the derivative:
\begin{equation}\label{eq:aux_curvature_kernal_I3_4}
\begin{aligned}
& \p_s \left(  \oT(s  ') \cdot \oN( s ) \frac{( \oT(s) -  \oT(s'    ))\cdot \oT(s)  }{ |\oga(s) -  \oga(s'   )|^{2+2\alpha } }  \right)  \\
& =  \k(s) \oT(s  ') \cdot \oT( s ) \frac{( \oT(s) -  \oT(s'    ))\cdot \oT(s)  }{ |\oga(s) -  \oga(s'   )|^{2+2\alpha } } \\
& \qquad +   \oT(s  ') \cdot \oN( s ) \frac{( \oT(s) -  \oT(s'    ))\cdot \oN(s)  }{ |\oga(s) -  \oga(s'   )|^{2+2\alpha } } \\
& \qquad + (2+2\alpha) \oT(s  ') \cdot \oN( s ) \frac{( \oT(s) -  \oT(s'    ))\cdot \oT(s)   (\oga(s) -  \oga(s'   ))\cdot \oT(s)}{ |\oga(s) -  \oga(s'   )|^{4+2\alpha } } .
\end{aligned}
\end{equation}
Due to the $C^{2,\beta}$ regularity of $\p \Om$, it is easy to see this derivative is bounded up to a constant by $ |s-s'|^{-2\alpha}$. Then by the fundamental theorem of calculus,
\begin{equation}\label{eq:aux_curvature_kernal_I3_5}
\begin{aligned}
& \left| \int_{|s-s'| \geq 2h }   \Delta_h  \left[ \oT(s  ') \cdot \oN( s ) \frac{( \oT(s) -  \oT(s'    ))\cdot \oT(s)  }{ |\oga(s) -  \oga(s'   )|^{2+2\alpha } }  \right] \, ds '  \right|  \\
& \lesssim \int_{|s-s'| \geq 2 h}  \int_{s}^{s+h }| \tau - s' |^{-2\alpha } \, d \tau \, ds ' \\
& \lesssim  h \int_{|s-s'| \geq 2 h}   | s - s' |^{-2\alpha }   \, ds '\\
& \lesssim h .
\end{aligned}
\end{equation}

It follows from \eqref{eq:aux_curvature_kernal_I3_2} and \eqref{eq:aux_curvature_kernal_I3_5} that
\begin{equation}
\begin{aligned}
& \Delta_h  \left[   \int   \oT(s  ') \cdot \oN( s ) \frac{( \oT(s) -  \oT(s'    ))\cdot \oT(s)  }{ |\oga(s) -  \oga(s'   )|^{2+2\alpha } }   \, ds '  \right] \\
&\lesssim h^{2-2\alpha } + h \leq h
\end{aligned}
\end{equation}
which is consistent with \eqref{eq:aux_curvature_kernal_I3}.

\noindent
\textbf{\underline{Analysis of $I_4$:}}

Recall that we need to estimate   the distribution
\begin{equation}
\begin{aligned}\label{eq:aux_curvature_kernal_14_Cbeta}
& \lim_{\ep \to 0+}I_4(\ep) \\
&  =  - \lim_{\ep \to 0+} \int \overline{\phi}(s) \ok(s    ) \int_{|  s  - s' | \geq \ep  }   \oT(s  ') \oN( s ) \frac{( \oga(s) -  \oga(s'    ))\cdot \oN(s)  }{ |\oga(s) -  \oga(s'   )|^{2+2\alpha } }    \, ds '\, ds  .
\end{aligned}
\end{equation}

By the same reasoning as in the case of $I_3$, we only need to show
\begin{equation}\label{eq:aux_curvature_kernal_I4}
s\mapsto  \int  \oT(s  ') \oN( s ) \frac{( \oga(s) -  \oga(s'    ))\cdot \oN(s)  }{ |\oga(s) -  \oga(s'   )|^{2+2\alpha } }    \, ds ' \quad \text{is $C^{\rho}$ continuous.}
\end{equation}
Let us consider again its finite difference
\begin{equation}\label{eq:aux_curvature_kernal_I4_1}
\begin{aligned}
& \Delta_h \left[ \int  \oT(s  ') \oN( s ) \frac{( \oga(s) -  \oga(s'    ))\cdot \oN(s)  }{ |\oga(s) -  \oga(s'   )|^{2+2\alpha } }    \, ds '  \right]\\
& = \int_{|s-s'| \leq 2h }  \Delta_h \left[\oT(s  ') \oN( s ) \frac{( \oga(s) -  \oga(s'    ))\cdot \oN(s)  }{ |\oga(s) -  \oga(s'   )|^{2+2\alpha } }  \right]  \, ds '  \\
& \qquad + \int_{|s-s'| \geq 2h }  \Delta_h \left[\oT(s  ') \oN( s ) \frac{( \oga(s) -  \oga(s'    ))\cdot \oN(s)  }{ |\oga(s) -  \oga(s'   )|^{2+2\alpha } }  \right]  \, ds '  .
\end{aligned}
\end{equation}

For the integral in the region $|s-s'| \leq 2h $, we again bound by the absolute value:
\begin{equation}\label{eq:aux_curvature_kernal_I4_2}
\begin{aligned}
& \left|   \int_{|s-s'| \leq 2h }  \Delta_h \left[\oT(s  ') \oN( s ) \frac{( \oga(s) -  \oga(s'    ))\cdot \oN(s)  }{ |\oga(s) -  \oga(s'   )|^{2+2\alpha } }  \right]  \, ds '  \right| \\
& \lesssim  \int_{|s-s'| \leq 2h }   |s+h -s'|^{1-2\alpha }  \, ds'  + \int_{|s-s'| \leq 2h }   |s  -s'|^{1-2\alpha }  \, ds' \\
& \lesssim h^{2-2\alpha} .
\end{aligned}
\end{equation}

For the integral in the region $|s-s'| \geq 2h $, similarly to $I_3,$ we use the fundamental theorem of calculus to obtain:
\begin{equation}\label{eq:aux_curvature_kernal_I4_3}
\begin{aligned}
& \left|   \int_{|s-s'| \geq 2h }  \Delta_h \left[\oT(s  ') \oN( s ) \frac{( \oga(s) -  \oga(s'    ))\cdot \oN(s)  }{ |\oga(s) -  \oga(s'   )|^{2+2\alpha } }  \right]  \, ds '  \right| \\
& \lesssim  \int_{|s-s'| \geq 2h } \int_{ s}^{s+h}  | \tau -s'|^{ -2\alpha }  \, d \tau \, ds'   \\
& \lesssim h  .
\end{aligned}
\end{equation}

\noindent
\textbf{\underline{Analysis of $I_5$:}}

Recall that we need to estimate
\begin{equation}
\begin{aligned}\label{eq:aux_curvature_kernal_15_Cbeta}
& \lim_{\ep \to 0+} I_5(\ep)  \\
& =  - (2+2\alpha ) \lim_{\ep \to 0+} \int  \int_{| s -    s' | \geq \ep  } \overline{\phi}(s)  \overline{\T}( s'  )\oN( s )  \frac{( \oga(s) -  \oga(s'  ))\cdot \oT(s)  }{ |\oga(s) -  \oga(s'  )|^{4+2\alpha } }  \\
&\qquad \qquad \qquad  \times ( \oga(s) -  \oga(s'  ) )\cdot (\oT(s) - \oT(s' ) )  \, ds ' \, ds .
\end{aligned}
\end{equation}

Again, it suffices to show the $C^{\rho}$ continuity of
\begin{equation}\label{eq:aux_curvature_kernal_I5}
\begin{aligned}
s\mapsto
 \int  & \overline{\T}( s'  )\oN( s )  \frac{( \oga(s) -  \oga(s'  ))\cdot \oT(s)  }{ |\oga(s) -  \oga(s'  )|^{4+2\alpha } }  \\
&\qquad   \times ( \oga(s) -  \oga(s'  ) )\cdot (\oT(s) - \oT(s' ) )  \, ds ' .
\end{aligned}
\end{equation}
The corresponding finite difference in the region $|s-s'| \leq 2h$ can be bounded as follows,
\begin{equation}\label{eq:aux_curvature_kernal_I5_1}
\begin{aligned}
&  \Bigg|   \int_{ |s-s'| \leq 2h  }  \Delta_h \bigg[ \overline{\T}( s'  )\oN( s )  \frac{( \oga(s) -  \oga(s'  ))\cdot \oT(s)  }{ |\oga(s) -  \oga(s'  )|^{4+2\alpha } }  \\
&\qquad  \qquad \times ( \oga(s) -  \oga(s'  ) )\cdot (\oT(s) - \oT(s' ) ) \bigg]  \, ds '   \Bigg| \\
& \lesssim  \int_{ |s-s'| \leq 2h  } \left( |s+ h -s'|^{ 1 -2\alpha } + |s -s'|^{ 1 -2\alpha } \right) \, ds'  \lesssim h^{2-2\alpha } .
\end{aligned}
\end{equation}
For the region $|s-s'| \geq 2h$ we use the fundamental theorem of calculus:
\begin{equation}\label{eq:aux_curvature_kernal_I5_2}
\begin{aligned}
&  \Bigg|   \int_{ |s-s'| \geq 2h  }  \Delta_h \bigg[ \overline{\T}( s'  )\oN( s )  \frac{( \oga(s) -  \oga(s'  ))\cdot \oT(s)  }{ |\oga(s) -  \oga(s'  )|^{4+2\alpha } }   \\
&\qquad \qquad  \times ( \oga(s) -  \oga(s'  ) )\cdot (\oT(s) - \oT(s' ) ) \bigg]  \, ds '   \Bigg| \\
& \lesssim  \int_{ |s-s'| \leq 2h  }   \int_s^{s+h} |\tau -s'|^{ -2\alpha} \, d \tau \, ds'  \lesssim h .
\end{aligned}
\end{equation}

\end{proof}

\subsection{Proof of Theorem \ref{thm:main_Holder}}

With all the strengthened ingredients, we finish the proof in the  $C^{2,\beta}$ case.

\begin{proof}[Proof of Theorem~\ref{thm:main_Holder}]

Let us first choose the initial data. Given  $0 \leq \beta < 1$ and $0< \alpha < \frac{1}{2}  $, we fix
\begin{equation}\label{eq:aux_final_proof_delta_Cbeta}
 \delta = \min\{ 2\alpha , \frac{1-\beta}{2}  \}  .
\end{equation}

Since $\delta>0$, we let $\k_\sharp \in C^\beta(\TT)$ be given by Corollary \ref{cor:model_Cbeta} with the parameters $\alpha,\beta$ and $\delta$ above. As in the $W^{2,p}$ case, by Corollary \ref{cor:model_Cbeta}, this initial data fails to produce any $C^\beta$ weak solution to any system \eqref{eq:model_linear} satisfying the $(C^\beta,\delta)$ assumptions.

As before, we  take the $C^{2,\beta}$ initial data $\Omega_0$ as the interior of a $C^{2,\beta}$ simple closed curve $\g_0 \in C^{2,\beta}(\TT)$ that is arc-length parameterized with length $2\pi$ given by Lemma \ref{lemma:curve_bending} whose curvature is given by
\begin{equation}\label{eq:aux_final_proof_k_Cbeta}
\k_0  = \ep \k_{\sharp} + \k_{g}
\end{equation}
for some $\ep>0$ where the bad part $\k_{\sharp} \in C^\beta(\TT) $ is from Corollary \ref{cor:model_Cbeta} and  the good part $\k_{g} \in C^\infty(\TT)$.

With the initial data chosen $\Omega_0$, we prove by contradiction. Suppose  there exists an $\alpha$-patch $\Om_t$ on some $[0,T]$ and $M >0 $ such that
\begin{equation}\label{eq:aux_final_proof_0_Cbeta}
\sup_{t \in [ 0, T ]  } \| \Om_t \|_{C^{2, \beta }} \leq M .
\end{equation}

Then we apply Theorem \ref{thm:flow_existence_Cbeta} with the $\alpha$-patch $\Om_t$ on $[0,T]$ and with the initial parameterization $\g_0 $ above to obtain a unique flow $\g \in C([0,T];C^{1,1-2 \alpha }(\TT)) \cap L^\infty([0,T];C^{1,1}(\TT))$ with the initial data $\g_0$.

Next,  arguing similarly to Lemma \ref{lemma:geo_quant_rough_gamma} using  \eqref{eq:aux_final_proof_0_Cbeta} show that the curvature of $\p \Om_t$ under the $ C([0,T];C^{1,\sigma}(\TT)) \cap L^\infty([0,T];C^{1,1}(\TT))$ parameterization $\g$ satisfies the regularity
\begin{equation}\label{eq:aux_final_proof_2b_Cbeta}
 \k \in L^\infty([0,T];C^{0,\beta}(\TT)) .
\end{equation}
 Furthermore, a direct analog of the Theorem \ref{prop:weak_k_flow} using estimates of Proposition \ref{prop:dispersive_term_c2} instead of those of Proposition \ref{prop:dispersive_term}, shows the curvature   $\k :\TT\times [0,T] \to \RR$ solves the equation
\begin{equation}\label{eq:aux_final_proof_1_Cbeta}
\begin{aligned}
&\int_{\TT} \k(\cdot, T) \phi(\cdot , T)\, g(\cdot,T ) \,dx -\int_{\TT} \k(\cdot, 0) \phi(\cdot ,0) g(\cdot,0)\, dx  \\
&  = \int_{\TT \times [0,T]} \left(g \k \p_t \phi \,dx dt  - c_\alpha   g^{1-2\alpha} \k \mathcal{L}_\alpha \phi   +  R  \phi  \right)\, dx \, dt
\end{aligned}
\end{equation}
for any test functions $ \phi \in C^\infty(\TT \times \RR)$ with  $R:\TT \times [0,T] \to \RR $ satisfying
\begin{equation}\label{eq:aux_final_proof_2_Cbeta}
| R |_{L^\infty C^{0,  \rho  }} <\infty
\end{equation}
where $ \rho  = \min \{\beta ,1 -2\alpha   \}$ .

We will derive a contradiction using \eqref{eq:aux_final_proof_2b_Cbeta}--\eqref{eq:aux_final_proof_2_Cbeta}.

As before, if we let $f = \k g$, define the mapping $ f \mapsto \overline{f} = g^{-2\alpha} f$  and the distributional forcing $F_t (\phi) :=  \int_{\TT} R \phi  \, dx $ in the above, then \eqref{eq:aux_final_proof_1_Cbeta} is exactly the weak formulation for the equation
\begin{equation}\label{eq:aux_final_proof_3_Cbeta}
\begin{cases}
\p_t f - c_\alpha \mathcal{L}_\alpha \overline{ f  } = F_t &\\
f|_{t =0} = f_0.
\end{cases}
\end{equation}
By rescaling the time, we can assume $c_\alpha =1$ as in Definition \ref{def:wea_sol_model_eq}.

To obtain the final contradiction, we need to verify the $(C^\beta,\delta)$ assumptions for \eqref{eq:aux_final_proof_3_Cbeta} in Corollary \ref{cor:model_Cbeta}. The second condition \eqref{eq:Forcing_bound_Cbeta} in assumption \ref{assu:Forcing_bound_Cbeta} follows directly from \eqref{eq:aux_final_proof_2_Cbeta}:
since $  \rho  =\min \{\beta ,1 -2\alpha   \} \geq \beta -2\alpha +\delta$ by the assumptions and \eqref{eq:aux_final_proof_delta_Cbeta}, we have the embedding $C^{ \sigma } (\TT)\subset B^{\beta -2\alpha +\delta}_{\infty, \infty }(\TT) $.

Let us consider the first condition \eqref{eq:fbar_bound_Cbeta} in the assumption \ref{assu:f_continuity_Cbeta}. By the definition of H\"older norms, we have the estimate
\begin{equation}
\begin{aligned}\label{eq:aux_final_proof_4_Cbeta}
& | f (\cdot, t) - \overline{f}(\cdot ,t )|_{C^{0, \beta}(\TT)} \\
&\lesssim |1 -g^{-2\alpha}  |_{L^\infty(\TT)} | f (\cdot, t) |_{C^{0, \beta}(\TT)} + |1 -g^{-2\alpha}  |_{C^{0, \beta}(\TT)} | f (\cdot, t) |_{L^\infty(\TT)}.
\end{aligned}
\end{equation}

As explained in Section \ref{subsec:curvature_flow_1},  the flow equation for $\g$ implies the evolution of the metric $g:\TT\times [0,T] \to  \RR^+$
\begin{equation}\label{eq:aux_final_proof_5_Cbeta}
\begin{cases}
\p_t g = g \p_s v \cdot \T   &\\
g |_{t =0} = 1,
\end{cases}
\end{equation}
where we have used that the initial data $\g_0$ is arc-length parameterized. The regularity $g \in C([0,T]; C^{1}(\TT))$  implies that in \eqref{eq:aux_final_proof_4_Cbeta} we have
\begin{align}
| f (\cdot, t) - \overline{f}(\cdot ,t )|_{C^{0, \beta}(\TT)} \lesssim t  \leq t^{ \delta} .
\end{align}
So we have also verified the condition \eqref{eq:fbar_bound_Cbeta} in assumption \ref{assu:f_continuity_Cbeta}.

By Corollary \ref{cor:model_Cbeta} we must have that
\begin{equation}\label{eq:aux_final_proof_6_Cbeta}
\sup_{t \in [0,\delta]} |f(t)|_{C^{0, \beta}(\TT)} =\sup_{t \in [0,\delta]} |g^{-2\alpha} \k (t) |_{C^{0, \beta}(\TT)} =  \infty
\end{equation}
which is a contradiction to \eqref{eq:aux_final_proof_2b_Cbeta}.

\end{proof}

\bibliographystyle{alpha}
\bibliography{asqg_patch}

\end{document}